\documentclass[12pt,twoside,a4paper]{article}
\usepackage{amsmath}
\usepackage{amssymb}
\usepackage{amsthm, bbm}
\usepackage{algpseudocode}
\usepackage{mathtools}
\usepackage{algorithmicx}
\usepackage{enumitem}
\usepackage{nicefrac}
\usepackage[numbers, sort]{natbib}
\usepackage[utf8]{inputenc}
\usepackage{hyperref}
\usepackage{cleveref}
\usepackage{lmodern}
\usepackage{marvosym}
\usepackage[scale=0.775]{geometry}
\title{Optimality Conditions and Moreau--Yosida Regularization for Almost Sure State Constraints}
\author{Caroline Geiersbach\thanks{Weierstrass Institute, 10117 Berlin, Germany 
  (\texttt{caroline.geiersbach@wias-berlin.de})}
\and Michael Hinterm\"{u}ller\thanks{Weierstrass Institute, 10117 Berlin, Germany 
  (\texttt{michael.hintermueller@wias-berlin.de}) and Humboldt-Universit\"at zu Berlin, Unter den Linden 6,10099 Berlin, Germany (\texttt{hint@math.hu-berlin.de})}}

\newcommand{\R}{\mathbb{R}}
\newcommand{\cA}{\mathcal{A}}

\newcommand{\cR}{\mathcal{R}}

\renewcommand{\u}{x_1}
\newcommand{\y}{x_2}
\newcommand{\z}{\lambda}
\newcommand{\bx}{\bar{x}}
\newcommand{\p}{\lambda_e}
\newcommand{\lami}{\lambda_i}
\newcommand{\bz}{\bar{\lambda}}

\newcommand{\E}{\mathbb{E}}
\newcommand{\pP}{\mathbb{P}}
\newcommand{\D}{\text{ d}}

\DeclareMathOperator*{\esssup}{ess\,sup}
\DeclareMathOperator*{\argmin}{arg\,min}
\DeclareMathOperator*{\argmax}{arg\,max}

\theoremstyle{definition}
\newtheorem{theorem}{Theorem}[section]
\newtheorem{lemma}[theorem]{Lemma}
\newtheorem{proposition}[theorem]{Proposition}
\newtheorem{corollary}[theorem]{Corollary}
\newtheorem{defn}[theorem]{Definition}
\newtheorem{example}[theorem]{Example}
\newtheorem{rem}[theorem]{Remark}

\newcounter{assumption}
\newtheorem{assumption}[theorem]{Assumption}
\newcounter{subassumption}[assumption]
\renewcommand{\thesubassumption}{(\textit{\roman{subassumption}})}
\makeatletter
\renewcommand{\p@subassumption}{\theassumption}
\makeatother
\newcommand{\subasu}{
  \refstepcounter{subassumption}%
  \thesubassumption~\ignorespaces}
 \crefname{subassumption}{Assumption}{Assumptions}

\begin{document}
\maketitle
\begin{abstract}
We analyze a potentially risk-averse convex stochastic optimization problem, where the control is deterministic and the state is a Banach-valued essentially bounded random variable. We obtain strong forms of necessary and sufficient optimality conditions for problems subject to equality and conical constraints. We propose a Moreau--Yosida regularization for the conical constraint and show consistency of the optimality conditions for the regularized problem as the regularization parameter is taken to infinity.  
\end{abstract}

\section*{Introduction}
Let $X_1$ and $X_2$ be real, reflexive, and separable Banach spaces, and let
$(\Omega, \mathcal{F}, \pP)$ denote a complete probability space, where
$\Omega$ represents the sample space, $\mathcal{F} \subset 2^{\Omega}$ is
the $\sigma$-algebra of events on the power set of $\Omega$, and
$\pP\colon \Omega \rightarrow [0,1]$ is a probability measure.
Consider a convex stochastic optimization problem of the form
\begin{equation}
 \label{eq:model-problem-abstract}
 \begin{aligned}
  & \underset{x_1,x_2(\cdot)}{\mathrm{minimize}}\quad J_1(x_1)+ \cR[J_2(x_1, x_2(\cdot); \cdot)]\\
   &\text{subject to (s.t.)}\quad\left\{\begin{aligned}x_1 &\in C, \\
   x_2(\omega) &\in X_{2,\textup{ad}}(x_1,\omega) \quad \text{a.s.},
   \end{aligned}\right.
 \end{aligned}
\end{equation}
where $C \subset X_1$ is nonempty, closed, and convex;
$X_{2,\textup{ad}}(x_1,\omega) \subset X_2$ is assumed to be nonempty, closed, and convex for all $x_1 \in C$ and almost all $\omega \in \Omega$.  In the problem \eqref{eq:model-problem-abstract}, the control $x_1$ is deterministic and the state $x_2$ is a vector-valued random variable. This is the natural paradigm if one views $x_1$ as a concrete ``decision'' that should be made while taking into account uncertainties from inputs or parameters in the state system. This paradigm differs in spirit from a two-stage stochastic optimization problem, where $x_2$ (called the second stage variable) is a recourse decision that should be made after a realization $\omega$ becomes known; see, e.g., \cite{Pflug2014, Shapiro2009} for an introduction to multistage stochastic optimization. However, the convex problem \eqref{eq:model-problem-abstract} can be shown to be equivalent to a two-stage stochastic optimization problem if $x_2(\omega)$ is almost surely contained in a bounded set; see the discussion in \cite{Rockafellar1975} and \cite[Section 3]{Rockafellar1976}. This equivalence was used by Rockafellar and Wets~\cite{Rockafellar1976,Rockafellar1976a,Rockafellar1975, Rockafellar1976b} to establish optimality conditions for two-stage stochastic optimization problems by first embedding them into a static problem of the form \eqref{eq:model-problem-abstract}.

There are numerous applications ranging from economics and finance \cite{Pereira1991, Mulvey2004}, engineering \cite{Conti2008, Rockafellar2015}, and life sciences \cite{Ermoliev2019} that can be mathematically expressed as \eqref{eq:model-problem-abstract}. Often, historical information (in the form of a probability distribution) about uncertain demand, performance, or physical parameters has been collected and one would like to take this information into account while making a decision. In some settings, it is enough to optimize over these possibilities on average, in which case one works with a risk-neutral formulation, meaning the measure $\cR$ is simply the expectation, or average. However, there are many applications where the average is not desirable, and one works instead with a different choice of risk functional $\cR:L^p(\Omega) \rightarrow \R \cup \{ \infty\}$, which captures the optimizer's desired risk tolerance.

For many applications, the spaces $X_1$ and $X_2$ are Euclidean spaces, which simplifies the analysis considerably. For instance, in portfolio optimization, the decision variable may be a (finite) vector of real-valued weights representing the percentages of total cash dedicated to each asset. The second-stage variable is another vector corresponding to a redistribution of assets after new information is gathered.  However, there are some problems in which the underlying spaces are infinite-dimensional, for instance, in optimization with partial differential equations (PDEs) subject to uncertainty; see e.g., \cite{Kouri2013, Kouri2016, Alexanderian2017, Kolvenbach2018, Schillings2011, Chen2013, Geiersbach2020b, Benner2016}, and the references therein. As an example, consider the optimal control of a stationary heat source over a domain $D \subset \R^d$, where the conductivity $a=a(s,\omega)$ is uncertain, but known to follow a given probability distribution. One seeks a control $x_1$ and the resulting solution $x_2 = x_2(s,\omega)$ of the PDE, the latter of which does not exceed a threshold $\psi$ and, on average, best approximates some target distribution $y_D$. With $\alpha \geq 0$, the model is formally expressed as
\begin{equation}
\label{eq:model-PDE-UQ-state-constraints}
\tag{$\textup{P}'$}
 \begin{aligned}
&\underset{x_1,x_2(\cdot)}{\mathrm{minimize}}\quad \frac{1}{2} \E \left[ \lVert x_2 - y_D \rVert_{L^2(D)}^2 \right]+ \frac{\alpha}{2} \lVert \u \rVert_{L^2(D)}^2 \\
&\quad  \text{s.t.} \left\{\begin{aligned}
x_1 &\in C, \\
  -\nabla \cdot (a(s,\omega) \nabla x_2(s,\omega)) &=  \u(s) & \text{ on }& D \times \Omega, \\
  x_2(s,\omega) &= 0 &\text{ on }& \partial D \times \Omega,\\
  x_2(s,\omega) &\leq \psi(s,\omega) & \text{ on }& D \times \Omega,
\end{aligned}\right.
 \end{aligned}
\end{equation}
where the constraints are to hold almost everywhere in $D$ and almost surely in $\Omega$. Other examples of conical constraints will be presented in \Cref{subsection-MY-revisited}. To develop optimality conditions for this problem, a constraint qualification is needed, which implicitly requires that subsets of the given function space have a nonempty interior. Standard PDE theory (cf.~\cite{Adams2003}) gives, under mild assumptions on $a$ and $D$, a unique solution $x_2(\cdot,\omega) \in X_2:=H_0^1(D) \cap H^2(D)$ for every $x_1 \in X_1:=L^2(D)$ and $\omega \in \Omega$. Through embeddings one obtains even $x_2(\cdot,\omega) \in C(\bar{D}),$ where an interior point condition is satisfiable. However, without additional structure on $\Omega$, $x_2$ as a function in $L^p(\Omega,X_2)$ can only satisfy an interior point condition if $p=\infty$. It is notable that for a
deterministic model, i.e., where $x_2 = x_2(s)$, the Lagrange multiplier associated with the conical constraint is generally only a measure. It is standard to employ regularization in numerical procedures to circumvent this difficulty; see \cite{Hintermueller2006, Hintermueller2006a}.

In this paper, we are concerned with obtaining strong forms of optimality conditions for a problem like \eqref{eq:model-problem-abstract}, by which we mean explicit, ``almost sure''-type conditions, which are amenable to building algorithms. The main difficulty here is in applying classical duality theory, where one would expect a Lagrange multiplier to belong to the space $(L^\infty(\Omega,X))^*$, $X$ being a Banach space. A multiplier from this space cannot generally be represented as a mapping from $\Omega$ to, say, $X^*$. Hence the task of establishing optimality conditions in the strong form, i.e., with almost sure conditions, is quite challenging. Another difficulty is in the structure of the problem itself, where the optimality conditions should make explicit the first stage variable's independence of $\omega$. Luckily, insights can be found in classical literature, above all in the important works by Rockafellar and Wets~\cite{Rockafellar1976,Rockafellar1976a,Rockafellar1975, Rockafellar1976b}, who developed theory for the case $X = \R^d$. A notable recent development in the theory of optimality conditions for the case $X = \R^d$ is \cite{Dentcheva2021}: the authors studied nonlinear multistage stochastic optimization problems with general objectives (not necessarily expectation functionals) and introduced a new concept of subregular recourse. Recently, the classical theory in \cite{Rockafellar1976,Rockafellar1976a,Rockafellar1975, Rockafellar1976b} was revisited in \cite{Geiersbach2020+} with a focus on the more general vector-valued case. The authors proved the existence of integrable Lagrange multipliers under a relatively complete recourse condition (defined in \Cref{assumption:relatively-complete-recourse}). For many applications, including the problem~\eqref{eq:model-PDE-UQ-state-constraints}, this assumption is too strong. In this paper, we consider problems not satisfying the condition of relatively complete recourse; in contrast to \cite{Geiersbach2020+}, the characterization of Lagrange multipliers is more delicate. Additionally, we generalize the theory from \cite{Geiersbach2020+} to include a potentially risk-averse formulation. Finally, in \cite{Geiersbach2020+}, a regularization was proposed for a specific example and consistency of the primal problem was analyzed as the penalization parameter is taken to infinity. However, a consistency analysis for a more general problem class was missing. This theoretical gap is addressed in the current paper, where we show consistency for both primal and dual variables as the penalization parameter is taken to infinity.

Optimality conditions with an application to PDE-constrained optimization under uncertainty were first presented in \cite{Kouri2018}. While the framework in \cite{Kouri2018} allows for generally nonconvex problems, the case where the feasible set $X_{2,\textup{ad}}(\omega)$ contains conical constraints is not covered.  Our work is related to a recent paper \cite{Gahururu2021}. As part of this work, a risk-neutral model problem from optimal control is presented and the authors also propose a Moreau--Yosida regularization to solve their problem. To this end, they impose more structure on their probability space, which allows for standard arguments via embeddings. However, they do not obtain strong forms of optimality conditions, which can yield deep insight into the structure of the optimality conditions for the original problem. The framework presented in our work is general enough to cover the theory in \cite{Gahururu2021} without additional structure on the probability space. It is worth mentioning that our work is related to the recent trend to include chance constraints in PDE-constrained optimization under uncertainty \cite{Chen2020, Farshbaf-Shaker2018, Geletu2020}. These models are of interest when deviations from a hard constraint are permissible with a certain probability. Notably, the chance constraint and regularized ``almost sure'' models can be shown to be related to one another \cite{Gahururu2021}.

We proceed as follows. In \Cref{sec:preliminaries}, we present notation and technical results needed for the handling of essentially bounded Banach-valued random variables. Additionally, we recall key properties of risk functionals. In \Cref{sec:extended-KKT}, we generalize the risk-neutral framework given in \cite{Geiersbach2020+} by allowing for potentially risk-averse objective functions. Additionally, we present necessary and sufficient conditions without the assumption of relatively complete recourse. While regular (integrable) Lagrange multipliers do not exist in this case, we show that Lagrange multipliers can be split into regular parts and singular terms, which can be handled separately. In \Cref{sec:Moreau-Yosida}, we present the Moreau--Yosida regularization  of almost sure conical constraints and show that the optimality system for the regularized problem is consistent with the original problem as the penalization parameter is taken to infinity. An outlook is given in \Cref{sec:Conclusion}.

\section{Preliminaries}
\label{sec:preliminaries}
\subsection{Background and Notation}
Throughout the paper, we denote the dual space of a (real) Banach space $X$ by $X^*$ and the canonical dual pairing by $\langle \cdot, \cdot \rangle_{X^*,X}$. If $X$ is paired with another space $U$, we write $\langle \cdot,\cdot \rangle_{U,X}$ to denote their pairing. The symbols $\rightarrow$, $\rightharpoonup$, and $\rightharpoonup^*$ denote strong, weak, and weak* convergence, respectively; $A\colon X\rightrightarrows Y$ denotes a set-valued operator from $X$ to $Y$. We recall that for a proper function $h\colon X \rightarrow \R \cup \{\infty \}$, the subdifferential (in the sense of convex analysis) is the set-valued operator
\begin{equation*}
  \partial h: X \rightrightarrows X^*: x \mapsto \{ q \in X^*: \langle q, y - x \rangle_{X^*,X} + h(x) \leq h(y) \quad  \forall y \in X \}.
\end{equation*} 
The domain and epigraph of $h$ are denoted by
$\textup{dom}(h)= \{ x \in X: h(x) < \infty \}$ and $\textup{epi}(h) = \{ (x,\alpha)\in X \times \R: h(x)\leq \alpha\}$, respectively.
The sum of two sets $A$ and $B$ with $\lambda \in \R$ is given by {$A+\lambda B:=\{ a+\lambda b: a \in A, b \in B\}.$}  
A nonempty subset $K$ of $X$ is said to be a cone if for any $x \in K$, $t\geq 0$, it follows that $tx \in K$. The indicator function of a set $C$ is denoted by $\delta_C$ and is defined by $\delta_C(x) = 0$ if $x \in C$ and $\delta_C(x) = \infty$ otherwise.
For a nonempty and convex set $C \subset X$, the normal cone $N_C(x)$ at $x \in C$ is defined by
$N_C(x):= \{ z \in X^*: \langle z, y-x \rangle_{X^*,X} \leq 0 \, \forall y \in C\}.$
We set $N_C(x) := \emptyset$ if $x \notin C$. Note that $\partial \delta_C(x) = N_C(x)$ for all $x \in C$.  

For all results, the probability space $(\Omega, \mathcal{F}, \pP)$ is assumed to be complete. Given a Banach space $X$ equipped with the norm $\lVert \cdot \rVert_X$, the Bochner space $L^r(\Omega, X):=L^r(\Omega,\mathcal{F}, \pP; X)$ is the set of all (equivalence classes of) strongly measurable functions $y:\Omega \rightarrow X$ having finite norm, where the norm is given by
\begin{equation*}
\lVert y \rVert_{L^r(\Omega,X)}:= \begin{cases}
                                     (\int_\Omega \lVert y(\omega) \rVert_X^r \D \pP(\omega))^{1/r}, \quad &r < \infty\\
                                     \esssup_{\omega \in \Omega} \lVert y(\omega) \rVert_X, \quad &r=\infty
                                    \end{cases}.
\end{equation*}
If $X=\R$ we use the shorthand $L^r(\Omega)$. An $X$-valued random variable $x$ is Bochner integrable if there exists a
sequence $\{ x_n\}$ of $\pP$-simple functions $x_n\colon \Omega \rightarrow X$
such that $\lim_{n \rightarrow \infty} \int_{\Omega} \lVert x_n(\omega)-x(\omega) \rVert_X \D \pP(\omega) = 0$.
The limit of the integrals of $x_n$ gives the Bochner integral
(the expectation), i.e., 
\begin{equation*}
  \E[x]:=\int_\Omega x(\omega) \D \pP(\omega) = \lim_{n \rightarrow \infty} \int_{\Omega} x_n(\omega) \D \pP(\omega).
\end{equation*}
Clearly, this expectation is an element of $X$.

Throughout the paper, we work with parametrized functions of the type $f\colon X \times \Omega \rightarrow Y$, and write $f(x;\omega)$ for the mapping from $X$ to $Y$ for a fixed parameter $\omega \in \Omega$. 
A property is said to hold almost surely (a.s.) provided that the set (in $\Omega$) where the property does not hold is a null set. As an example, two random variables $\xi, \xi'$ are said to be equal almost surely, $\xi = \xi'$ a.s., if and only if $\pP(\{ \omega \in \Omega: \xi(\omega) \neq \xi'(\omega)\}) = 0$. Sometimes we write ``for almost every $\omega$'' to mean a.s.

\subsection{Properties of $L^\infty(\Omega,X)$}
In this section, we present some technical results that will be crucial to characterizing optimality for problems like \eqref{eq:model-problem-abstract}.
Throughout this section, let $X$ be an arbitrary (real) reflexive and separable Banach space, and let the $\sigma$-algebra of Borel sets on $X$ be denoted by $\mathcal{B}$.
A function $f\colon X\times \Omega \rightarrow \R \cup \{ \infty\}$ is called a convex integrand if $f(\cdot;\omega)$ is convex for almost every $\omega$. 
Following Rockafellar \cite{Rockafellar1971}, an integrand is called \textit{normal} if it is not identically infinity\footnote{This definition corresponds to the original definition used by Rockafellar in the Bochner space setting \cite{Rockafellar1971}. The condition ``not identically infinity'' was included in the definition of a ``normal integrand'' in early works, including \cite{Rockafellar1971, Rockafellar1975, Rockafellar1976, Rockafellar1976a, Rockafellar1976b}).  Over time, Rockafellar changed his definition to exclude this condition; he acknowledged this in \cite{Rockafellar1976c} by writing ``what previously was a normal convex integrand is now a \underline{proper} normal convex integrand.''}, it is $(\mathcal{B}\times\mathcal{F})$-measurable, and
$f(\cdot;\omega)$ is lower semicontinuous in $X$ for each $\omega \in \Omega$.
An example of a function that is normal is one that is finite everywhere and Carath\'eodory, meaning $f$ is measurable in $\omega$ for fixed $x$ and continuous in $x$ for fixed $\omega$. Normality of $f$ implies that it is superpositionally measurable, meaning $\omega \mapsto f(x(\omega);\omega)$ is measurable as long as $x:\Omega \rightarrow X$ is measurable.

For a normal convex integrand $f$, we define the integral functional on $L^\infty(\Omega,X)$ by
\begin{equation*}
I_f(x) = \int_\Omega f(x(\omega);\omega) \D \pP(\omega),
\end{equation*}
i.e., the expectation of the superposition of $f$. A continuous linear functional $v \in (L^\infty(\Omega,X))^*$ of the form
\begin{equation}
\label{eq:absolutely-continuous-functionals}
  \langle v, x\rangle_{(L^\infty(\Omega,X))^*,L^\infty(\Omega,X)} = \int_{\Omega} \langle x^*(\omega), x(\omega) \rangle_{X^*,X} \D \pP(\omega)
\end{equation}
for some $x^* \in L^1(\Omega, X^*)$ is said to be absolutely continuous. A proof of the following result is in \cite{Geiersbach2020+}. This was originally shown for 
$X=\R^d$ in \cite{Rockafellar1971a} and a similar relation was shown in the context of law invariant convex risk functions in \cite{Filipovic2012}.

\begin{lemma}
\label{lemma:duality-conjugate-functions}
Assume $f$ is a normal convex integrand and $f(x(\omega);\omega)$ is an
integrable
function of $\omega$ for every $x \in L^\infty(\Omega,X)$.
Then $I_f$ and $I_{f^*}$ are
well-defined convex functionals on $L^\infty(\Omega,X)$ and $L^1(\Omega, X^*)$,
respectively, that are conjugate to each other in the sense that
\begin{align*}
I_{f^*}(x^*) &= \sup_{x\in L^\infty(\Omega,X)} \left\lbrace \int_{\Omega} \langle x^*(\omega),x(\omega)\rangle_{X^*,X} \D \pP(\omega) - I_f(x)\right\rbrace,\\
I_{f}(x) &= \sup_{x^* \in L^1(\Omega,X^*)} \left\lbrace \int_{\Omega} \langle x^*(\omega),x(\omega)\rangle_{X^*,X}  \D \pP(\omega) - I_{f^*}(x^*)\right\rbrace.
\end{align*}
Furthermore, if $v^*$ is an absolutely continuous functional corresponding
to a function $x^* \in L^1(\Omega, X^*)$, then $I_f^*(v^*) = I_{f^*}(x^*)$,
while $I_f^*(v^*) = \infty$ for any $v^*$ that is not absolutely continuous.
\end{lemma}

\begin{rem}
\label{remark:weak-star-lsc}
Since $X$ is reflexive, it satisfies the Radon--Nikodym property; moreover, the underlying probability space is $\sigma$-finite. It follows that the dual space $(L^1(\Omega,X^*))^*$ can be identified with $L^\infty(\Omega,X)$.
Thus the mapping 
\begin{equation*}
x \mapsto \int_{\Omega} \langle x^*(\omega),x(\omega)\rangle_{X^*,X} \D \pP(\omega) = \langle x^*,x\rangle_{L^1(\Omega,X^*),L^\infty(\Omega,X)}
\end{equation*}
 is weakly* continuous and convex. Moreover, $I_{f^*}$, as the supremum over $(x,\alpha) \in \textup{epi}(I_f)$ of functionals of the form
 \begin{equation*}
 x \mapsto \langle x^*,x\rangle_{L^1(\Omega,X^*),L^\infty(\Omega,X)} - \alpha,
 \end{equation*}
is convex and weakly* lower semicontinuous. A consequence of \Cref{lemma:duality-conjugate-functions} is that in the case where $f$ is a normal convex integrand and $f(x(\omega);\omega)$ is integrable for all $x \in L^\infty(\Omega,X)$, 
we have $I_f^{**} \equiv I^*_{f^*} \equiv I_f$ and thus $I_f$ is also weakly* lower semicontinuous. Another helpful method to prove weak* lower semicontinuity of $I_f$ follows from \cite[p.~227]{Rockafellar1971}: if  $f$ is a normal convex integrand and neither $I_f$ nor $I_{f^*}$ are identically infinity, then $I_f$ and $I_{f^*}$ are conjugate to each other.
\end{rem}

In this work, we make extensive use of a decomposition of elements belonging to $(L^\infty(\Omega,X))^*$. Note first that the space of absolutely continuous functionals defined by \eqref{eq:absolutely-continuous-functionals} form a closed subspace of $(L^\infty(\Omega,X))^*$ that is isometric to $L^1(\Omega,X^*)$. This subspace has a complement consisting of singular functionals, given in the next definition.
\begin{defn}
\label{def:singular-functionals}
A functional $v^\circ \in (L^\infty(\Omega,X))^*$ is called \textit{singular (relative to $\pP$)} if there exists a sequence $\{F_n\} \subset  \mathcal{F}$ with $F_{n+1} \subset F_n$ for all $n$ and $\pP(F_n) \rightarrow 0$ as $n\rightarrow \infty$ such that $v^\circ$ is concentrated on $\{ F_n\}$, i.e., $ \langle v^\circ, x\rangle_{(L^\infty(\Omega,X))^*,L^\infty(\Omega,X)} = 0$
for all $x \in L^\infty(\Omega, X)$ that vanishes on some $F_n$. 
\end{defn}

The following Yosida--Hewitt-type decomposition result for vector measures was proven in~\cite[Appendix 1, Theorem 3]{Ioffe1972} (with a slight correction to the original proof in \cite{Levin1974}). A related result is in \cite[Theorem 4]{Uhl1971}.
\begin{theorem}[Ioffe and Levin]
\label{thm:ioffe-levin}
Each functional $v^* \in (L^\infty(\Omega,X))^*$ has a unique decomposition 
\begin{equation*}
\label{eq:decomposition-dual-space}
v^* = v + v^\circ,
\end{equation*}
where $v$ is absolutely continuous, $v^\circ$ is singular relative to
$\pP$, and 
\begin{equation*}
  \lVert v^* \rVert_{(L^\infty(\Omega,X))^*} = \lVert v \rVert_{(L^\infty(\Omega,X))^*}+\lVert v^\circ \rVert_{(L^\infty(\Omega,X))^*}.
\end{equation*}
\end{theorem}

Using this decomposition, it is possible to obtain a particularly useful expression for the subdifferential of integral functionals on $L^\infty(\Omega,X)$. The next result was proven in \cite[Corollary 2.7]{Geiersbach2020+}.
\begin{lemma}
\label{cor:characterization-subdifferentials}
Suppose $f$ is a normal convex integrand. Let $\bar{x} \in L^\infty(\Omega,X)$
be a point at which there exists $r >0$ and $k_r \in L^1(\Omega)$ satisfying $f(x(\omega);\omega) \leq k_r(\omega)$ as long as $\lVert x- \bx\rVert_{L^\infty(\Omega,X)} < r$. Then $v^* \in (L^\infty(\Omega, X))^*$
is an element of $\partial I_f(\bx)$ if and only if
\begin{equation}
\label{eq:subdifferential-a.s.}
x^*(\omega) \in \partial f(\bx(\omega);\omega) \quad \text{a.s.},
\end{equation}
where $x^* \in L^1(\Omega,X^*)$ corresponds to the absolutely
continuous part $v$ of $v^*$. 
Moreover,
$\partial I_f(\bx)$ can be identified with a nonempty, weakly compact
subset of $L^1(\Omega, X^*)$. In particular, $v^*$ belongs to
$\partial I_f(\bx)$ if and only if $v^\circ \equiv 0$ and $v \equiv x^*$
satisfies~\eqref{eq:subdifferential-a.s.}.
\end{lemma}

To close this section, we present some auxiliary results for operator-valued random variables that will be of use later in the paper. Let $\mathcal{L}(Y,Z)$ denote the space of all bounded linear operators from $Y$ to $Z$ (both real, reflexive, and separable Banach spaces). The set $L^\infty(\Omega,\mathcal{L}(Y,Z))$ corresponds to all (strongly measurable) operator-valued random variables $\cA \colon \Omega \rightarrow \mathcal{L}(Y,Z)$ (also known as uniformly measurable operator-valued functions) such that $\esssup_{\omega \in \Omega} \lVert \cA(\omega)\rVert_{\mathcal{L}(Y,Z)} < \infty$. 
The adjoint and inverse operators are to be understood in the ``almost sure'' sense; e.g., for $\cA$, the adjoint operator is the (strongly measurable) random
operator $\cA^*$ such that for all $(y, z^*) \in Y \times Z^*$, 
\begin{align*}
\pP(\{\omega \in \Omega: \langle z^*, \cA(\omega)y \rangle_{Z^*,Z} =  \langle  \cA^{*}(\omega)z^*,y \rangle_{Y^*,Y} \}) = 1.
\end{align*}

The next two results are proven in~\Cref{sec:linear-transformations}. In the proofs, and in other arguments throughout the paper, we make use of the fact that bounded sets in $L^\infty(\Omega,Y)$ (the dual of a separable space) equipped with the weak* topology are metrizable; the same applies for $L^\infty(\Omega,Z)$. In particular, topological and sequential notions coincide. 
\begin{lemma}
\label{lemma:regularity-operators}
Let $\cA \in L^\infty(\Omega, \mathcal{L}(Y,Z))$. Then
\begin{enumerate}
\item[(i)] The operators $\cA_1\colon L^\infty(\Omega,Y) \rightarrow L^\infty(\Omega,Z), y(\cdot) \mapsto \cA(\cdot) y(\cdot)$ and $\cA_1^* \colon (L^\infty(\Omega,Z))^* \rightarrow (L^\infty(\Omega,Y))^*$ are bounded. Additionally, $\cA_1^*$ maps singular elements of $(L^\infty(\Omega,Z))^*$ to singular elements of $(L^\infty(\Omega,Y))^*$. Moreover, $\hat{\cA}_1^*\colon L^1(\Omega,Z^*) \rightarrow L^1(\Omega,Y^*),$ $z^*(\cdot) \mapsto \cA^*(\cdot) z^*(\cdot)$ is bounded.
\item[(ii)] The operators $\cA_2\colon Y \rightarrow L^\infty(\Omega,Z), y \mapsto \cA(\cdot)y$ and $\cA_2^* \colon (L^\infty(\Omega,Z))^* \rightarrow Y^*$ are bounded. Moreover, $\hat{\cA}_2^* \colon L^1(\Omega,Z^*) \rightarrow L^1(\Omega, Y^*), z^*(\cdot) \mapsto \cA^*(\cdot) z^*(\cdot)$ is bounded.
\end{enumerate}
\end{lemma}
In \Cref{lemma:regularity-operators}, the notation $\cA(\cdot) y(\cdot)$ signifies an element of $L^\infty(\Omega,Z)$ with the pointwise action $\cA(\cdot)y(\cdot):\Omega \rightarrow Z$, $\omega \mapsto \cA(\omega)y(\omega).$

\begin{corollary}
\label{corollary:affine-maps-are-weak-star-closed}
Let $\cA \in L^\infty(\Omega, \mathcal{L}(Y,Z))$. Then the operator $\cA_1\colon L^\infty(\Omega,Y) \rightarrow L^\infty(\Omega,Z), y(\cdot) \mapsto \cA(\cdot) y(\cdot)$ is weakly*-to-weakly* continuous, and the affine set $\{ y \in L^\infty(\Omega, Y)\colon \cA(\omega) y(\omega) = b(\omega) \text{ a.s.}\}$ is for all $b \in L^\infty(\Omega,Z)$ weakly* closed. Moreover, the operator $\cA_2\colon Y \rightarrow L^\infty(\Omega,Z), y \mapsto \cA(\cdot)y$ is weakly-to-weakly* continuous.
\end{corollary}

\subsection{Risk measures}
\label{subsec:risk-functionals}
As the class of problems represented by  \eqref{eq:model-problem-abstract} contains risk measures, we briefly mention some properties that will later be  used in the main results.  Let $L^p(\Omega)$ be the space of random variables with $p \in [1,\infty)$. 
Throughout the paper, $p'$ will denote the H\"older conjugate of $p$, i.e., $\tfrac{1}{p}+ \tfrac{1}{p'} = 1.$  Following \cite{Artzner1999}, a risk measure $\cR: L^p(\Omega) \rightarrow\R \cup \{ \infty\}$ is called \textit{coherent}, provided it satisfies the following conditions for $\xi,\xi' \in L^p(\Omega)$:
\begin{itemize}
\item[(R1)] \textit{Convexity}: $\cR[\lambda \xi+(1-\lambda)\xi'] \leq \lambda\cR[\xi]+ (1-\lambda)\cR[\xi'] \quad \forall \lambda \in [0, 1],$
\item[(R2)]  \textit{Monotonicity}: If $\xi \leq \xi'$ a.s., then $\cR[\xi] \leq \cR[\xi'],$
\item[(R3)] \textit{Translation equivariance}: If $c \in \R$, then $\cR[\xi+c] = \cR[\xi]+ c$,
\item[(R4)] \textit{Positive homogeneity}: If $\lambda >0$, then $\cR[\lambda \xi] = \lambda \cR[\xi].$
\end{itemize}
Coherent risk measures include, for example, $\cR[\xi] = \E[\xi]$, $\cR[\xi] =\sup \xi$, and average value-at-risk $\cR[\xi] =\text{AVaR}_\alpha[\xi]$ (often called conditional value-at-risk). However, some important risk measures do not satisfy all of these axioms. For the purposes of duality, (R1) and (R2) play a central role, and we focus on this case. If $\mathcal{R}$ satisfies (R1) and is proper and lower semicontinuous, then $\mathcal{R}$ is related to its convex conjugate $\cR^*\colon L^{p'}(\Omega) \rightarrow \R \cup \{\infty\}$ by
\begin{equation}
\label{eq:biconjugate-risk-functional}
\cR[\xi] = \sup_{\vartheta \in  \mathrm{dom}(\mathcal{R}^*)} \{ \E[\xi \vartheta] - \cR^*[\vartheta] \}.
\end{equation}
Notice that $\E[\xi\vartheta] = \langle \vartheta, \xi\rangle_{L^{p'}(\Omega),L^p(\Omega)}.$ If $\cR[\xi]$ is also finite, then the maximum is attained in \eqref{eq:biconjugate-risk-functional} with
\begin{equation}
\label{eq:subdifferentiability-expression}
\partial \cR[\xi] = \argmax_{\vartheta \in \mathrm{dom}(\cR^*)} \{ \E[\xi \vartheta] - \cR^*[\vartheta]\}.
\end{equation}
It is straightforward to show that if (R2) is satisfied, then $\vartheta \in \textup{dom}(\cR^*)$ if and only if $\vartheta \geq 0 \text{ a.s.}$ Additionally, if $\mathcal{R}\colon L^p(\Omega) \rightarrow \R$ satisfies (R1) and (R2) and is everywhere finite, then it is continuous and subdifferentiable on $L^p(\Omega)$ \cite[Proposition 6.5]{Shapiro2009}.

While we do not make use of this fact here, it is worth mentioning that the conditions (R3)-(R4) provide a useful characterization of $\mathrm{dom}(\cR^*)$ and the subdifferential. In particular, a risk measure satisfying (R1)-(R3) is equivalent to the restriction $\vartheta$ to the set of probability density functions (cf.~\cite[Theorem 6.4]{Shapiro2009}), i.e., the set
$ \mathfrak{P}= \{\vartheta \in L^{p'}(\Omega): \E[\vartheta] = 1, \vartheta \geq 0 \text{ a.s.} \}.
$
Finally, a risk measure satisfying (R1)-(R4) implies $\partial \cR[\xi] = \argmax_{\vartheta \in \mathfrak{P}} \E[\xi \vartheta].$
For a more thorough introduction to risk measures, we refer to \cite{Shapiro2009, Pflug2007}.

\section{Karush--Kuhn--Tucker (KKT) Conditions}
\label{sec:extended-KKT}
We now focus on a particular subclass of \eqref{eq:model-problem-abstract}, namely on convex problems where the set $X_{2,\text{ad}}(x_1,\omega)$ restricts $x_2(\omega)$ in the form of an equality and conical constraint. As already mentioned in the introduction, under standard assumptions, classical convex duality theory would provide the existence of Lagrange multipliers in spaces of the form $(L^\infty(\Omega,X))^*$. The formulation of pointwise optimality conditions, which we present in this section, is more delicate.

\subsection{Problem Formulation}
We first equip the problem \eqref{eq:model-problem-abstract} with more structure. For the remainder of the paper, $X_1, X_2, W,$ and $R$ are assumed to be real, reflexive, and separable Banach spaces. We define the equality and inequality constraints by the mappings $e\colon X_1\times X_2\times \Omega \rightarrow W$ and $i \colon X_1 \times X_2 \times \Omega \rightarrow R$, respectively. Given a cone $K \subset R$, the partial order $\leq_K$ is defined by
\begin{equation*}
\label{eq:partial-order}
  r \leq_K 0 \quad \Leftrightarrow \quad -r \in K,
\end{equation*} 
or equivalently, $r \geq_K 0$ if and only if $r \in K$.
With 
\begin{equation*}
x:=(x_1,x_2) \in X:=X_1 \times L^\infty(\Omega,X_2),
\end{equation*}
we define the problem
\begin{equation}
 \label{eq:model-problem-abstract-long}\tag{$\textup{P}$}
 \begin{aligned}
& \underset{x \in X}{\text{minimize}} \quad \{j(x):=J_1(x_1) +\cR[J_2(x)]\}\\
 &\text{s.t.} \quad\left\{\begin{aligned}
     x \in X_0 :=\{x \in X&\colon x_1 \in C\},\\
     e(x_1,x_2(\omega);\omega) &= 0 \text{ a.s.},\\
     i(x_1, x_2(\omega); \omega) &\leq_K 0 \text{ a.s.}
   \end{aligned}\right.
 \end{aligned}
\end{equation}

Throughout our work, we often use the notational convention, where if $x \in X$, $J_2(x)$ is a shorthand for $J_2(x_1,x_2(\cdot);\cdot)$. Similarly, if $x \in X$, we write $e(x)=e(x_1,x_2(\cdot);\cdot)$ and $i(x)=i(x_1,x_2(\cdot);\cdot)$. Below, $K$-convexity of $i(\cdot,\cdot;\omega)$ means 
\begin{equation*}
i(\lambda x_1 + (1-\lambda)\hat{x}_1, \lambda x_2 + (1-\lambda)\hat{x}_2;\omega) \leq_K \lambda i(x_1,x_2;\omega) + (1-\lambda)i(\hat{x}_1, \hat{x}_2;\omega) \quad \forall (x_1,x_2),(\hat{x}_1,\hat{x}_2) \in X_1 \times X_2.
\end{equation*}

Problem  \eqref{eq:model-problem-abstract-long} is subject to the following assumptions.
\begin{assumption}
\label{assumption:general-problem}
\subasu \label{subasu:general-i} $C \subset X_1$ is nonempty, closed, and convex and $K \subset R$ is a nonempty, closed, and convex cone. \\
\subasu \label{subasu:general-ii} $J_1$ is convex, lower semicontinuous, and finite on $X_1$.\\ 
\subasu \label{subasu:general-iii} $J_2(\cdot,\cdot;\omega)$ is continuous, convex, and finite with respect to $X_1 \times X_2$ for almost all $\omega \in \Omega$ and for every $(x_1,x_2) \in X_1\times X_2$, the functions $J_2(x_1,x_2;\cdot)$ are measurable on $\Omega$. For every $r>0$ there exists $a_r \in L^p(\Omega)$, $p \in [1,\infty)$ such that for almost all $\omega$ and any $\lVert x_1\rVert_{X_1} + \lVert x_2\rVert_{X_2} \leq r$, we have
\begin{equation}
\label{eq:growth-condition-J2}
|J_2(x_1,x_2;\omega)| \leq a_r(\omega).  
\end{equation}
\subasu \label{subasu:general-iv} With respect to $X_1 \times X_2$, $e(\cdot,\cdot;\omega)$ is continuous and linear and $i(\cdot,\cdot;\omega)$ is continuous and $K$-convex for almost all $\omega \in \Omega$.\\
\subasu \label{subasu:general-v} For every $(x_1, x_2) \in X_1 \times X_2$, the functions $e(x_1,x_2;\cdot)$ and $i(x_1, x_2; \cdot)$ are measurable on $\Omega$. For every $r>0$, there exist constants $b_{r,e} >0$ and $b_{r,i} >0$ such that for any $\lVert x_1\rVert_{X_1} + \lVert x_2\rVert_{X_2} \leq r$, we have
\begin{equation}
\label{eq:growth-condition-constraints}
  \|e(x_1,x_2;\omega)\|_{W} \le b_{r,e}, \quad \|i(x_1,x_2;\omega)\|_{R} \le b_{r,i} \quad \text{a.s.}
\end{equation}
\subasu \label{subasu:general-vi} $\mathcal{R}\colon L^p(\Omega) \rightarrow \R$ is convex, monotone, and everywhere finite. 
\end{assumption}

\begin{rem}
\label{rem:growth-conditions}
The growth condition \eqref{eq:growth-condition-J2} in combination with the Carath\'eodory property ensures that $J_{2}\colon X \rightarrow L^p(\Omega)$ is continuous; see \cite[Theorem 3.5]{Kouri2018}. Continuity in the case $p = \infty$ is possible under additional assumptions; see \cite[Theorem 3.17]{Appell1990}. The conditions \eqref{eq:growth-condition-constraints} ensure  $e(\chi_1(\cdot),\chi_2(\cdot); \cdot) \in L^\infty(\Omega,W)$ and $i(\chi_1(\cdot),\chi_2(\cdot); \cdot) \in L^\infty(\Omega,R)$ for all $(\chi_1,\chi_2) \in L^\infty(\Omega,X_1) \times L^\infty(\Omega,X_2)$. This structure is used in \Cref{thm:extended-KKT-conditions} in connection with the nonanticipativity constraint.
\end{rem}
\begin{rem}
\label{remark:weak-star-lsc-2}
Suppose $h\colon X_1 \times X_2 \times \Omega \rightarrow \R$ is a normal convex integrand and $h(x_1,x_2(\omega);\omega)$ is an integrable function of $\omega$ for every $x \in X$. Then $x \mapsto I_h(x)$ is a weakly* lower semicontinuous mapping on $X$. This can be seen by the pairing $X$ with $X':=X_1^* \times L^1(\Omega,X_2^*)$  and arguing that $I_h$ and $I_{h^*}$ are well-defined on $X$ and $X'$, respectively. These functionals are conjugate to each other as in \Cref{lemma:duality-conjugate-functions} with
\begin{align*}
I_{h^*}(x') = \sup_{x \in X} \{ \langle x', x\rangle_{X',X} - I_h(x)\} \quad \text{and}\quad I_h(x) &= \sup_{x' \in X'} \{  \langle x', x\rangle_{X',X} - I_{h^*}(x')\}.
\end{align*}
This gives weak* lower semicontinuity using the arguments in \Cref{remark:weak-star-lsc}. Here, the weak* topology coincides with the weak topology on (the reflexive space) $X_1$, so one even has weak-weak* lower semicontinuity of $I_h$, i.e., weak with respect to $X_1$ and weak* with respect to  $L^\infty(\Omega,X_2)$. To simplify the discussion, we will not make the distinction between weak and weak* with respect to $x_1$.
\end{rem}

The following result is an adaptation of \cite[Proposition 3.8]{Kouri2018}.

\begin{proposition}
\label{prop:weak-weak*-lsc}
Suppose \Cref{subasu:general-iii} holds and $\mathcal{R}\colon L^p(\Omega) \rightarrow \R\cup \{ \infty\}$ is convex and monotone. If $\cR$ is finite and continuous at $J_2(x)$, then $\mathcal{R} \circ J_{2}$ is weakly* lower semicontinuous at $x$.
\end{proposition}

\begin{proof}
Let $f(x_1,x_2;\omega):= J_{2}(x_1,x_2; \omega) \vartheta(\omega)$, where $\vartheta \in \textup{dom}(\cR^*) \subset L^{p'}(\Omega)$. This function is clearly a normal convex integrand since $\vartheta \geq 0$ a.s.~by monotonicity of $\cR$. Moreover \eqref{eq:growth-condition-J2} in combination with $\vartheta \in L^{p'}(\Omega)$ imply that $f$ is integrable on $X$.
Weak* lower semicontinuity of $I_f$ on $X$ follows by \Cref{remark:weak-star-lsc-2}.

Now, suppose $\{x^n \} \subset X$ is a sequence such that $x^n \rightharpoonup^* x$. Then
\begin{equation}
\label{eq:inequality-risk-subdifferential}
\begin{aligned}
\cR[{J}_{2}(x^n)] &= \sup_{\vartheta \in \textup{dom}(\cR^*)} \{ \E[ {J}_{2}(x^n) \vartheta] - \cR^* [\vartheta]\} \geq \E[ {J}_{2}(x^n)\theta] - \cR^* [\theta]
\end{aligned}
\end{equation}
for any $\theta \in \textup{dom}(\cR^*)$. Since $\cR$ is finite and continuous at ${J}_{2}(x)$, it is subdifferentiable there (cf.~\cite[Proposition 7.74]{Shapiro2009}), so \eqref{eq:inequality-risk-subdifferential} holds for $\theta \in \partial \cR[J_{2}(x)].$ With \eqref{eq:subdifferentiability-expression} and weak* lower semicontinuity of $I_f$ on $X$, weak* lower semicontinuity of $\cR \circ J_{2}$ at $x$ follows.
\end{proof}

As a result of \Cref{prop:weak-weak*-lsc}, \Cref{subasu:general-iii}, and \Cref{subasu:general-vi} imply weak* lower semicontinuity of $\cR \circ J_2$ on $X$.

\subsection{Definition of Lagrangians}
To obtain pointwise optimality conditions, it will be advantageous to use the decomposition afforded by \Cref{thm:ioffe-levin}. This structure will be used to distinguish between the extended (standard) Lagrangian associated with the natural pairing and the Lagrangian associated with the weak* pairing.
We define the sets $\Lambda :=L^1(\Omega,W^*) \times L^1(\Omega,R^*)$ and $\Lambda^\circ := \{ \z^\circ = (\p^\circ,\lambda_i^\circ) \in \mathcal{S}_e \times \mathcal{S}_i\},$
where $\mathcal{S}_e$ and $\mathcal{S}_i$ denote the sets of singular functionals (as in~\Cref{def:singular-functionals}) defined on $L^\infty(\Omega,W)$ and $L^\infty(\Omega,R)$, respectively. We define the dual cones
\begin{equation*}
\begin{aligned}
K^\oplus &:= \{  r^* \in R^*: \langle r^*,r \rangle_{R^*,R} \geq 0 \quad \forall r \in K \},\\
\mathcal{K}^\oplus &:= \{\lambda_i^\circ \in \mathcal{S}_i: \langle \lambda_i^\circ, y \rangle_{(L^\infty(\Omega,R))^*,L^\infty(\Omega,R)} \geq 0 \quad \forall y \in L^\infty(\Omega,R) \text{ satisfying } y(\omega)\in K \text{ a.s.}\}.
\end{aligned}
\end{equation*}

Now, we define the admissible sets
$\Lambda_0 =\{ \lambda=(\p,\lambda_i) \in \Lambda \colon \lambda_i(\omega) \in K^\oplus \text{ a.s.}\}$ and $\Lambda_0^\circ = \{ \z^\circ = (\p^\circ,\lambda_i^\circ) \in \Lambda^\circ \colon  \lambda_i^\circ \in \mathcal{K}^\oplus \}$ corresponding to the associated dual problem.

Finally, we define the \textit{extended Lagrangian} over $X_0 \times \Lambda_0 \times \Lambda_0^\circ$ by
\begin{equation}
\label{eq:Lagrangian-extended}
\bar{L}(x,\z,\z^\circ)= L(x,\lambda)+ L^\circ(x,\z^\circ),
\end{equation}
where
\begin{align*}
L(x,\lambda)& := j(x)+\langle \p, e(x) \rangle_{L^1(\Omega,W^*), L^\infty(\Omega,W)} +\langle \lambda_i, i(x) \rangle_{L^1(\Omega,R^*), L^\infty(\Omega,R)} \\
& \phantom{:}=j(x)+ \E[\langle \p, e(x) \rangle_{W^*,W}+\langle \lambda_i, i(x) \rangle_{R^*,R}]
\end{align*}
is called the \textit{Lagrangian} and
\begin{align*}
  L^\circ(x,\z^\circ) &:=\langle \p^\circ, e(x)\rangle_{(L^\infty(\Omega,W))^*,L^\infty(\Omega,W)} +\langle \lambda_i^\circ, i(x)\rangle_{(L^\infty(\Omega,R))^*,L^\infty(\Omega,R)}.
\end{align*}
By convention, we set $\bar{L}(x,\lambda,\lambda^\circ) = -\infty$ if $x \in X_0$ but $(\lambda,\lambda^\circ) \not\in \Lambda_0 \times \Lambda_0^\circ$, and $\bar{L}(x,\lambda,\lambda^\circ) = \infty$ if $x \not\in X_0$.

By \Cref{thm:ioffe-levin}, $\lambda_e$ and $\lambda_i$ can be identified with elements $\lambda_e^a$ and $\lambda_i^a$ of the spaces $(L^\infty(\Omega,W))^*$ and $(L^\infty(\Omega,R))^*$, respectively. Hence \eqref{eq:Lagrangian-extended} simply corresponds to the standard Lagrangian when 
\begin{equation*}
U:=L^\infty(\Omega,W)\times L^\infty(\Omega,R)
\end{equation*}
is paired with its dual $U^*$.

We will now present the conditions under which saddle points to the extended Lagrangian $\bar{L}$ exist. To that end, let $u=(u_e,u_i) $
be a point in a neighborhood of zero in $U$. We define the perturbed admissible set 
\begin{align*}
F_{\text{ad}, u}&:=\{x \in X \colon x_1 \in C,e(x_1,x_2(\omega);\omega) = u_e(\omega) \text{ a.s.}, i(x_1, x_2(\omega); \omega) \leq_K u_i(\omega) \text{ a.s.} \} 
\end{align*}
and the value function
\begin{equation*}
v(u):=\inf_{x \in X} j(x) + \delta_{F_{\textup{ad},u}}(x).
\end{equation*}

For the existence of saddle points, we impose the following standard assumption.
\begin{assumption}
\label{assumption:existence-saddle-points}
\setcounter{subassumption}{0}
\subasu \label{subasu:existence-saddle-points-i}
$F_{\text{ad}, u}$ is bounded for all $u$ in a neighborhood of zero or $j$ is radially unbounded, meaning $\lVert x \rVert_X \rightarrow \infty$ implies $j(x) \rightarrow \infty$. \\
\subasu \label{subasu:existence-saddle-points-ii} The problem is strictly feasible in the sense that 
\begin{equation}
\label{eq:constraint-qualification}
0 \in \textup{int}\,\textup{dom}\, v. 
\end{equation}
\end{assumption}
The condition \eqref{eq:constraint-qualification} is satisfied if $v$ is finite in a neighborhood of zero and corresponds to an ``almost sure''-type Robinson constraint qualification. Suppose
\begin{equation*}
G(x_1, x_2; \omega):= \begin{pmatrix}
e(x_1, x_2; \omega)\\
i(x_1, x_2; \omega)
\end{pmatrix} , \quad \hat{K}:= \begin{pmatrix}
\{ 0\}\\
K
\end{pmatrix}.
\end{equation*}
Notice that the condition \eqref{eq:constraint-qualification} implies that for all $u$ in a neighborhood of zero in $L^\infty(\Omega,W) \times L^\infty(\Omega,R)$, there exists $x_1 \in C$, $x_2 \in L^\infty(\Omega,X_2)$ such that  
\begin{equation}
\label{eq:metric-regularity}
-G(x_1,x_2(\omega);\omega)-u(\omega) \in \hat{K} \quad \text{a.s.}
\end{equation}
Note that $G(\cdot,\cdot;\omega)$ is $\hat{K}$-convex.
Then the condition \eqref{eq:metric-regularity} is equivalent to a Robinson's constraint qualification for fixed $\omega$ in the case where $G(\cdot,\cdot;\omega)$ is continuously differentiable on $X_1\times X_2$; cf.~\cite[Proposition 2.104]{Bonnans2013}.

To obtain ``regular'' Lagrange multipliers, meaning multipliers with respect to the Lagrangian $L$, we make an additional assumption. 
\begin{assumption}[relatively complete recourse]
\label{assumption:relatively-complete-recourse}
The induced feasible set 
\begin{align*}
\tilde{C} := \{ x_1 \in X_1 & \colon \exists x_2 \in L^\infty(\Omega,X_2) \text{ satisfying } e(x_1,x_2(\omega);\omega) = 0 \text{ a.s.},   i(x_1, x_2(\omega); \omega) \leq_K 0 \text{ a.s.} \}
\end{align*} contains the first-stage feasible set, i.e. $C \subset \tilde{C}.$
\end{assumption}

\Cref{assumption:relatively-complete-recourse} is quite strong, essentially requiring the second-stage variable to be feasible for any feasible first-stage variable. In \Cref{subsec:regular-mutlipliers}, we will see that for problems satisfying this condition, Lagrange multipliers happen to be more regular. In that case, integrable multipliers in $\Lambda$ can be found and the singular terms from $\Lambda^\circ$ vanish.

\subsection{KKT conditions}
In this section, we present strong KKT conditions without the additional \Cref{assumption:relatively-complete-recourse}. This leads to Lagrange multipliers belonging to $\Lambda \times \Lambda^\circ$, i.e., multipliers have singular parts.
\Cref{thm:extended-KKT-conditions} uses arguments made in \cite[Section 5]{Rockafellar1976a}, where the finite-dimensional case was considered. Many arguments from \cite{Geiersbach2020+} can be carried over; we repeat some arguments here and in \Cref{section:saddle-points-relatively-complete-recourse} for completeness. First, we obtain the existence of saddle points.

\begin{lemma}
\label{lemma:saddle-points-lagrangian}
Let \Cref{assumption:general-problem} and \Cref{assumption:existence-saddle-points} be satisfied. Then there exists a saddle point $(\bar{x},\bz,\bz^\circ) \in X_0 \times \Lambda_0 \times \Lambda_0^\circ$ to the Lagrangian $\bar{L}$, i.e, the point $(\bar{x},\bz,\bz^\circ)$ satisfies
\begin{equation}
\label{eq:saddle-point-extended-Lagrangian}
\bar{L}(\bx,\z,\z^\circ)  \leq \bar{L}(\bx,\bz,\bz^\circ)\leq \bar{L}(x,\bz,\bz^\circ) \quad \forall (x,\z,\z^\circ) \in X \times \Lambda \times \Lambda^\circ.
\end{equation}
\end{lemma}
\begin{proof}
This is proven in \Cref{thm:minP-supD} and \ref{thm:infP-maxDbar}.
\end{proof}

It will be convenient to define the function
\begin{equation}
\label{eq:definition-singular-functional}
 \ell(x_1,{\lambda}^\circ) =\inf_{x_2 \in L^\infty(\Omega,X_2)} L^\circ(x,\lambda^\circ).
\end{equation}

Let us start with a technical lemma, which adapts  \cite[Theorem 3]{Rockafellar1976b} to include risk functionals.
\begin{lemma}
\label{lemma:technical-for-singular-terms}
Let \Cref{assumption:general-problem} be satisfied. Then for every $(\lambda,\lambda^\circ) \in \Lambda_0 \times \Lambda_0^\circ$, we have
\begin{equation*}
\inf_{x \in X_0} \{ L(x,\lambda) + L^\circ(x,\lambda^\circ) \} = \inf_{x \in X_0} \{ L(x,\lambda) + \ell(x_1,\lambda^\circ )\}.
\end{equation*}
\end{lemma}

\begin{proof}
To begin, notice that for all $x_1 \in C_1$, we trivially have
\begin{equation*}
 \inf_{x_2 \in L^\infty(\Omega,X_2)} \{ L(x,\lambda) + L^\circ(x,\lambda^\circ)\} \geq  \inf_{x_2 \in L^\infty(\Omega,X_2)} L(x,\lambda) + \inf_{x_2 \in L^\infty(\Omega,X_2)} L^\circ(x,\lambda^\circ).
\end{equation*}
Hence it is enough to show the inequality ``$\leq$'', or equivalently, that for arbitrary $\varepsilon >0$ and arbitrarily fixed $x':=(x_1, y'), x'':=(x_1, y'') \in X_0$, there exists $x=(x_1,y) \in X_0$ such that the inequality
\begin{equation}
\label{eq:proof-infP-maxD-q2}
 L(x,\lambda) + L^\circ(x,\lambda^\circ) \leq L(x'',\lambda) + L^\circ(x',\lambda^\circ) + \varepsilon
\end{equation}
is true. Then \eqref{eq:proof-infP-maxD-q2} in combination with the definition \eqref{eq:definition-singular-functional} will prove the assertion.

Let $\{ F_{e,n}\}$ and $\{ F_{i,n}\}$ be the sets on which $\lambda_e^\circ$ and $\lambda_i^\circ$ are concentrated according to \Cref{def:singular-functionals}.
We define $F_n = F_{e,n}\cup F_{i,n}$ and
\begin{equation*}
  y^n(\omega) = \begin{cases}
y'(\omega), &\omega \in F_n\\
y''(\omega), &\omega \not\in F_n
  \end{cases}.
\end{equation*}
On $F_n$, we have
$e(\u,y^n(\omega);\omega) = e(\u,y'(\omega);\omega)$ and
$i(\u,y^n(\omega);\omega) = i(\u,y'(\omega);\omega).$ Therefore, thanks to \Cref{def:singular-functionals}, for $x^n = (x_1,y^n)$ we have 
\begin{equation*}
\langle  \p^\circ, e(x^n)-e(x')\rangle_{(L^\infty(\Omega,W))^*,L^\infty(\Omega,W)} = 0, \quad
\langle  \lami^\circ, i(x^n) - i(x')\rangle_{(L^\infty(\Omega,R))^*,L^\infty(\Omega,R)} = 0,
\end{equation*}
which implies
\begin{equation}
\label{eq:equality-on-singular-Lagrangian}
L^\circ(x^n,\lambda^\circ) = L^\circ(x',\lambda^\circ) \quad \forall n.
\end{equation}
Let $\mathbbm{1}_A$ denote the indicator function on $A$, defined by $\mathbbm{1}_A(x) = 1$ if $x \in A$ and $\mathbbm{1}_{A}(x) = 0$ otherwise. Then $J_2(x_1, y^n(\cdot); \cdot) \rightarrow J_2(x_1, y''(\cdot); \cdot)$ in $L^p(\Omega)$ as $n \rightarrow \infty.$ Indeed, 
\begin{equation*}
\begin{aligned}
&\lVert J_2(x_1, y^n(\cdot); \cdot) -   J_2(x_1, y''(\cdot); \cdot)\rVert_{L^p(\Omega)}^p \\
&\quad = \int_\Omega |J_2(x_1, y''(\omega); \omega)\mathbbm{1}_{\Omega\backslash F_n}(\omega) + J_2(x_1, y'(\omega); \omega)\mathbbm{1}_{F_n}(\omega) -   J_2(x_1, y''(\omega); \omega)|^p \D \pP(\omega) \\
& \quad = \int_{\Omega\backslash F_n} |J_2(x_1, y''(\omega); \omega) -   J_2(x_1, y''(\omega); \omega)|^p \D \pP(\omega) \\
&\quad \quad\quad + \int_{F_n}  |J_2(x_1, y'(\omega); \omega) -   J_2(x_1, y''(\omega); \omega)|^p \D \pP(\omega).
\end{aligned}
\end{equation*}
Since the integrand in the last integral above is bounded and $\pP(F_n) \rightarrow 0$, the last integral vanishes. We now show that $x^n$ converges strongly to $x''$ in $X$. Notice that
\begin{align*}
\lVert x^n - x''\rVert_{X} &= \lVert y^n - y''\rVert_{L^\infty(\Omega, X_2)} = \esssup_{\omega \in \Omega} \lVert y^n(\omega)-y''(\omega) \rVert_{X_2}\\
& = \esssup_{\omega \in F_n} \lVert y^n(\omega)-y''(\omega) \rVert_{X_2} \\
&= \inf \{a \in \R: \pP\{ \omega \in F_n : \lVert y^n(\omega)-y''(\omega) \rVert_{X_2} > a  \} =0 \}.
\end{align*}
In combination with $\pP(F_n) \rightarrow 0$ as $n\rightarrow \infty$, the event $\{ \omega \in F_n : \lVert y^n(\omega)-y''(\omega) \rVert_{X_2} > a  \} $ has vanishing small probability for all $a \geq 0$ as $n \rightarrow \infty$, i.e., in particular for the choice $a=0$. Recall that $J_2  \colon X \rightarrow L^p(\Omega)$ is continuous by \Cref{rem:growth-conditions}. Now, continuity of $\cR: L^p(\Omega) \rightarrow \R$ gives 
\begin{equation*}
\lim_{n\rightarrow \infty} \cR[J_2(x^n)] = \cR[J_2(x'')].
\end{equation*}
Analogous arguments yield 
\begin{equation*}
\lim_{n \rightarrow \infty} \E[\langle \lambda_e,e(x^n)\rangle_{W^*,W}] =  \E[\langle \lambda_e,e(x'')\rangle_{W^*,W}], \quad \lim_{n \rightarrow \infty} \E[\langle \lambda_i,i(x^n)\rangle_{R^*,R}] = \E[\langle \lambda_i,i(x'')\rangle_{R^*,R}].
\end{equation*}
As a result, we obtain $\lim_{n \rightarrow \infty} L(x^n, \lambda) = L(x'', \lambda).$
In combination with \eqref{eq:equality-on-singular-Lagrangian},  for any $\varepsilon >0$, there exists an $n_0$ such that for all $n\geq n_0$ we have 
\begin{align*}
 L(x^n,\lambda) +L^\circ(x^n, \lambda^\circ)\leq L(x'',\lambda) + L^\circ(x', \lambda^\circ) + \varepsilon.
\end{align*}
After setting $x:=x^n$ (for an arbitrary $n \geq n_0$), we obtain \eqref{eq:proof-infP-maxD-q2}.
\end{proof}

\begin{theorem}
\label{thm:extended-KKT-conditions}
Let \Cref{assumption:general-problem} be satisfied. Then $(\bar{x}, \bar{\lambda},\bar{\lambda}^\circ) \in X_0 \times \Lambda_0 \times \Lambda_0^\circ$ satisfies \eqref{eq:saddle-point-extended-Lagrangian} if and only if there exist $\bar{\rho} \in L^1(\Omega, X_1^*)$ and $\bar{\vartheta} \in \partial \cR[J_2(\bx)]$ fulfilling the following conditions:
\begin{enumerate}
\item[(i)] The function
\begin{equation*}
x_1 \mapsto J_1(x_1) + \langle \E[\bar{\rho}], x_1\rangle_{X_1^*,X_1} + \ell(x_1,\bar{\lambda}^\circ)
\end{equation*}
 attains its minimum over $C$ at $\bar{x}_1.$
\item[(ii)] The function 
\begin{align*}
(x_1,x_2) &\mapsto J_2(x_1,x_2;\omega)\bar{\vartheta}(\omega)+\langle \bar{\lambda}_e(\omega),e(x_1,x_{2};\omega)\rangle_{W^*,W} \\
&\quad + \langle \bar{\lambda}_i(\omega), i(x_1,x_2;\omega)\rangle_{R^*,R} - \langle \bar{\rho}(\omega), x_1\rangle_{X_1^*,X_1}
\end{align*}
attains its minimum in $X_1 \times X_2$ at $(\bar{x}_1,\bar{x}_2(\omega))$ for almost every $\omega \in \Omega$.
\item[(iii)] It holds that $\bar{x}_1 \in C$ and the following conditions hold almost surely:
\begin{align*}
e(\bar{x}_1,\bar{x}_2(\omega);\omega) &= 0,\\
 \bar{\lambda}_i(\omega) \in K^\oplus, \quad i(\bar{x}_1,\bar{x}_2(\omega);\omega)\leq_K 0,& \quad \langle \bar{\lambda}_i(\omega), i(\bar{x}_1, \bar{x}_2(\omega); \omega)\rangle_{R^*,R}=0.
\end{align*}
\item[(iv)] It holds that $ \ell(\bar{x}_1,\bar{\lambda}^\circ) = 0$ and
\begin{align*}
\bar{\lambda}_i^\circ \in \mathcal{K}^\oplus, \quad \langle \bar{\lambda}_i^\circ, i(\bx)\rangle_{(L^\infty(\Omega,R))^*, L^\infty(\Omega,R)}=0.
\end{align*}
\end{enumerate}
\end{theorem}

The origin of the multiplier $\bar{\rho}$ in \Cref{thm:extended-KKT-conditions} will become apparent in the proof. There, in view of applying the Moreau--Rockafellar theorem, the variable $x_1 \in X_1$ first needs to be identified with its corresponding element in the Bochner space $L^\infty(\Omega,X_1)$. This identification is equivalent to imposing $x_1(\omega) \equiv x_1$, meaning $x_1(\omega) = x_1$ a.s. Sometimes, this condition is imposed explicitly in the problem formulation and is known as a ``nonanticipativity'' constraint; see \cite[Section 2.4]{Shapiro2009}.
\begin{rem}
\label{rem:interpretation-ell-condition}
Recall the convention $\bar{L}(x,\lambda,\lambda^\circ) = -\infty$ if $x \in X_0$ and $(\lambda,\lambda^\circ) \not\in \Lambda_0 \times \Lambda_0^\circ$. Given $\lambda^\circ \in \Lambda^\circ$, we have the implication
 \begin{equation}
\label{eq:feasibility-negative-Lagrangian-term-1}
e(\bx_1,\bx_2(\omega);\omega) = 0 \text{ a.s.}, \quad i(\bx_1, \bx_2(\omega);\omega) \leq_K 0 \text{ a.s.} \quad \Rightarrow \quad L^\circ(\bx,\lambda^\circ)\leq 0.
\end{equation}
Hence conditions (iii) and (iv) imply 
\begin{equation}
\label{eq:ell-function-zero}
0=\ell(\bar{x}_1,\bar{\lambda}^\circ) \leq L^\circ(\bar{x},\bar{\lambda}^\circ) \leq 0.
\end{equation}
\end{rem}

\begin{proof}[Proof of \Cref{thm:extended-KKT-conditions}]
We show that if conditions (i)--(iv) are satisfied, then a saddle point exists. Notice that $\bar{\vartheta} \in \partial \cR[J_2(\bx)]$,  \eqref{eq:biconjugate-risk-functional}, and \eqref{eq:subdifferentiability-expression} imply with $\xi=J_2(\bx)$
\begin{equation}
\label{eq:inequality-risk-functional-1}
\E[J_2(\bx)\bar{\vartheta}] - \cR^*[\bar{\vartheta}] = \cR[J_2(\bx)].
\end{equation}
Additionally, \eqref{eq:biconjugate-risk-functional} implies with $\xi=J_2(x)$
\begin{equation}
\label{eq:inequality-risk-functional}
\begin{aligned}
 \E[J_2(x)\bar{\vartheta}]- \cR^*[\bar{\vartheta}] &\leq \sup_{\vartheta \in \textup{dom} (\cR^*)} \{ \E[J_2(x){\vartheta}]- \cR^*[{\vartheta}]\} =\cR[J_2(x)].
\end{aligned}
\end{equation}
Now,
\begin{align*}
\bar{L}(\bar{x},\bar{\lambda},\bar{\lambda}^\circ)  &= L(\bar{x},\bar{\lambda}) + L^\circ(\bar{x},\bar{\lambda}^\circ)\\
&\overset{\eqref{eq:ell-function-zero}}{=}J_1(\bx_1)+\cR[J_2(\bx)]  +\E[\langle \bar{\lambda}_e, e(\bx)\rangle_{W^*,W} + \langle \bar{\lambda}_i, i(\bx)\rangle_{R^*,R}]  + \ell(\bx_1,\bz^\circ)\\
& \overset{\eqref{eq:inequality-risk-functional-1}}{=} J_1(\bx_1)+  \ell(\bx_1,\bz^\circ)+\E[J_2(\bx)\bar{\vartheta}] - \cR^*[\bar{\vartheta}]  +\E[\langle \bar{\lambda}_e, e(\bx)\rangle_{W^*,W} +\langle \bar{\lambda}_i, i(\bx)\rangle_{R^*,R}] \\
& = J_1(\bx_1)+ \langle \E[\bar{\rho}],\bx_1 \rangle_{X_1^*,X_1} + \ell(\bx_1,\bz^\circ)+\E[J_2(\bx)\bar{\vartheta}] - \cR^*[\bar{\vartheta}]  \\
& \quad \quad +\E[\langle \bar{\lambda}_e, e(\bx)\rangle_{W^*,W} +\langle \bar{\lambda}_i, i(\bx)\rangle_{R^*,R} - \langle \bar{\rho},\bx_1 \rangle_{X_1^*,X_1} ]  \\
&\overset{\textup{(i),(ii)}}{\leq}J_1(x_1)+\langle \E[\bar{\rho}],x_1 \rangle_{X_1^*,X_1} + \ell(x_1,\bz^\circ) + \E[J_2(x)\bar{\vartheta}] - \cR^*[\bar{\vartheta}] \\
& \quad \quad +\E[\langle \bar{\lambda}_e, e(x)\rangle_{W^*,W} + \langle \bar{\lambda}_i, i(x)\rangle_{R^*,R}- \langle \bar{\rho},x_1 \rangle_{X_1^*,X_1} ]\\
& \overset{\eqref{eq:inequality-risk-functional}}{\leq} \bar{L}(x,\bz,\bz^\circ) \quad \forall x \in X.
\end{align*}
Note that we used monotonicity of the expectation operator in the second to last inequality above.

To prove that $\bar{L}(\bar{x},\z,\z^\circ)\leq \bar{L}(\bx,\bz,\bz^\circ)$ for all $(\z,\z^\circ) \in \Lambda \times \Lambda^\circ$, we first prove $L(\bx, \lambda) \leq L(\bx, \bar{\lambda})$ for all $\lambda \in \Lambda$. Notice this follows if
\begin{equation*}
\E[\langle \lambda_e,e(\bx)\rangle_{W^*,W} +\langle \lambda_i,i(\bx)\rangle_{R^*,R} ] \leq \E[\langle \bar{\lambda}_e,e(\bx)\rangle_{W^*,W} +\langle \bar{\lambda}_i,i(\bx)\rangle_{R^*,R} ],
\end{equation*}
for all $\lambda \in \Lambda$. This is fulfilled thanks to condition (iii) and the fact that $\langle \lambda_i(\omega), i(\bx_1, \bx_2(\omega);\omega)\rangle_{R^*,R} \leq 0$ a.s.~for all $\lambda \in \Lambda$. Now, it is enough to argue that $L^\circ(\bx,\lambda^\circ) \leq L^\circ(\bx,\bz^\circ)$ for all $\lambda^\circ \in \Lambda^\circ.$ This is also clear, since given $\lambda^\circ \in \Lambda^\circ$ and conditions (iii) and (iv), \eqref{eq:ell-function-zero} holds, hence $0 = L^\circ(\bar{x},\bz^\circ) = \sup_{\lambda^\circ \in \Lambda^\circ} L^\circ(\bx,\z^\circ) \geq L^\circ(\bx,\z^\circ).$

For the second part of the proof, we show that $(\bx,\bz,\bz^\circ)$ being a saddle point implies conditions (i)--(iv).  Trivially, $(\bar{x}, \bar{\lambda}, \bar{\lambda}^\circ) \in X_0 \times \Lambda_0 \times \Lambda_0^\circ$  implies $\bx_1 \in C$, $\bar{\lambda}_i(\omega) \in K^\oplus$ a.s., and $\bz_i^\circ \in \mathcal{K}^\oplus$. It is straightforward to argue that $e(\bx_1,\bx_2(\omega);\omega) = 0$ a.s.~and $i(\bx_1,\bx_2(\omega);\omega) \leq_K 0$ a.s., since otherwise one would have $\sup_{\lambda \in \Lambda} L(\bx, \lambda) = \infty$. 
Now,
\begin{equation*}
\label{eq:saddle-point-singular-terms-zero}
\begin{aligned}
0 &\geq   L^\circ(\bx,\bar{\lambda}^\circ) = \sup_{\lambda^\circ \in \Lambda_0^\circ} \{\langle \lambda_e^\circ, e(\bx)\rangle_{(L^\infty(\Omega,W))^*,L^\infty(\Omega,W)} +\langle \lambda_i^\circ, i(\bx)\rangle_{(L^\infty(\Omega,R))^*,L^\infty(\Omega,R)} \}.
\end{aligned}
\end{equation*}
Since $0 \in \mathcal{K}^\oplus$, we obtain
\begin{equation*}
\label{eq:singular-Lagrange-disappears}
L^\circ(\bar{x},\bar{\lambda}^\circ) = 0
\end{equation*}
and thus $\ell(\bx_1,\bar{\lambda}^\circ) = 0$ and $\langle \lambda_i^\circ, i(\bx)\rangle_{(L^\infty(\Omega,R))^*,L^\infty(\Omega,R)} =0.$ Analogous arguments can be made to conclude that $\langle \bar{\lambda}_i(\omega),i(\bx_1,\bx_2(\omega);\omega)\rangle_{R^*,R} = 0$ a.s.
We have shown that conditions (iii)-(iv) hold.

Now, we show that conditions (i)--(ii) must hold. Since $(\bx, \bar{\lambda}, \bar{\lambda}^\circ)$ is a saddle point, we have in combination with \Cref{lemma:technical-for-singular-terms} that
\begin{align}
\label{eq:optimality-saddle-point-primal}
&\bar{L}(\bx,\bz,\bz^\circ)  = \inf_{x \in X_0} \bar{L}(x,\bz,\bz^\circ)=  \inf_{x \in X_0} \{L(x,\bar{\lambda})  + \ell(x_1,\bar{\lambda}^\circ) \}.
\end{align}
We first work on rewriting the expression $L(x,\bar{\lambda})  + \ell(x_1,\bar{\lambda}^\circ).$ We define the convex function 
\begin{equation}
\label{eq:definition-H1}
H_1(x) = \begin{cases}
 J_1(x_1) + \ell(x_1, \bar{\lambda}^\circ), & x \in X_0,\\
 \infty, & \text{else}.
 \end{cases}
\end{equation}
Additionally, we define the integrand
\begin{equation*}
\begin{aligned}
h_2(x_1,x_2;\omega)& = J_2(x_1,x_2; \omega)\bar{\vartheta}(\omega) + \langle \bar{\lambda}_e(\omega), e(x_1,x_2;\omega)\rangle_{W^*,W} \\
&\quad \quad +  \langle \bar{\lambda}_i(\omega), i(x_1,x_2;\omega)\rangle_{R^*,R} -\cR^*[\bar{\vartheta}].
\end{aligned}
\end{equation*}
\Cref{assumption:general-problem} ensures that $h_2$ is a Carath\'eodory, i.e., continuous in $(x_1,x_2)$ for every $\omega$ and measurable in $\omega$ for every $(x_1,x_2)$. Additionally, $h_2$ is everywhere finite and convex with respect to $(x_1,x_2)$; recall that the monotonicity of $\cR$ implies $\bar{\vartheta} \geq 0$ a.s. Note also that $\bar{\lambda} \in \Lambda_0$ implies $\bar{\lambda}_i(\omega) \in K^\oplus$ a.s., giving convexity of the mapping $(x_1,x_2) \mapsto \langle \bar{\lambda}_i(\omega), i(x_1,x_2;\omega)\rangle_{R^*,R}.$ As a result, $h_2$ is a normal convex integrand. Let $\mathfrak{X}:=L^\infty(\Omega,X_1)\times L^\infty(\Omega,X_2)$.
The growth conditions from \Cref{assumption:general-problem} ensure that for any $r>0$, there exist $a_r \in L^p(\Omega)$ and positive constants $b_{r,e}$ and $b_{r,i}$ such that for all $\chi=(\chi_1, \chi_2) \in \mathfrak{X}$ satisfying $\lVert \chi \rVert_{\mathfrak{X}} \leq r$, we have 
\begin{align*}
|h_2(\chi_1(\omega), \chi_2(\omega);\omega)| &\leq |J_2(\chi_1(\omega), \chi_2(\omega);\omega)\bar{\vartheta}(\omega)|  +|\langle \bar{\lambda}_e(\omega), e(\chi_1(\omega), \chi_2(\omega);\omega)\rangle_{W^*,W}| \\
&\quad \quad +  |\langle \bar{\lambda}_i(\omega), i(\chi_1(\omega), \chi_2(\omega);\omega)\rangle_{R^*,R}| + |\cR^*[\bar{\vartheta}]| \\
& \leq |a_r(\omega)||\bar{\vartheta}(\omega)| + b_{r,e} \lVert \bar{\lambda}_e(\omega)\rVert_{W^*} + b_{r,i}  \lVert \bar{\lambda}_i(\omega)\rVert_{R^*}+ |\cR^*[\bar{\vartheta}]| =:k_r(\omega).
\end{align*}
It is straightforward to show that $k_r \in L^1(\Omega)$. Since $r$ is arbitrary, 
$h_2(\chi_1(\omega),\chi_2(\omega);\omega)$ is integrable for all $\chi=(\chi_1, \chi_2) \in \mathfrak{X}$. In particular, the functional
\begin{equation*}
\begin{aligned}
H_2(\chi) &= \int_{\Omega} J_2(\chi_1(\omega),\chi_2(\omega); \omega)\bar{\vartheta}(\omega) + \langle \bar{\lambda}_e(\omega), e(\chi_1(\omega),\chi_2(\omega);\omega)\rangle_{W^*,W} \D \pP(\omega) \\
&\quad \quad + \int_{\Omega} \langle \bar{\lambda}_i(\omega), i(\chi_1(\omega),\chi_2(\omega);\omega)\rangle_{R^*,R} \D \pP(\omega) -\cR^*[\bar{\vartheta}]
\end{aligned}
\end{equation*}
is well-defined and finite on $\mathfrak{X}$ by \Cref{lemma:duality-conjugate-functions}; it is continuous on $\mathfrak{X}$ since for an arbitrary $r>0$ it is always possible to bound $H_2$ for all $\chi$ satisfying $\lVert \chi\rVert_{\mathfrak{X}}\leq r$ due to
\begin{equation*}
H_2(\chi) = \int_\Omega h_2(\chi_1(\omega),\chi_2(\omega);\omega) \D \pP(\omega) \leq \int_\Omega k_r(\omega) \D \pP(\omega) < \infty.
\end{equation*}
Since the domain of $H_2$ is $\mathfrak{X}$, it follows by, e.g., \cite[Proposition 2.107]{Bonnans2013} that $H_2$ is continuous on $\mathfrak{X}.$
  
Let $\iota\colon X \rightarrow \mathfrak{X}$
be the continuous injection, which maps elements of $X_1$ to the
corresponding constant in $L^\infty(\Omega,X_1)$ and maps each element
of $L^\infty(\Omega,X_2)$ to itself. In particular, $\iota(\bx) = (\bar{\chi}_1(\cdot), \bar{\chi}_2(\cdot))$ with 
 $(\bar{\chi}_1(\omega),\bar{\chi}_2(\omega)) = (\bx_1,\bx_2(\omega))$ a.s., implying
\begin{equation}
\label{eq:definition-H2}
\begin{aligned}
H_2(\iota(\bar{x})) &=\int_{\Omega} J_2(\bar{x}_1,\bar{x}_2(\omega); \omega)\bar{\vartheta}(\omega)  + \langle \bar{\lambda}_e(\omega), e(\bx_1,\bx_2(\omega);\omega)\rangle_{W^*,W} \D \pP(\omega) \\
&\quad \quad + \int_{\Omega} \langle \bar{\lambda}_i(\omega), i(\bx_1,\bx_2(\omega);\omega)\rangle_{R^*,R} \D \pP(\omega) -\cR^*[\bar{\vartheta}]\\
&= \cR[J_2(\bx)] +  \E[\langle \lambda_e, e(\bx) \rangle_{W^*,W}+\langle \lambda_i, i(\bx) \rangle_{R^*,R}],
\end{aligned}
\end{equation}
where the last equality follows since $\bar{\vartheta} \in \partial \cR[J_2(\bx)]$.

Combining \eqref{eq:optimality-saddle-point-primal}, \eqref{eq:definition-H1} and \eqref{eq:definition-H2}, we have
\begin{equation*}
\inf_{x \in X_0} \bar{L}(x,\bz,\bz^\circ) =  \bar{L}(\bx,\bz,\bz^\circ) = H_1(\bx)+ H_2(\iota(\bx)).
\end{equation*}

Since $H_2$ is continuous on $\mathfrak{X}$ and $\bx \in \textup{dom}(H_1)$ (note $x \in X_0$ implies $x_1 \in C$),  
it is possible to invoke the Moreau--Rockafellar theorem (cf.~\cite[Theorem 2.168]{Bonnans2013})
to obtain
\begin{equation*}
0 \in \partial H_1(\bx) + \iota^* \partial H_2(\iota (\bx)).
\end{equation*}
In particular, there exists
$\bar{q} \in \mathfrak{X}^*= (L^\infty(\Omega,X_1) \times L^\infty(\Omega,X_2))^*$ such that 
\begin{equation*}
  -\iota^* \bar{q} \in \partial H_1(\bx) \quad \text{and} \quad \bar{q} \in \partial H_2(\iota(\bx)).
\end{equation*}

From \Cref{cor:characterization-subdifferentials}, it follows that $\partial H_2(\iota(\bx))$
consists of continuous linear functionals on
$L^\infty(\Omega,X_1) \times L^\infty(\Omega,X_2)$, which can be identified with
pairs $(\bar{q}_1,\bar{q}_2)$ from a weakly compact subset of $L^1(\Omega,X_1^*) \times L^1(\Omega,X_2^*)$ such that
\begin{equation}
\label{eq:subdifferential-h2}
\bar{q}(\omega) = (\bar{q}_1(\omega),\bar{q}_2(\omega)) \in \partial h_2(\bx_1, \bx_2(\omega);\omega) \quad \text{a.s.}
\end{equation}
Let $\iota_1\colon X_1 \rightarrow L^\infty(\Omega,X_1)$ be the continuous injection such that $\iota(x) = (\iota_1 x_1, x_2)$.
Notice that for $\bar{q}_1 \in L^1(\Omega,X_1^*)$, the adjoint
$\iota_1^*\colon (L^\infty(\Omega,X_1))^* \rightarrow X_1^*$ satisfies
\begin{align*}
\langle \iota_1^* \bar{q}_1,\u \rangle_{X_1^*,X_1} & =  \langle \bar{q}_1, \iota_1 \u \rangle_{L^1(\Omega,X_1^*),L^\infty(\Omega,X_1)} = \E[ \langle \bar{q}_1(\cdot), x_1\rangle_{X_1^*,X_1}] \quad \forall x_1 \in X_1. 
\end{align*}
Hence $\iota^* \bar{q} = (\E[\bar{q}_1], \bar{q}_2) \in X_1^* \times L^1(\Omega,X_2^*).$
Thus $-\iota^* \bar{q} \in \partial H_1(\bx)$ can be equivalently written as
\begin{equation*}
  H_1(x) \geq H_1(\bx) - \langle \E[\bar{q}_1], \u - \bx_1\rangle_{X_1^*,X_1} -\E[\langle \bar{q}_2, \y - \bx_2\rangle_{X_2^*,X_2}]
\end{equation*}
for all $x \in X$. We obtain from \eqref{eq:definition-H1}
\begin{equation}
\label{eq:inequality-KKT-proof-0}
J_1(\u) + \ell(x_1, \bar{\lambda}^\circ) \geq J_1(\bx_1) +\ell(\bx_1, \bar{\lambda}^\circ) - \langle \E[\bar{q}_1], \u - \bx_1 \rangle_{X_1^*,X_1} \quad \forall \u \in C
\end{equation}
and
\begin{equation}
\label{eq:inequality-KKT-proof-1}
\E[\langle \bar{q}_2, \y - \bx_2\rangle_{X_2^*,X_2}] \geq 0 \quad \forall \y \in L^\infty(\Omega,X_2).
\end{equation}
The expression~\eqref{eq:inequality-KKT-proof-0} is clearly equivalent to condition (i). Moreover, \eqref{eq:inequality-KKT-proof-1} implies $\bar{q}_2(\omega) = 0$ a.s. In combination with~\eqref{eq:subdifferential-h2}, this implies that for all
$(x_1,x_2) \in X_1 \times X_2$,
\begin{equation*}
  h_2(\u,\y;\omega)\geq h_2(\bx_1, \bx_2(\omega);\omega) + \langle \bar{q}_1(\omega), \u - \bx_1 \rangle_{X_1^*,X_1} \quad \text{a.s.}
\end{equation*}
From the definition of $h_2$, it follows that
\begin{equation}
\label{eq:inequality-KKT-proof-3}
\begin{aligned}
&J_2(\u,\y;\omega)\bar{\vartheta}(\omega) + \langle \bar{\lambda}_e(\omega), e(\u,\y;\omega)\rangle_{W^*,W} \\
&\quad + \langle \bar{\lambda}_i(\omega), i(\u,\y;\omega) \rangle_{R^*,R} - \langle \bar{q}_1(\omega), \u\rangle_{X_1^*,X_1} \\
&\quad\quad\geq J_2(\bx_1,\bx_2(\omega);\omega)\bar{\vartheta}(\omega) + \langle \bar{\lambda}_e(\omega),e(\bx_1,\bx_2(\omega);\omega) \rangle_{W^*,W} \\
&\quad\quad\quad+ \langle \bar{\lambda}_i(\omega), i(\bx_1, \bx_2(\omega);\omega)\rangle_{R^*,R} - \langle \bar{q}_1(\omega), \bx_1\rangle_{X_1^*,X_1}
\end{aligned}
\end{equation}
for all $(x_1,x_2) \in X_1 \times X_2.$
The inequality~\eqref{eq:inequality-KKT-proof-3} is clearly equivalent to
condition (ii) with $\bar{\rho}(\omega):=\bar{q}_1(\omega)$. 
\end{proof}

\begin{rem}
\label{rem:weak-compactness-subdifferential-J2}
Implicitly proven in \Cref{thm:extended-KKT-conditions} are some additional properties that will later be convenient. First, $f(x_1, x_2;\omega):=J_2(x_1,x_2;\omega)\bar{\vartheta}(\omega)$ is Carath\'eodory and convex with respect to $(x_1,x_2)$. The growth conditions ensure that $F(\chi) =\E[ f(\chi_1(\cdot),\chi_2(\cdot);\cdot)]$ is continuous with respect to $\chi \in \mathfrak{X}=L^\infty(\Omega,X_1)\times L^\infty(\Omega,X_2)$. Therefore, an element of the subdifferential $\partial F(\iota(\bx))$ can be identified with a pair $(\bar{\phi}_1,\bar{\phi}_2)$ belonging to a weakly compact subset in $L^1(\Omega,X_1^*)\times L^1(\Omega,X_2^*)$ such that
\begin{equation*}
\bar{\phi}(\omega) = (\bar{\phi}_1(\omega),\bar{\phi}_2(\omega)) \in \partial f(\bx_1,\bx_2(\omega);\omega) = \partial J_{2}(\bx_1,\bx_2(\omega);\omega) \bar{\vartheta}(\omega) \quad \text{a.s.} 
\end{equation*}
The equality of the subdifferentials follows since $\bar{\vartheta} \geq 0$ a.s.

Additionally, we show that $\bar{\zeta}$ satisfying $\bar{\zeta}(\omega) \in \partial J_{2}(\bx_1,\bx_2(\omega);\omega)$ a.s.~belongs to the set $L^p(\Omega,X_1^*) \times L^p(\Omega,X_2^*)$. By definition of the subdifferential, we have for all $\tilde{\chi} \in \mathfrak{X}$
\begin{equation*}
\langle \bar{\zeta}(\omega), \tilde{\chi}(\omega)-\bx(\omega)\rangle_{X_1^* \times X_2^*, X_1\times X_2} \leq J_2(\tilde{\chi}_1(\omega), \tilde{\chi}_2(\omega); \omega) - J_2(\bx_1, \bx_2(\omega);\omega) \quad \text{a.s.}
\end{equation*}
Hence with $\chi:=\tilde{\chi}-\bx$ we obtain $\langle \bar{\zeta}(\omega), \chi(\omega)\rangle_{X_1^* \times X_2^*, X_1\times X_2} \leq J_2(\chi_1(\omega)+\bx_1, \chi_2(\omega)+\bx_2(\omega); \omega)-J_2(\bx_1, \bx_2(\omega); \omega)$ a.s. For any $r>0$, there exists by \Cref{subasu:general-iii} a function $a_r \in L^p(\Omega)$ such that $|J_2(\chi_1(\omega),\chi_2(\omega);\omega)| \leq a_{r}(\omega)$ for $\chi$ satisfying $\lVert \chi \rVert_{\mathfrak{X}} < r$. With $r>0$ chosen large enough we get
\begin{equation*}
\lVert \bar{\zeta}(\omega) \rVert_{X_1^* \times X_2^*} = \sup_{\chi(\omega):\lVert \chi(\omega) \rVert_{X_1 \times X_2}\leq 1 } \langle \bar{\zeta}(\omega), \chi(\omega)\rangle_{X_1^* \times X_2^*, X_1\times X_2} \leq 2a_r(\omega),
\end{equation*}
so it follows that $\E[\lVert \bar{\zeta}_i \rVert^p_{X_i^*}]< \infty$ for $i=1,2$.
Since $\bar{\zeta}$ is measurable, it follows that $\bar{\zeta} \in L^p(\Omega,X_1^*) \times L^p(\Omega,X_2^*)$ as claimed.

\end{rem}

\subsection{Existence of regular multipliers}
\label{subsec:regular-mutlipliers}

In this section, we refine the results by \cite{Geiersbach2020+} to include objective functions with risk measures. The additional \Cref{assumption:relatively-complete-recourse} ensures that Lagrange multipliers are more regular.

\begin{lemma}
\label{lemma:saddle-points-lagrangian}
Let \Cref{assumption:general-problem}, \Cref{assumption:existence-saddle-points}, and \Cref{assumption:relatively-complete-recourse} be satisfied. Then there exists a saddle point $(\bar{x},\bar{\lambda}) \in X_0 \times \Lambda_0$ to the Lagrangian $L$, i.e., $(\bar{x},\bar{\lambda})$ satisfies
\begin{equation}
\label{eq:saddle-point-Lagrangian}
L(\bx, \lambda) \leq L(\bx,\bar{\lambda}) \leq L(x,\bar{\lambda}) \quad \forall (x,\lambda) \in X \times \Lambda.
\end{equation}
\end{lemma}

\begin{proof}
This is proven in \Cref{thm:minP-supD} and \Cref{thm:infP-maxD}.
\label{thm:infP-maxD}
\end{proof}

\begin{theorem}
\label{thm:KKT-basic-conditions}
Let \Cref{assumption:general-problem} be satisfied. Then $(\bar{x}, \bar{\lambda}) \in X_0 \times \Lambda_0$ satisfies \eqref{eq:saddle-point-Lagrangian} if and only if there exist
$\bar{\rho} \in L^1(\Omega,X_1^*)$ and $\bar{\vartheta} \in \partial \cR[J_2(\bx)]$ fulfilling
the following conditions:
\begin{enumerate}[label=(\roman*)]
\item The function
  \begin{equation*}
    x_1 \mapsto J_1(x_1) + \langle \E[\bar{\rho}], x_1 \rangle_{X_1^*,X_1}
  \end{equation*}
attains its
minimum over $C$ at $\bx_1$. \label{eq:FOC-basic-problem1}
\item The function 
\begin{align*}(x_1,x_2) \mapsto & \, J_2(x_1,x_2; \omega)\bar{\vartheta}(\omega) + \langle \bar{\lambda}_e(\omega), e(x_1,x_2;\omega)\rangle_{W^*,W} \\
\quad &+ \langle \bar{\lambda}_i(\omega),i(x_1,x_2;\omega) \rangle_{R^*,R} - \langle \bar{\rho}(\omega), x_1 \rangle_{X_1^*,X_1} 
\end{align*}
attains its minimum in $X_1 \times X_2$ at $(\bx_1,\bx_2(\omega))$ for almost
every $\omega \in \Omega$.  \label{eq:FOC-basic-problem2}
\item It holds that $\bx_1 \in C$ and the following conditions hold almost
  surely:
\begin{align*}
e(\bx_1, \bx_2(\omega);\omega)&=0,\\
\bar{\lambda}_i(\omega)\in K^\oplus,\quad i(\bx_1, \bx_2(\omega);\omega) \leq_K 0,& \quad \langle \bar{\lambda}_i(\omega), i(\bx_1, \bx_2(\omega);\omega) \rangle_{R^*,R} = 0.
\end{align*}
\label{eq:FOC-basic-problem3}
\end{enumerate}
\end{theorem}

\begin{proof}
The proof follows the arguments from \Cref{thm:extended-KKT-conditions} after setting $L^\circ \equiv 0$. 
\end{proof}

In the risk-neutral case $\cR = \E$, the subdifferential consists of the unique subgradient $\bar{\vartheta} \equiv 1.$ In that case, \Cref{thm:KKT-basic-conditions} recovers the optimality conditions presented in \cite{Geiersbach2020+}.

\section{Moreau--Yosida Regularization}
\label{sec:Moreau-Yosida}
Optimal control problems with state constraints were first handled in \cite{Casas1986} and the use of a Moreau--Yosida regularization for such problems was first studied in \cite{Hintermueller2006, Hintermueller2006a}. In this section, we detail the use of Moreau--Yosida regularization for the state constraint
 \begin{equation}
\label{eq:conical-constraint}
i(x_1,x_2(\omega);\omega) \leq_K 0 \quad \quad (\Leftrightarrow -i(x_1,x_2(\omega);\omega) \in K)
\end{equation}
from Problem~\eqref{eq:model-problem-abstract-long}, i.e., in the novel setting of stochastic optimization problem with almost sure constraints.

\subsection{Moreau--Yosida Revisited}
\label{subsection-MY-revisited}
In applications, it is often the case that $K$ coincides with a cone on a weaker space than $R$, the space where $K$ has a nonempty interior. Let $(H, (\cdot,\cdot)_H)$ be a Hilbert space.
We consider cases in which the image space $R$ of the conical constraint belongs to a Gelfand triple in the sense that
\begin{equation}
\label{eq:rigged-space}
R \hookrightarrow H \cong H^* \hookrightarrow R^*,
\end{equation}
where $\hookrightarrow$ denotes the dense and continuous embedding. In particular $\langle h, r \rangle_{R^*, R} = (h,r)_{H}$ whenever $r \in R$ and $h \in H$. We assume the cone $K$ is compatible with a cone $K_H$ on the less regular space $H$ in the sense that $K_H \cap R = K$. This construction is quite natural and allows the construction of a regularization with respect to the weaker $H$-norm, where computations are cheaper. Notably, since $i$ has an image in $R$, it follows that
\begin{equation}
\label{eq:equivalence-inequality-cones}
i(x_1, x_2(\omega);\omega) \leq_{K_H} 0 \quad \Leftrightarrow \quad i(x_1, x_2(\omega);\omega) \leq_{K} 0. 
\end{equation}
Examples of applications are given at the end of this section. 

Given a parameter $\gamma >0$, the Moreau--Yosida regularization (or envelope) for $\delta_{K_H}$ is given by the function $\beta^\gamma: H \rightarrow \R$ with
\begin{equation*}
\label{eq:MY-function}
\beta^\gamma(k) = \inf_{y \in H} \left\lbrace \delta_{K_H}(y) +\frac{\gamma}{2} \lVert k-y \rVert_H^2 \right\rbrace.
\end{equation*}

Clearly, $\beta^\gamma$ has the property that $\beta^\gamma(k)=0$ whenever $k \in {K_H}$, and otherwise $\beta^\gamma(k) >0$. Additionally,
\begin{equation}
\label{eq:explicit-formula-MY-term}
\beta^\gamma(k) =\frac{\gamma}{2}  \inf_{y \in H} \left\lbrace \delta_{K_H}(y) + \lVert k-y\rVert_{H}^2 \right\rbrace = \frac{\gamma}{2} \lVert k-\pi_{K_H}(k)\rVert_{H}^2,
\end{equation}
where $\pi_{K_H}\colon H \rightarrow {K_H}, k \mapsto \argmin_{y \in H} \{ \delta_{K_H}(y) + \lVert k-y\rVert_{H}\}$ denotes the projection onto $K_H$. Let $K_H^\oplus= \{ h \in H\colon (h,k)_H \geq 0 \, \forall k \in K_H \}$ denote the dual cone to $K$. The projection $\pi_{K_H}$ onto a nonempty, closed, convex cone $K_H$ admits the characterization (cf. \cite[Proposition 6.27]{Bauschke2011})
\begin{equation}
\label{eq:projection-inequality}
\begin{aligned}
\pi_{K_H}(x) = p \quad \Leftrightarrow \quad p \in K_H, (p,x-p)_H = 0, \text{ and } p-x \in K_H^\oplus.
\end{aligned}
\end{equation}
Moreover (cf.~\cite[Proposition 12.29]{Bauschke2011}), $\beta^\gamma$ is Fr\'echet differentiable on $H$ with the gradient $\nabla \beta^\gamma =\gamma (\textup{id}_H - \pi_{K_H})$.

The following result is proven in Appendix~\ref{sec:appendix-additional-proofs} and is needed in the following \Cref{lemma:properties-MY-regularization}.
\begin{lemma}
\label{lemma:concatenation-convex-functions}
Let $Y$ and $Z$ be real Banach spaces and $K \subset Z$ be a nonempty, closed, and convex cone. Suppose $f\colon Y \rightarrow Z$ is $K$-convex and $g: Z \rightarrow \R$ is convex having the property that $g(z) = 0$ if $z \in K$. Then $g \circ (-f)$ is convex.
\end{lemma}

\begin{lemma}
\label{lemma:properties-MY-regularization}
Under Assumption ~\ref{assumption:general-problem},  $x \mapsto \E[\beta^\gamma(-i(x))]$ is weakly* lower semicontinuous on $X$ for any $\gamma >0$. Moreover, we have $\E[\beta^\gamma(-i(x))] = 0$ if and only if $i(x_1,x_2(\omega);\omega) \leq_{K} 0$ a.s. 
\end{lemma}

\begin{proof}
First, we prove that $h(x_1,x_2;\omega):=\beta^\gamma(-i(x_1,x_2;\omega))$ is a normal convex integrand.
Since the indicator function $\delta_{K_H}$ is a proper, convex, and lower semicontinuous function on $H$, $\beta^\gamma$ is a convex, real-valued, and continuous function (cf.~\cite[Proposition 12.15]{Bauschke2011}). In particular, $\beta^\gamma$ is a normal integrand, so $\omega \mapsto \beta^\gamma(f(\omega))$ is measurable for any measurable mapping $f\colon \Omega\rightarrow \R$. Continuity of both $\beta^\gamma$ and $i$ imply continuity of $h$ on $X_1 \times X_2$; convexity of $h$ on $X_1 \times X_2$ follows by \Cref{lemma:concatenation-convex-functions}. As a result, $h$ is a normal convex integrand.

Let $r>0$ and suppose $(x_1, x_2) \in X_1 \times X_2$ satisfies $\lVert x_1 \rVert_{X_1} + \lVert x_2 \rVert_{X_2} \leq r$. Since $0 \in K = K_H \cap R$ and $\pi_{K_H}$ is nonexpansive, we have thanks to the characterization \eqref{eq:explicit-formula-MY-term}
\begin{align*}
h(x_1,x_2;\omega) &= \frac{\gamma}{2} \lVert -i(x_1,x_2;\omega)-\pi_{K_H}(-i(x_1,x_2;\omega))\rVert_H^2\\
&\leq \gamma (\lVert-i(x_1,x_2;\omega)- 0\rVert_H^2 + \lVert \pi_{K_H}(0) - \pi_{K_H}(-i(x_1,x_2;\omega)) \rVert_H^2)\\
&\leq 2 \gamma \lVert i(x_1,x_2;\omega)\rVert_H^2 \leq 2 \gamma c b_{r,i}^2, 
\end{align*}
where $b_{r,i}$ is defined as in \Cref{subasu:general-v} and $c>0$ depends on the embedding constant from $R \hookrightarrow H$. Hence $h(x_1,x_2(\omega);\omega)$ is integrable in $\omega$ for all $x$. Weak* lower semicontinuity of $I_h$ follows by \Cref{remark:weak-star-lsc-2}. The second statement is clear when considering \eqref{eq:equivalence-inequality-cones} in combination with $\beta^\gamma \geq 0$ and $\beta^\gamma(k) = 0$ if and only if $k\in K_H$. 
\end{proof}

In the introduction, we already gave one example \eqref{eq:model-PDE-UQ-state-constraints} of an application to PDE-constrained optimization under uncertainty.
Now, we will consider three representative examples of the almost sure constraint \eqref{eq:conical-constraint} from the same problem class. Additionally, we provide their regularization functions. 
\begin{example}[Mixed control/state constraints]
\label{example:state-constraints}
Suppose $D$ is a bounded subset of $\R$. Assume $X_2 = H^1(D)$, $X_1 = H=L^2(D)$, and $K _H= \{z \in H: z(s) \geq 0 \text{ a.e.~in } D \}$. For $a \in L^\infty(\Omega,X_2)$, the pointwise almost sure constraint
\begin{equation*}
x_2(s,\omega) \leq a(s,\omega)+\varepsilon x_1(s), \quad \text{a.e. } s\in D, \text{ a.s. } \omega \in \Omega
\end{equation*}
has the form \eqref{eq:conical-constraint} with $i(x_1(\cdot),x_2(\cdot,\omega);\omega) = x_2(\cdot,\omega) - a(\cdot,\omega) - \varepsilon x_1(\cdot)$.  In the case $\varepsilon = 0$, $R = H^1(D)$ and we have a stochastic version of state constraints originally investigated by \cite{Casas1986}. In the case $\varepsilon >0$, $R=L^2(D)$, and we have a stochastic version of a bottleneck constraint as investigated by \cite{Bergounioux1998}.

By \eqref{eq:explicit-formula-MY-term},
\begin{equation*}
\begin{aligned}
\beta^\gamma(-i(x_1(\cdot),x_2(\cdot,\omega);\omega)) & = \frac{\gamma}{2} \lVert \max(0,x_2(\cdot,\omega)-a(\cdot,\omega) - \varepsilon x_1(\cdot))\rVert_{L^2(D)}^2,
\end{aligned}
\end{equation*}
 where $\max$ is defined pointwise in $D$.
\end{example}

\begin{example}[Volume constraint on state]
Assume $X_2=H^1(D)$, $H=R=\R$, and $K_H = \R_+$. For $b \in L^\infty(\Omega)$, the constraint 
\begin{equation*}
 \int_D x_2(s,\omega) \D s \leq b(\omega), \quad \text{a.s. } \omega \in \Omega
\end{equation*}
takes the form $i(x_2(\cdot,\omega);\omega) = \int_D x_2(s,\omega) \D s - b(\omega)$ and the regularization function is given by
\begin{equation*}
\beta^\gamma(-i(x_2(\cdot,\omega);\omega)) = \frac{\gamma}{2}  \left\vert \max \left(0,\int_D x_2(s,\omega) \D s - b(\omega) \right) \right\vert^2.
\end{equation*}
\end{example}

\begin{example}[State gradient constraint]
Let $D$ be a bounded subset of $\R^d$. Assume $X_2 = W^{2,p}(D) \cap W_0^{1,p}(D)$ for  $p >2d$, $H=L^2(D)$, and $K_H$ is the same cone as in \Cref{example:state-constraints}. For $\psi \in L^\infty(\Omega,H)$, the almost sure constraint
\begin{equation*}
|\nabla x_2(s,\omega) |_2 \leq \psi(s,\omega) \quad \text{a.e. } s\in D, \text{ a.s. } \omega \in \Omega
\end{equation*}
takes the form $i(x_2(\cdot,\omega);\omega) = |\nabla x_2(\cdot,\omega) |_2 - \psi(\cdot,\omega)$. In this case, $R = W^{1,p}(D)$. The regularization is
\begin{equation*}
\beta^\gamma(-i(x_2(\cdot,\omega);\omega)) = \frac{\gamma}{2} \lVert \max(0,|\nabla x_2(\cdot,\omega)|_2 -\psi(\cdot,\omega)) \rVert_{L^2(D)}^2.
\end{equation*}
\end{example}

\subsection{Convergence of Regularized Problems}
We now discuss the Moreau--Yosida regularization of the conical constraint from Problem~\eqref{eq:model-problem-abstract-long}. To that end, we define the family of regularized problems, parametrized by $\gamma >0$, by
\begin{equation}
\label{eq:model-PDE-UQ-state-constraints-reformulation}
\tag{$\textup{P}^\gamma$}
 \begin{aligned}
  &\underset{x \in X}{\text{minimize}} \quad \left\lbrace j^\gamma(x):= j(x)+ \E\left[\beta^\gamma( -i(x) )\right]  \right\rbrace \\
  & \text{s.t.} \quad \left\{\begin{aligned}
x_1 &\in C,\\  
  e(x_1,x_2(\omega);\omega) &=0\quad \text{a.s.}
    \end{aligned}\right.
 \end{aligned}
\end{equation} 

Problem \eqref{eq:model-PDE-UQ-state-constraints-reformulation} is interesting in its own right. Due to the linearity of the constraint, in the case where $j$ is quadratic, it is related to a subproblem used in sequential quadratic programming. 

Thanks to \Cref{assumption:general-problem}, the functions $e(\cdot,\cdot;\omega)$ are Fr\'echet differentiable with partial derivatives that are bounded and linear, i.e.,~$e_{x_i}(\omega) \in \mathcal{L}(X_i,W).$ Without loss of generality, affine terms can be incorporated in the objective function, so by linearity, 
\begin{equation*}
e(x_1,x_2;\omega) = e_{x_1}(\omega) x_1 + e_{x_2}(\omega) x_2.
\end{equation*}

We focus on the case where Problem~\eqref{eq:model-problem-abstract-long} does not satisfy the relatively complete recourse condition. This occurs naturally when the equality constraint has a unique solution for every $x_1 \in C$. Then it is clear that the relatively complete recourse condition will not be satisfied except for trivial choices of $K$. 

For the next results, we need the following additional assumptions.

\begin{assumption} 
\label{assumption:example-problem-MY}  \setcounter{subassumption}{0}
\subasu \label{assumption:example-problem-MY-i} $e_{x_2}(\omega)$ is a linear isomorphism for almost every $\omega$ and $e_{x_2}^{-1} \in L^\infty(\Omega,\mathcal{L}(W,X_2))$.\\
\subasu \label{assumption:example-problem-MY-ii} $i$ is continuously Fr\'echet differentiable on $X_1 \times X_2$ for almost every $\omega$ with partial derivatives satisfying $i_{x_k}(x_1, x_2;\cdot) \in L^\infty(\Omega,\mathcal{L}(X_k,R))$ for every $(x_1,x_2) \in X_1\times X_2$ and $k=1,2$.\\
\subasu \label{assumption:example-problem-MY-iii} If $x^\gamma \rightharpoonup^* \bar{x}$ and $j(x^\gamma) \rightarrow j(\bx)$, then $\lVert x_1^\gamma \rVert_{X_1} \rightarrow \lVert \bx_1 \rVert_{X_1}$.
\end{assumption}

\begin{rem}
For the consistency of the primal problems, shown in \Cref{lemma:MY-basic-convergence}, we only need \Cref{assumption:example-problem-MY-i}. By \Cref{assumption:example-problem-MY-i}, for every $x_1 \in C$, there exists a unique solution $x_2 \in L^\infty(\Omega,X_2)$ such that 
\begin{equation}
\label{eq:solution-PDE-abstract}
x_2(\omega) = -e_{x_2}^{-1}(\omega) e_{x_1}(\omega) x_1.
\end{equation}
Thanks to the additional structure afforded by \Cref{assumption:example-problem-MY-i}, it is in fact enough to require that $j$ is radially unbounded with respect to $x_1$ only (instead of $x$). 
\Cref{assumption:example-problem-MY-iii} is a property that is satisfied in combination with \eqref{eq:solution-PDE-abstract}, e.g., if $J_1$ is a $\mu$-strongly convex function, as is the case for the Tichonov regularization $J_1(x_1) = \frac{\mu}{2} \lVert x_1 \rVert_{X_1}^2, \mu >0$. To see this, notice that by \eqref{eq:solution-PDE-abstract}, the control-to-state operator $S(\omega):X_1 \rightarrow X_2$, $x_1 \mapsto x_2(\omega)$ is a linear mapping for almost every $\omega$. The reduced functional $\hat{J}_2(x_1):=\cR[J_2(x_1, S(\cdot)x_1; \cdot)]$ is therefore convex. It is straightforward to prove that $\hat{j}(x_1^\gamma):=J_1(x_1^\gamma) + \hat{J}_2(x_1^\gamma) = j(x^\gamma)$ is a $\mu$-strongly convex function over $X_1$. Note that if $x^\gamma \rightharpoonup^* \bar{x}$, then $x_1^\gamma \rightharpoonup^* \bar{x}_1$ and thus the convex combination $z_t^\gamma:=t x_1^\gamma + (1-t) \bar{x}_1 \in C$ satisfies $z_t^\gamma \rightharpoonup^* \bar{x}_1$ for any $t \in [0,1]$. By weak* lower semicontinuity and convexity of $\hat{j}$ as well as $j(x^\gamma) \rightarrow j(\bar{x})$,
\begin{equation*}
\hat{j}(\bar{x}_1) \leq \liminf_{\gamma \rightarrow \infty} \hat{j}(z_t^\gamma) \leq \limsup_{\gamma \rightarrow \infty} \hat{j}(z_t^\gamma) \leq  \limsup_{\gamma \rightarrow \infty} t \hat{j}(x_1^\gamma) + (1-t) \hat{j}(\bar{x}_1) = \hat{j}(\bar{x}_1),
\end{equation*}
meaning $\lim_{\gamma \rightarrow \infty}  \hat{j}(t x_1^\gamma + (1-t) \bar{x}_1) = \hat{j}(\bar{x}_1)$ for any $t \in [0,1]$. Now, $\mu$-strong convexity of $\hat{j}$ gives 
\begin{equation*}
\begin{aligned}
&\lim_{\gamma \rightarrow \infty} \frac{\mu}{2} t(1-t) \lVert x_1^\gamma - \bar{x}_1 \rVert_{X_1}^2 + \lim_{\gamma \rightarrow \infty} \hat{j}(t x_1^\gamma + (1-t) \bar{x}_1) \leq \lim_{\gamma \rightarrow \infty} t \hat{j}(x_1^\gamma) + (1-t)\hat{j}(\bar{x}_1)\\
&\Leftrightarrow \quad \lim_{\gamma \rightarrow \infty} \frac{\mu}{2} t(1-t) \lVert x_1^\gamma - \bar{x}_1 \rVert_{X_1}^2 \leq 0,
\end{aligned}
\end{equation*}
which gives $\lVert x_1^\gamma \rVert_{X_1} \rightarrow \lVert \bx_1 \rVert_{X_1}$ as required.
\end{rem}

The first result concerns optimality of limit points of solutions to Problem~\eqref{eq:model-PDE-UQ-state-constraints-reformulation}. 
\begin{proposition}
\label{lemma:MY-basic-convergence}
Suppose Problem~\eqref{eq:model-problem-abstract-long} satisfies \Cref{assumption:general-problem}, \Cref{assumption:existence-saddle-points}, and \Cref{assumption:example-problem-MY-i}. Then there exists a solution $x^\gamma \in X$ to Problem~\eqref{eq:model-PDE-UQ-state-constraints-reformulation} for all $\gamma>0$. Additionally, given $\{ \gamma_n\}$ with $\gamma_n \rightarrow \infty$, weak* limit points of $\{ x^{\gamma_n}\}$ solve Problem~\eqref{eq:model-problem-abstract-long}.
\end{proposition}

\begin{proof}
Let $F_{\text{ad}}' := \{ x \in X: x_1 \in C, e(x_1, x_2(\omega);\omega) = 0 \text{ a.s.}\}$ denote the feasible set for Problem \eqref{eq:model-PDE-UQ-state-constraints-reformulation}. From \eqref{eq:solution-PDE-abstract}, we have
\begin{equation}
\label{eq:boundedness-x2}
\lVert x_2 \rVert_{L^\infty(\Omega,X_2)} \leq \lVert e_{x_2}^{-1} \rVert_{L^\infty(\Omega,\mathcal{L}(W,X_2))} \lVert e_{x_1} x_1 \rVert_{L^\infty(\Omega,W)} \leq c \lVert x_1 \rVert_{X_1}.
\end{equation}
Either boundedness of $C$ or radial unboundedness of $j$ imply boundedness of $j^\gamma$ over $F'_{\text{ad}}$.
Thus a minimizing sequence $\{x^n\} \subset F_{\text{ad}}'$ exists such that 
\begin{equation}
\label{eq:minimizing-sequence}
\lim_{n \rightarrow \infty} j^\gamma(x^n) = \bar{j}^\gamma:=\inf_{x \in F_{\text{ad}}'} j^\gamma(x).
\end{equation}
If $C$ is bounded, then \eqref{eq:boundedness-x2} implies $x^n$ is bounded in $X$ for all $n$. Otherwise, if $j$ is radially unbounded, then for every $c>0$ there exists $r >0$ such that for all $n$ satisfying $\lVert x^n \rVert_{X} >r$, we have $j^\gamma(x^n) \geq j(x^n)>c$. Choosing $c$ large enough produces a contradiction to \eqref{eq:minimizing-sequence}. Hence $x^n$ is bounded for all $n$, so there exists a subsequence $\{n_k\}$ such that $x^{n_k} \rightharpoonup^* x^\gamma$.  \Cref{assumption:general-problem} guarantees that \Cref{prop:weak-weak*-lsc} and \Cref{lemma:properties-MY-regularization} apply, from which we deduce weak* lower semicontinuity of $j^\gamma$. As a result, we obtain
\begin{equation*}
\label{eq:weak-star-lsc-minimizing-sequence}
\bar{j}^\gamma \leq j^\gamma(x^\gamma) \leq \lim_{k \rightarrow \infty} j^\gamma(x^{n_k}) = \bar{j}^\gamma. 
\end{equation*}
Notice that $C$ is closed and by \Cref{corollary:affine-maps-are-weak-star-closed}, $F'_{\textup{ad}}$ is weakly* closed. Hence $x^\gamma \in F'_{\textup{ad}}$ and therefore $x^\gamma$ is an optimum for Problem \eqref{eq:model-PDE-UQ-state-constraints-reformulation}.

Now, observe a sequence $\{ x^{\gamma_n}\}$ of optima to Problem \eqref{eq:model-PDE-UQ-state-constraints-reformulation} with $\gamma_n \rightarrow \infty$.  If $\bar{x}$ is an optimum of Problem~\eqref{eq:model-problem-abstract-long}, then by optimality,
\begin{equation}
\label{eq:MY-points-uniformly-bounded}
j(x^{\gamma_{n}}) \leq j^{\gamma_{n}}(x^{\gamma_{n}}) \leq j^{\gamma_{n}}(\bx) = j(\bx).
\end{equation}
In fact \eqref{eq:MY-points-uniformly-bounded} implies that $\{x^{\gamma_n}\}$ makes up a minimizing sequence for $j$. Again, boundedness of $C$ or radial unboundedness of $j$ imply boundedness of $\{x^{\gamma_n} \}$.
Therefore, there exists a subsequence $\{n_k\}$ such that $x^{\gamma_{n_k}} \rightharpoonup^* \hat{x}$.
Weak* lower semicontinuity of $j$ gives $j(\hat{x}) \leq j(\bar{x}).$ Now, it is enough to show feasibility of $\hat{x}$, since from feasibility one must conclude $j(\hat{x}) = j(\bar{x})$ by optimality of $\bar{x}$. 
Notice that $j$ is bounded over $F'_{\text{ad}}$, so from \eqref{eq:MY-points-uniformly-bounded}, we have 
\begin{equation*}
\E[\beta^{\gamma_{n_k}}(-i(x^{\gamma_{n_k}}))] = \frac{\gamma_{n_k}}{2} \E[\lVert -i(x^{\gamma_{n_k}})- \pi_{K_H}( -i(x^{\gamma_{n_k}}))\rVert_H^2] \leq c.
\end{equation*}
Using the weak* lower semicontinuity granted by \Cref{lemma:properties-MY-regularization}, we obtain the convergence $\E[\lVert -i(x^{\gamma_{n_k}})- \pi_{K_H}( -i(x^{\gamma_{n_k}}))\rVert_H^2] \rightarrow 0$ as $k \rightarrow \infty$, meaning $-i(\hat{x}_1,\hat{x}_2(\omega);\omega) \in K$ a.s.~by \Cref{lemma:properties-MY-regularization}.
Feasibility of $\hat{x}$ for Problem~\eqref{eq:model-problem-abstract-long} now follows by closedness of $C$ and \Cref{corollary:affine-maps-are-weak-star-closed}. In closing, we have $j(x^{\gamma_n}) \rightarrow j(\hat{x})= j(\bar{x})$ since
\begin{equation}
\label{eq:optimality-of-limit-point}
j(\hat{x}) \leq \liminf_{\gamma_n \rightarrow \infty} j(x^{\gamma_n}) \leq  \limsup_{\gamma_n \rightarrow \infty} j(x^{\gamma_n}) \leq j(\bar{x}) = j(\hat{x}).
\end{equation}
\end{proof}

In the following, we analyze the convergence of Lagrange multipliers as $\gamma \rightarrow \infty$. To that end, it will be helpful to first formulate more explicit optimality conditions for Problem~\eqref{eq:model-problem-abstract-long}, and after that the explicit conditions for Problem~\eqref{eq:model-PDE-UQ-state-constraints-reformulation}. In what follows, we use the abusive notation $e_{x_i}^*$ for the adjoint operator from $(L^\infty(\Omega,W))^*$ to $(L^\infty(\Omega,X_i))^*$ and $e^*_{x_i}(\cdot)$ for the adjoint operator from $L^1(\Omega,W^*)$ to $L^1(\Omega,X_i^*)$ as described in \Cref{lemma:regularity-operators}. Similarly, $i_{x_i}^*(\bx)$ is the shorthand for the operator from  $(L^\infty(\Omega,R))^*$ to $(L^\infty(\Omega,X_i))^*$ whereas $i_{x_i}^*(\bx_1, \bx_2(\cdot); \cdot)$ represents the operator from $L^1(\Omega,R^*)$ to $L^1(\Omega,X_i^*)$.
\begin{lemma}
\label{lemma:KKT-general}
Suppose \Cref{assumption:general-problem}, \Cref{assumption:existence-saddle-points}, and \Cref{assumption:example-problem-MY-i}--\ref{assumption:example-problem-MY-ii} are satisfied. Then there exists a saddle point $(\bar{x}, \bar{\lambda}, \bar{\lambda}^\circ)$ of the Lagrangian $\bar{L}$. Furthermore, the existence of the saddle point is equivalent to the following necessary and sufficient optimality conditions:
there exist $\bar{\rho} \in L^1(\Omega,X_1^*)$, $\bar{\vartheta} \in \partial \cR[J_2(\bx)],$
$\bar{\eta} \in \partial J_1(\bx_1)$, $\bar{\xi} \in N_{C}(\bx_1)$, and $\bar{\zeta} = (\bar{\zeta}_1,\bar{\zeta}_2) \in L^p(\Omega,X_1^*) \times L^p(\Omega,X_2^*)$ with $\bar{\zeta}(\omega) \in \partial J_{2}(\bx_1,\bx_2(\omega); \omega)$  such that
\begin{subequations}
\begin{align}
 \bar{\eta} +\E[\bar{\rho}] + e_{x_1}^{*} \bar{\lambda}_e^\circ + i_{x_1}^*(\bx) \bar{\lambda}_i^\circ +\bar{\xi} &= 0,\label{eq:FOC-MY-problem-1}\\
e_{x_2}^* \bar{\lambda}_e^\circ + i_{x_2}^*(\bx) \bar{\lambda}_i^\circ &=0,\label{eq:FOC-MY-problem-2}\\
\bar{\zeta}_1(\omega)\bar{\vartheta}(\omega) +e_{x_1}^*(\omega) \bar{\lambda}_e(\omega)+ i_{x_1}^*(\bx_1,\bx_2(\omega);\omega) \bar{\lambda}_i(\omega) -\bar{\rho}(\omega) &=0,\label{eq:FOC-MY-problem-3}\\
\bar{\zeta}_2(\omega)\bar{\vartheta}(\omega) + e_{x_2}^*(\omega) \bar{\lambda}_e(\omega) +i_{x_2}^*(\bx_1,\bx_2(\omega);\omega) \bar{\lambda}_i(\omega)&= 0,\label{eq:FOC-MY-problem-4}\\
\bx_1 \in C, \quad e(\bx_1, \bx_2(\omega);\omega) &=0,\label{eq:FOC-MY-problem-5}\\
 i(\bx_1,\bx_2(\omega);\omega)\leq_K 0, \quad \bar{\lambda}_i(\omega) \in K^{\oplus}, \quad  \langle \bar{\lambda}_i(\omega), i(\bx_1,\bx_2(\omega);\omega)\rangle_{R^*,R}&=0,\label{eq:FOC-MY-problem-6} \\
 \bar{\lambda}_i^\circ \in \mathcal{K}^{\oplus}, \quad \langle \bar{\lambda}_i^\circ, i(\bx)\rangle_{(L^\infty(\Omega,R))^*,L^\infty(\Omega,R)} &= 0, \label{eq:FOC-MY-problem-7}
\end{align}
\end{subequations}
with pointwise conditions holding for almost all $\omega \in \Omega$. 
\end{lemma}

\begin{proof}
By \Cref{lemma:saddle-points-lagrangian}, a saddle point $(\bx,\bar{\lambda}, \bar{\lambda}^\circ)$ exists, and so we can apply \Cref{thm:extended-KKT-conditions}. The conditions \eqref{eq:FOC-MY-problem-5}--\eqref{eq:FOC-MY-problem-7} are trivially satisfied.
Minimizing $J_1(x_1) + \langle \E[\bar{\rho}], x_1\rangle_{X_1^*,X_1} + \ell(x_1, \bar{\lambda}^\circ)$
in $x_1$ over $C$ is equivalent to minimizing
\begin{equation*}
\begin{aligned}
F(x) &:=J_1(x_1) + \langle \E[\bar{\rho}], x_1\rangle_{X_1^*,X_1} +\langle \bar{\lambda}_e^\circ,e(x) \rangle_{(L^\infty(\Omega,W))^*,L^\infty(\Omega,W)}\\
&\qquad +\langle \bar{\lambda}_i^\circ,i(x)\rangle_{(L^\infty(\Omega,R))^*,L^\infty(\Omega,R)} + \delta_C(x_1)
\end{aligned}
\end{equation*}
over $X$. In combination with \eqref{eq:FOC-MY-problem-5}--\eqref{eq:FOC-MY-problem-6}, the condition $\ell(\bx_1,\bar{\lambda}^\circ) =0$ implies the minimum in the definition of \eqref{eq:definition-singular-functional} can only be attained at $\bx_2$ by Remark~\ref{rem:interpretation-ell-condition}. Hence condition (i) of \Cref{thm:extended-KKT-conditions} is the same as $0 \in \partial F(\bx)$. All terms in $F$ other than $\delta_C$ are continuous on $X$.  Thus the sum rule for the subdifferential applies (cf.~\cite[Theorem 2.168]{Bonnans2013}) and we have after exploiting differentiability afforded by \Cref{assumption:example-problem-MY-ii} the expression
\begin{equation*}
0 \in \begin{pmatrix}
\partial J_1(\bx_1) \\ \{ 0\}
\end{pmatrix}  + \begin{pmatrix}
\{\E[\bar{\rho}] + e_{x_1}^* \bar{\lambda}_e^\circ +  i_{x_1}^*(\bx) \bar{\lambda}_i^\circ\}\\ \{e_{x_2}^* \bar{\lambda}_e^\circ + i_{x_2}^*(\bx) \bar{\lambda}_i^\circ\}
\end{pmatrix} + \begin{pmatrix}
 \partial \delta_C(\bx_1)\\
 \{ 0\}
\end{pmatrix}.
\end{equation*}
Note that the corresponding adjoint operators are well-defined by \Cref{lemma:regularity-operators}. 
The existence of $\bar{\eta}$ and $\bar{\xi}$ as well as \eqref{eq:FOC-MY-problem-1}--\eqref{eq:FOC-MY-problem-2} follow.
Condition (ii) of \Cref{thm:extended-KKT-conditions} yields 
\begin{equation*}
\begin{aligned}
0 &\in \partial (J_{2}(\bx_1,\bx_2(\omega);\omega) \bar{\vartheta}(\omega))  + \begin{pmatrix}
\{e^*_{x_1}(\omega) \bar{\lambda}_e(\omega) + i^*_{x_1}(\bx_1,\bx_2(\omega);\omega) \bar{\lambda}_i(\omega) - \bar{\rho}(\omega)\}\\
\{e^*_{x_2}(\omega) \bar{\lambda}_e(\omega) + i^*_{x_2}(\bx_1,\bx_2(\omega);\omega) \bar{\lambda}_i(\omega) \}
\end{pmatrix}.
\end{aligned}
\end{equation*}
Since $\bar{\vartheta} \geq 0$ a.s., $\partial (J_{2}(\bx_1,\bx_2(\omega);\omega )\bar{\vartheta}(\omega)) =  \partial J_{2}(\bx_1,\bx_2(\omega);\omega) \bar{\vartheta}(\omega)$.  Hence there exists $\bar{\zeta}(\omega) \in \partial J_{2}(\bx_1,\bx_2(\omega);\omega)$ such that \eqref{eq:FOC-MY-problem-3}--\eqref{eq:FOC-MY-problem-4} are satisfied a.s. The fact that $\bar{\zeta} \in L^p(\Omega,X_1^*) \times L^p(\Omega,X_2^*)$ follows from \Cref{rem:weak-compactness-subdifferential-J2}.
\end{proof}

Now, we formulate the optimality conditions for Problem~\eqref{eq:model-PDE-UQ-state-constraints-reformulation}. We define the corresponding Lagrangian by $L^\gamma(x,\lambda):=j^\gamma(x) + \E[\langle \lambda_e, e(x)\rangle_{W^*,W}]$.
\begin{lemma}
\label{lemma:optimality-regular-problem}
Suppose \Cref{assumption:general-problem}, \Cref{assumption:existence-saddle-points}, and \Cref{assumption:example-problem-MY-i}--\ref{assumption:example-problem-MY-ii} are satisfied. Then there exists a saddle point $(x^\gamma, \lambda^\gamma)$ of the Lagrangian $L^\gamma$. Furthermore, the existence of the saddle point is equivalent to the following necessary and sufficient optimality conditions: there exist
$\rho^\gamma \in L^1(\Omega,X_1^*)$,  $\vartheta^\gamma \in \partial \cR[J_2(x^\gamma)],$
$\eta^\gamma \in  \partial J_1(x_1^\gamma)$, $\xi^\gamma \in N_{C}(x_1^\gamma)$ and $\zeta^\gamma = (\zeta_1^\gamma,\zeta_2^\gamma) \in L^p(\Omega,X_1^*)\times L^p(\Omega,X_2^*)$  with $\zeta^\gamma(\omega) \in\partial J_{2}(x_1^\gamma,x_2^\gamma(\omega);\omega)$ such that
\begin{subequations}
\begin{align}
\eta^\gamma +\E[\rho^\gamma] +  \xi^\gamma &=0, \label{eq:FOC-MY-problem-gamma-1}\\
\zeta_1^\gamma(\omega)\vartheta^\gamma(\omega)+ e_{x_1}^*(\omega)  \lambda_e^\gamma(\omega)+i_{x_1}^*(x_1^\gamma,x_2^\gamma(\omega);\omega) \lambda_i^\gamma(\omega)- \rho^\gamma(\omega) &= 0, \label{eq:FOC-MY-problem-gamma-2}  \\
\zeta_2^\gamma(\omega)\vartheta^\gamma(\omega)+ e_{x_2}^*(\omega) \lambda_e^\gamma(\omega) +i_{x_2}^*(x_1^\gamma,x_2^\gamma(\omega);\omega) \lambda_i^\gamma(\omega)  &= 0, \label{eq:FOC-MY-problem-gamma-3}\\
 x_1^\gamma \in C, \quad e(x_1^\gamma,x_2^\gamma(\omega);\omega)&=0, \label{eq:FOC-MY-problem-gamma-4}\\
\lambda_i^\gamma (\omega) =\gamma \big(i(x_1^\gamma,x_2^\gamma(\omega);\omega)+\pi_{K_H}(-i(x_1^\gamma,x_2^\gamma(\omega);\omega)))\quad &\text{in } H, \label{eq:FOC-MY-problem-gamma-5}
\end{align}
\end{subequations}
with pointwise conditions holding for almost all $\omega \in \Omega$.  
\end{lemma}

\begin{proof}
By \Cref{assumption:example-problem-MY-i}, we have that the set $C$ is contained in the induced feasible set for $x_1$ given by
$\tilde{C} = \{ x_1 \in X_1: \exists x_2 \in L^\infty(\Omega,X_2) \text{ s.t. } e(x_1,x_2(\omega);\omega)=0 \text{ a.s.} \}$. Hence \Cref{assumption:relatively-complete-recourse} is satisfied for Problem~\eqref{eq:model-PDE-UQ-state-constraints-reformulation} and
\Cref{lemma:saddle-points-lagrangian} can be applied to the Lagrangian $L^\gamma$, giving the existence of a saddle point $(x^\gamma,\lambda^\gamma)$. Now, \Cref{thm:KKT-basic-conditions} can be applied.

Condition (i) of \Cref{thm:KKT-basic-conditions} clearly translates to \eqref{eq:FOC-MY-problem-gamma-1}. As discussed in Subsection~\ref{subsection-MY-revisited}, $\beta^\gamma$ is Fr\'echet differentiable on $H$ with $\nabla \beta^\gamma(k) = \gamma(k - \pi_{K_H}(k)) \in H \hookrightarrow R^*$. By the chain rule, 
\begin{equation*}
\partial_{x_i} \beta^\gamma(-i(x_1,x_2;\omega)) = -i^*_{x_i}(x_1,x_2;\omega) \nabla\beta^\gamma(-i(x_1,x_2;\omega)).
\end{equation*}
Defining $\lambda_i^\gamma(\omega) = -\nabla \beta^\gamma(-i(x_1,x_2;\omega))$, we get the expression \eqref{eq:FOC-MY-problem-gamma-5}. Condition (ii) of \Cref{thm:KKT-basic-conditions} amounts to
\begin{equation*}
\begin{aligned}
f(x_1,x_2;\omega) &:= J_2(x_1,x_2;\omega)\vartheta^\gamma(\omega) + \beta^\gamma(-i(x_1,x_2(\omega);\omega)) \\
&\quad\quad +\langle \lambda_e^\gamma(\omega), e(x_1,x_2;\omega)\rangle_{W^*,W}-\langle \rho^\gamma(\omega),x_1\rangle_{X_1^*,X_1}
\end{aligned}
\end{equation*}
attaining its minimum at $(x_1^\gamma,x_2^\gamma(\omega))$, i.e., $0 \in \partial f(x_1^\gamma, x_2^\gamma(\omega);\omega)$, which yields \eqref{eq:FOC-MY-problem-gamma-2}--\eqref{eq:FOC-MY-problem-gamma-3}. The remaining conditions follow directly from \Cref{thm:KKT-basic-conditions}.
\end{proof}

The following characterization of the multiplier $\lambda_i^\gamma$ will be useful later.
\begin{lemma}
\label{lemma:orthogonality-lambda-i-gamma}
For all $k \in K_H$, we have
\begin{equation}
\label{eq:MY-Lagrange-multiplier-inequality}
(\lambda_i^\gamma(\omega),k)_H \geq 0
\end{equation}
and
\begin{equation*}
\label{eq:bound-feasible-point-MY-point-complementarity}
(\lambda_i^\gamma(\omega), i(x_1^\gamma,x_2^\gamma(\omega);\omega))_H \geq - (\lambda_i^\gamma(\omega),k)_H
\end{equation*}
\end{lemma}
\begin{proof}
Thanks to \eqref{eq:projection-inequality} and \eqref{eq:FOC-MY-problem-gamma-5}, we have $\lambda_i^\gamma(\omega) \in K_H^\oplus$ a.s.~and thus the expression \eqref{eq:MY-Lagrange-multiplier-inequality} holds. Additionally, \eqref{eq:projection-inequality} implies $(\lambda_i^\gamma(\omega),\pi_{K_H}(-i(x_1^\gamma,x_2^\gamma(\omega);\omega))_H=0$, so
\begin{equation}
\label{eq:partial-proof-bound-feasible-point-MY-point-complementarity}
(\lambda_i^\gamma(\omega),k)_H \geq (\lambda_i^\gamma(\omega),\pi_{K_H}(-i(x_1^\gamma,x_2^\gamma(\omega);\omega))_H.
\end{equation}
Adding $(\lambda_i^\gamma(\omega), i(x_1^\gamma,x_2^\gamma(\omega);\omega))_H$ to both sides of \eqref{eq:partial-proof-bound-feasible-point-MY-point-complementarity} gives
\begin{equation*}
(\lambda_i^\gamma(\omega),i(x_1^\gamma,x_2^\gamma(\omega);\omega)+k)_H \geq 0.
\end{equation*}
\end{proof}

Let $\mathcal{X}:=X\times \Lambda \times L^1(\Omega,X_1^*) \times X_1^* \times X_1^* \times (L^1(\Omega, X_1^*) \times L^1(\Omega, X_2^*)).$  In the following, we use the shorthand ${\phi}^\gamma:=(\phi_1^\gamma, \phi_2^\gamma)=(\zeta_1^\gamma \vartheta^\gamma, \zeta_2^\gamma \vartheta^\gamma)$.
We define the primal-dual path 
\begin{equation*}
\mathcal{C}_r := \{ (x^\gamma, \lambda^{\gamma}, \rho^{\gamma}, \eta^\gamma, \xi^\gamma, \phi^\gamma) \in \mathcal{X}\colon \gamma \in [r,\infty)\}.
\end{equation*}

\begin{proposition}
\label{prop:primal-dual-path-bounded}
Let \Cref{assumption:general-problem}, \Cref{assumption:existence-saddle-points}, and \Cref{assumption:example-problem-MY-i}--\ref{assumption:example-problem-MY-ii} hold. Then, for every $r\geq 0$, the path $\mathcal{C}_r$ is bounded in $\mathcal{X}$. 
\end{proposition}

\begin{proof}
\textit{Bounding of $x^\gamma$, $\eta^\gamma$, $\phi^\gamma$.}
Making the same arguments as in \Cref{lemma:MY-basic-convergence}, $x^\gamma$ is bounded independently of $\gamma$, from which boundedness of $\eta^\gamma$ immediately follows.  Notice that $\vartheta^\gamma \in \partial \cR[J_2(x^\gamma)]$, implying boundedness of $\vartheta^\gamma$. (This follows from the fact that $\cR$ is finite, convex, and continuous on $L^p(\Omega)$, making it locally Lipschitz on $L^p(\Omega)$; see \cite[Proposition 2.107]{Bonnans2013}. It follows that the subdifferential maps bounded sets to bounded sets by \Cref{lem:subdifferential-bounded-to-bounded}.)
Using the relations \eqref{eq:biconjugate-risk-functional} and \eqref{eq:subdifferentiability-expression}, since $\E[J_2(x^\gamma)\vartheta^\gamma] = \cR[J_2(x^\gamma)]+\cR^*[\vartheta^\gamma],$ and by \Cref{rem:weak-compactness-subdifferential-J2}, $\phi^\gamma$ (after identification) belongs to $\partial \E[ J_2(x^\gamma) \vartheta^\gamma]$, boundedness of $x^\gamma$ and $\vartheta^\gamma$ implies boundedness of $\phi^\gamma$.

\textit{Bounding of $\lambda_i^{\gamma}$.}
In the following, statements written as depending on $\omega$ are to be understood as holding almost surely. From \eqref{eq:constraint-qualification}, there exists $r>0$ such that for all $u=(u_e,u_i)\in B_r(0)=B_{r,e}(0)\times B_{r,i}(0) \subset L^\infty(\Omega,W) \times L^\infty(\Omega,R)$, there exist $k \in K$ and $x^0 = (x_1^0,x_2^0(\cdot))$ satisfying $x_1^0 \in C$ and
\begin{align}
e_{x_2}(\omega) x_2^0(\omega) + e_{x_1}(\omega) x_1^0 - u_e(\omega) &=0\ \nonumber\\ 
-i(x_1^0,x_2^0(\omega);\omega) - u_i(\omega) &=k. \label{eq:expression-conical-constraint-derivative}
\end{align}

Now, \eqref{eq:expression-conical-constraint-derivative} combined with
\Cref{lemma:orthogonality-lambda-i-gamma} gives
\begin{equation*}
\label{eq:bound-Lagrange-multiplier-point}
\begin{aligned}
( \lambda_i^\gamma(\omega),u_i(\omega))_{H}&=  (\lambda_i^\gamma(\omega),i(x_1^0,x_2^0(\omega);\omega) +u_i(\omega) )_H - (\lambda_i^\gamma(\omega),i(x_1^0,x_2^0(\omega);\omega) )_H  \\
&=  -(\lambda_i^\gamma(\omega),k)_{H} - (\lambda_i^\gamma(\omega),i(x_1^0,x_2^0(\omega);\omega) )_H  \\
&\leq (\lambda_i^\gamma(\omega),i(x_1^\gamma,x_2^\gamma(\omega);\omega) )_H - (\lambda_i^\gamma(\omega),i(x_1^0,x_2^0(\omega);\omega) )_H  \\
&\leq (\lambda_i^\gamma(\omega),i_{x_1}(x_1^\gamma,x_2^\gamma(\omega);\omega)(x_1^\gamma-x_1^0))_H \\
&\quad + (\lambda_i^\gamma(\omega),i_{x_2}(x_1^\gamma,x_2^\gamma(\omega);\omega)(x_2^\gamma(\omega)-x_2^0(\omega)) )_H,
 \end{aligned}
\end{equation*}
where we used convexity of $(x_1,x_2) \mapsto (\lambda_i^\gamma, i(x_1,x_2,\omega))_H$ in the second inequality.
Recalling \eqref{eq:rigged-space} and using  \eqref{eq:FOC-MY-problem-gamma-2}--\eqref{eq:FOC-MY-problem-gamma-4}, we obtain
\begin{equation*}
\label{eq:bound-Lagrange-multiplier-point-2}
\begin{aligned}
\langle  \lambda_i^\gamma(\omega),u_i(\omega)\rangle_{R^*,R} &\leq \langle \rho^\gamma(\omega)-\phi_1^\gamma(\omega), x_1^\gamma-x_1^0\rangle_{X_1^*,X_1} - \langle\phi_2^\gamma(\omega), x_2^\gamma(\omega)-x_2^0(\omega)\rangle_{X_2^*,X_2}. 
\end{aligned}
\end{equation*}
Thus 
\begin{equation}
\label{eq:bound-Lagrange-multiplier-point-3}
\begin{aligned}
\langle  \lambda_i^\gamma,u_i \rangle_{L^1(\Omega,R^*),L^\infty(\Omega,R)}& \leq \langle \E[\rho^\gamma -\phi_1^\gamma], x_1^\gamma-x_1^0\rangle_{X_1^*,X_1} - \E[\langle \phi_2^\gamma, x_2^\gamma-x_2^0\rangle_{X_2^*,X_2}]\\
& \leq \langle -\xi^\gamma-\eta^\gamma-\E[\phi_1^\gamma], x_1^\gamma-x_1^0\rangle_{X_1^*,X_1} - \E[\langle \phi_2^\gamma, x_2^\gamma-x_2^0\rangle_{X_2^*,X_2}]\\
& \leq \langle -\eta^\gamma-\E[\phi_1^\gamma], x_1^\gamma-x_1^0\rangle_{X_1^*,X_1} - \E[\langle \phi_2^\gamma, x_2^\gamma-x_2^0 \rangle_{X_2^*,X_2}]\\
& \leq (\lVert \eta^\gamma \rVert_{X_1^*} + \lVert \phi_1^\gamma \rVert_{L^1(\Omega,X_1^*)}) \lVert x_1^\gamma - x_1^0\rVert_{X_1}  \\
&\quad+ \lVert \phi_2^\gamma \rVert_{L^1(\Omega, X_2^*)} \lVert x_2^\gamma - x_2^0 \rVert_{L^\infty(\Omega,X_2)},
\end{aligned}
\end{equation}
where in the third inequality we used $\xi^\gamma \in N_C(x_1^\gamma)$, and in the fourth inequality we used Jensen's and H\"older's inequalities.

In the following, we use a generic constant $c>0$, which is independent of $\omega$ and takes potentially different values at each use.
Since $ x^\gamma$, $\eta^\gamma$, and $\phi^\gamma$ are bounded independently of $\gamma$,  \eqref{eq:bound-Lagrange-multiplier-point-3} yields
\begin{equation*}
\langle  \lambda_i^{\gamma},u_i\rangle_{L^1(\Omega,R^*),L^\infty(\Omega,R)}  \leq c < \infty
\end{equation*}
and hence 
\begin{equation*}
\lVert \lambda_i^{\gamma}\rVert_{L^1(\Omega,R^*)} \leq \frac{1}{r} \sup_{y \in B_{r,i}(0) } \{  \langle \lambda_i^{\gamma}, y \rangle_{L^1(\Omega,R^*),L^\infty(\Omega,R)} \} \leq \frac{1}{r} c < \infty.
\end{equation*}

\textit{Bounding of $\lambda_e^{\gamma}$, $\rho^{\gamma}$, $\xi^\gamma$.} 
Now, \eqref{eq:FOC-MY-problem-gamma-3} and \Cref{assumption:example-problem-MY-i}--\ref{assumption:example-problem-MY-ii} yield
\begin{equation*}
\label{eq:bound-lambda_e}
\begin{aligned}
 \lVert  \lambda_e^{\gamma}\rVert_{L^1(\Omega,W^*)} &=  \lVert -e_{x_2}^{-*} (i^*_{x_2}(x^\gamma)\lambda_i^\gamma + \phi_2^\gamma)\rVert_{L^1(\Omega,W^*)}\\
& \leq c \lVert i^*_{x_2}(x^\gamma)\lambda_i^\gamma + \phi_2^\gamma \rVert_{L^1(\Omega,X_2^*)}\\
& \leq c ( \lVert \lambda_i^\gamma\rVert _{L^1(\Omega,R^*)} + \lVert \phi_2^\gamma \rVert_{L^1(\Omega,X_2^*)}) \leq c < \infty.
\end{aligned}
\end{equation*}
Additionally,  \eqref{eq:FOC-MY-problem-gamma-2} gives
\begin{equation*}
\label{eq:bound-lambda_e}
\begin{aligned}
\lVert \rho^{\gamma}\rVert_{L^1(\Omega,X_1^*)}& =  \lVert \phi_1^\gamma + e_{x_1}^{*} \lambda_e^\gamma +i_{x_1}^*(x^\gamma) \lambda_i^\gamma \rVert_{L^1(\Omega,X_1^*)}\\
&  \leq c(\lVert \phi_1^\gamma \rVert_{L^1(\Omega, X_1^*)} + \lVert \lambda_e^{\gamma}\rVert_{L^1(\Omega,W^*)} +\lVert \lambda_i^{\gamma}\rVert_{L^1(\Omega,R^*)}) \leq c < \infty.
\end{aligned}
\end{equation*}
Finally, \eqref{eq:FOC-MY-problem-gamma-1} combined with Jensen's inequality yields
\begin{equation*}
\lVert \xi^\gamma \rVert_{X_1^*} \leq \lVert \eta^\gamma \rVert_{X_1^*} + \E[ \lVert \rho^\gamma \rVert_{X_1^*}] \leq c < \infty.
\end{equation*}
\end{proof}

We now show optimality of accumulation points of the primal-dual path $\mathcal{C}_0$. Here, a central difficulty is that the Banach--Alaoglu theorem cannot be applied to obtain weak* compactness of closed bounded subsets of $L^1(\Omega,X^*)$  ($X$ being a generic reflexive Banach space), as it would require understanding $L^1(\Omega,X^*)$ \textit{as a dual} of another Banach space. 
Another difficulty is that weakly* compact subsets of $L^1(\Omega,X^*)$ are not weakly* \textit{sequentially} compact. 

We will use two tools to circumvent these difficulties: first, we will identify $L^1(\Omega,X^*)$ with a closed subspace of the bidual space $(L^1(\Omega,X^*))^{**}=(L^\infty(\Omega,X))^*$. This is always possible, and in our setting, this corresponds to the relationship between absolutely continuous functionals and their integrable functionals via the expression \eqref{eq:absolutely-continuous-functionals}. The absolutely continuous functional corresponding to $\lambda^\gamma$ is denoted by $\lambda^{\gamma,a}$; the other functionals are defined similarly. 

To circumvent the second difficulty, we will work with a \textit{net} (generalized sequence) in $\mathcal{C}_0$ instead of a sequence. Given a directed set $I$ on which a relation $\preceq$ is defined, which is reflexive, transitive, and directed, a net $(x^\alpha)_{\alpha \in I}$ on $X$ is a map $x:I \rightarrow X$. A net $(y^{\beta})_{\beta \in J}$ is a subnet of $(x^\alpha)_{\alpha \in I}$ via $\varphi:J\rightarrow I$ if for all $\beta \in J$, $y^{\beta} = x^{\varphi(\beta)}$ and for every $\alpha_0 \in I$, there exists $\beta_0 \in J$ for which $\beta_0 \preceq \beta$ implies $\alpha_0 \preceq \varphi(\beta)$. We remind that $x^\alpha$ converges to an accumulation point $x$ if $(x^\alpha)_{\alpha \in I}$ lies eventually in every neighborhood of $x$, i.e., for any neighborhood $V$ of $x$, there exists $\alpha_0 \in I$ such that for all $\alpha \succeq \alpha_0$, we have $x^\alpha \in V$. For our purposes, the choice $(I,\preceq) = ([0,\infty),\leq)$ will be sufficient. Notably, topological notions of closure and continuity carry over for nets. We will write $\rightarrow$, $\rightharpoonup$, $\rightharpoonup^*$ to emphasize topologies when taking limits.
Nets on compact sets in Hausdorff spaces have the following convenient property: there exists a subnet that converges to a point in the set; see, e.g., \cite[Section 1.4, 1.7, 1.9]{Bauschke2011} for definitions and properties.

\begin{theorem}
\label{thm:optimality-of-limit-primal-dual}
Let \Cref{assumption:general-problem}, \Cref{assumption:existence-saddle-points}, and \Cref{assumption:example-problem-MY} hold, and observe the net generated by the primal-dual path $(x^{\gamma'}, \lambda^{\gamma'}, \rho^{\gamma'}, \eta^{\gamma'}, \xi^{\gamma'}, \phi^{\gamma'})_{\gamma' \in [0,\infty)} \subset \mathcal{C}_0$, with $\phi^{\gamma'} = \zeta^{\gamma'}\vartheta^{\gamma'}$. Then there exists for $J \subset [0,\infty)$ a subnet $(x^\gamma, \lambda^{\gamma,a}, \rho^{\gamma,a}, \eta^{\gamma}, \xi^\gamma, \phi^{\gamma,a})_{\gamma \in J}$, a sequence ${\gamma_n} \subset J$ with $\gamma_n \rightarrow \infty$, and an accumulation point $(\hat{x}, \lambda', \rho', \hat{\eta}, \hat{\xi}, \phi')$, $\hat{\zeta} \in L^p(\Omega,X_1^*) \times L^p(\Omega,X_2^*)$, and $\hat{\vartheta} \in L^{p'}(\Omega)$ such that
\begin{enumerate}[itemsep=0mm]
    \item $x^{\gamma_n} \rightarrow \hat{x}$ in $X$,
    \item $ \lambda_e^{\gamma,a} \rightharpoonup^* {\lambda}'_e$ in $(L^\infty(\Omega,W))^*$,
    \item $\lambda_i^{\gamma,a}  \rightharpoonup^* {\lambda}'_i$ in $(L^\infty(\Omega,R))^*$,
    \item $\rho^{\gamma, a}  \rightharpoonup^* {\rho}'$ in $(L^\infty(\Omega,X_1))^*$,
    \item $\eta^{\gamma_n} \rightharpoonup \hat{\eta}$ in $X_1^*$,
    \item $\xi^{\gamma_n} \rightharpoonup \hat{\xi}$ in $X_1^*$,
       \item $\phi^{\gamma,a} \rightharpoonup^* \phi'$ in $(L^\infty(\Omega,X_1))^*\times (L^\infty(\Omega,X_2))^*$ with  $\vartheta^{\gamma_n} \rightharpoonup^* \hat{\vartheta}$ in $L^{p'}(\Omega)$ and $\zeta^{\gamma_n}(\omega) \rightharpoonup \hat{\zeta}(\omega)$ in $X_1^* \times X_2^*$ a.s.
\end{enumerate}
We have ${\lambda}'_e=\hat{\lambda}_e^a + \hat{\lambda}_e^\circ$,  ${\lambda}'_i=\hat{\lambda}_i^a +\hat{\lambda}_i^\circ$,  ${\rho}'=\hat{\rho}^a+\hat{\rho}^\circ$, and $\phi' = \hat{\phi}^a + \hat{\phi}^\circ$ for the absolutely continuous functionals $\hat{\lambda}_e^a \in (L^\infty(\Omega,W))^*$, $\hat{\lambda}_i^a \in (L^\infty(\Omega,R))^*$ $\hat{\rho}^a \in (L^\infty(\Omega,X_1))^*$, and $\hat{\phi}^a\in (L^\infty(\Omega,X_1))^*\times (L^\infty(\Omega,X_2))^*$,  corresponding to $\hat{\lambda}_e \in L^1(\Omega,W^*)$,  $\hat{\lambda}_i \in L^1(\Omega,R^*)$, $\hat{\rho} \in L^1(\Omega,X_1^*)$, and $\hat{\phi} \in (L^1(\Omega,X_1^*)\times L^1(\Omega,X_2^*)),$ respectively.  Finally, the accumulation point $(\hat{x},\hat{\lambda},\hat{\lambda}^\circ,\hat{\rho},\hat{\rho}^\circ,\hat{\eta}, \hat{\xi}, \hat{\phi}, \hat{\phi}^\circ)$ satisfies the following conditions: $\hat{\eta} \in \partial J_1(\hat{x}_1)$, $\hat{\xi} \in N_C(\hat{x}_1)$, $\hat{\vartheta} \in \partial \cR[J_2(\hat{x})]$, $\hat{\zeta}(\omega) \in \partial J_2(\hat{x}_1,\hat{x}_2(\omega);\omega)$ a.s., and: 
\begin{subequations}
\label{eq:FOC-MY-problem-limit}
\begin{align}
 \hat{\eta} +\E[\hat{\rho}] + \hat{\phi}_1^\circ+ e_{x_1}^* \hat{\lambda}^\circ_e+i_{x_1}^*(\hat{x})\hat{\lambda}^\circ_i +\hat{\xi} &= 0,\label{eq:FOC-MY-problem'-1}\\
\hat{\phi}_2^\circ+e_{x_2}^* \hat{\lambda}_e^\circ + i_{x_2}^*(\hat{x}) \hat{\lambda}_i^\circ &=0,\label{eq:FOC-MY-problem'-2}\\
\hat{\phi}_1(\omega) +e_{x_1}^*(\omega) \hat{\lambda}_e(\omega)+ i_{x_1}^*(\hat{x}_1,\hat{x}_2(\omega);\omega) \hat{\lambda}_i(\omega) -\hat{\rho}(\omega) &=0,\label{eq:FOC-MY-problem'-3}\\
\hat{\phi}_2(\omega) + e_{x_2}^*(\omega) \hat{\lambda}_e(\omega) +i_{x_2}^*(\hat{x}_1,\hat{x}_2(\omega);\omega) \hat{\lambda}_i(\omega)&= 0,\label{eq:FOC-MY-problem'-4}\\
\hat{x}_1 \in C, \quad e(\hat{x}_1, \hat{x}_2(\omega);\omega) &=0,\label{eq:FOC-MY-problem'-5}\\
 i(\hat{x}_1,\hat{x}_2(\omega);\omega)\leq_K 0, \quad \hat{\lambda}_i(\omega) \in K^{\oplus}, \quad  \langle \hat{\lambda}_i(\omega), i(\hat{x}_1,\hat{x}_2(\omega);\omega)\rangle_{R^*,R}&=0,\label{eq:FOC-MY-problem'-6} \\
 \, \hat{\lambda}_i^\circ \in \mathcal{K}^{\oplus}, \quad \langle \hat{\lambda}_i^\circ, i(\hat{x})\rangle_{(L^\infty(\Omega,R))^*,L^\infty(\Omega,R)} &= 0,\label{eq:FOC-MY-problem'-7}
\end{align}
\end{subequations}
with pointwise conditions holding for almost all $\omega \in \Omega$.
\end{theorem}

\begin{proof}
\textit{Characterization of accumulation points and topologies.}
By Proposition~\ref{prop:primal-dual-path-bounded}, for every $r\geq 0$, the path $\mathcal{C}_r$ is bounded in $\mathcal{X}$, in particular for $r=0$. Thus $\mathcal{C}_0$ is weakly* compact and hence there exists a subnet $(x^{\gamma}, \lambda^{\gamma}, \rho^\gamma, \eta^\gamma, \xi^\gamma, \phi^\gamma)_{\gamma \in J} \subset \mathcal{C}_0$ such that (after identifying absolutely continuous terms), we have $(x^{\gamma}, \lambda^{\gamma, a}, \rho^{\gamma,a}, \eta^\gamma, \xi^\gamma, \phi^{\gamma,a}) \rightharpoonup^* (\hat{x}, \lambda', \rho', \hat{\eta}, \hat{\xi}, \phi').$  Note that the decomposition of accumulation points, i.e., ${\lambda}'_e=\hat{\lambda}_e^a + \hat{\lambda}_e^\circ$,  ${\lambda}'_i=\hat{\lambda}_i^a +\hat{\lambda}_i^\circ$,  ${\rho}'=\hat{\rho}^a+\hat{\rho}^\circ$, and $\phi' = \hat{\phi}^a + \hat{\phi}^\circ$ follows from \Cref{thm:ioffe-levin}.

One can prove $j(x^{\gamma}) \rightarrow j(\hat{x})$ for the subnet $(x^\gamma)_{\gamma \in J}$ exactly as in \Cref{lemma:MY-basic-convergence}; see in particular \eqref{eq:optimality-of-limit-point}. Together with \Cref{assumption:example-problem-MY-iii}, we obtain $\lVert x_1^{\gamma}\rVert_{X_1} \rightarrow \lVert \hat{x}_1 \rVert_{X_1}$ as $\gamma \rightarrow \infty$ and thus $x_1^{\gamma} {\rightarrow} \hat{x}_1$ in $X_1$. From the formula \eqref{eq:solution-PDE-abstract} we have
\begin{equation*}
\esssup_{\omega \in \Omega} \lVert x_2^\gamma(\omega)-\hat{x}_2(\omega)\rVert_{X_2} = \esssup_{\omega \in \Omega}  \lVert e_{x_2}^{-1}(\omega)e_{x_1}(\omega) (x_1^\gamma-\hat{x}_1) \rVert_{X_2},
\end{equation*}
and hence $x_2^{\gamma} {\rightarrow} \hat{x}_2$ in $L^\infty(\Omega,X_2)$. We obtain $x^\gamma \rightarrow \hat{x}$ in $X$.

\textit{Verification of \eqref{eq:FOC-MY-problem'-3}.}
From \eqref{eq:FOC-MY-problem-gamma-2}, we have for all $\mu \in L^\infty(\Omega,X_1)$ and all $\gamma$
\begin{equation*}
\begin{aligned}
0 &= \langle\phi_1^\gamma + e_{x_1}^*(\cdot)  \lambda_e^{\gamma}(\cdot)+i_{x_1}^*(x_1^\gamma, x_2^\gamma(\cdot);\cdot) \lambda_i^{\gamma}(\cdot)- \rho^{\gamma}, \mu \rangle_{L^1(\Omega, X_1^*),L^\infty(\Omega,X_1)} . \\
\end{aligned}
\end{equation*} 
Notice that
\begin{itemize}
\item  
$\langle\phi_1^\gamma, \mu \rangle_{L^1(\Omega, X_1^*),L^\infty(\Omega,X_1)} = \langle\phi_1^{\gamma,a}, \mu \rangle_{(L^\infty(\Omega, X_1))^*,L^\infty(\Omega,X_1)} \rightarrow \langle \phi'_1, \mu \rangle_{(L^\infty(\Omega, X_1))^*,L^\infty(\Omega,X_1)}$ by weak* convergence of $\phi_1^\gamma$ in $(L^\infty(\Omega,X_1))^*$.
\item Using the identity for absolutely continuous functionals~\eqref{eq:absolutely-continuous-functionals} combined with \Cref{lemma:regularity-operators} we have
\begin{align*}
&\langle e_{x_1}^*(\cdot)  \lambda_e^{\gamma}(\cdot), \mu \rangle_{L^1(\Omega, X_1^*),L^\infty(\Omega,X_1)} = \langle   \lambda_e^{\gamma},  e_{x_1}(\cdot) \mu(\cdot) \rangle_{L^1(\Omega, W)),L^\infty(\Omega,W)} \\
&\quad =  \langle   \lambda_e^{\gamma,a},  e_{x_1}(\cdot) \mu(\cdot) \rangle_{(L^\infty(\Omega, W))^*,L^\infty(\Omega,W)} \\
& \rightarrow \langle  \hat{\lambda}_e,  e_{x_1}(\cdot) \mu(\cdot) \rangle_{(L^\infty(\Omega, W))^*,L^\infty(\Omega,W)} 
 = \langle e_{x_1}^*  \hat{\lambda}_e,  \mu \rangle_{(L^\infty(\Omega, X_1))^*,L^\infty(\Omega,X_1)},
\end{align*}
by weak* convergence of $\lambda_e^{\gamma,a}$ in $(L^\infty(\Omega, W))^*$.
\item From strong convergence of $(x^\gamma)_{\gamma \in J}$ and weak* convergence of $\lambda_i^{\gamma,a}$ in $(L^\infty(\Omega, R))^*$ we have  
\begin{align*}
&\langle i_{x_1}^*(x_1^\gamma, x_2^\gamma(\cdot);\cdot) \lambda_i^{\gamma}(\cdot), \mu \rangle_{L^1(\Omega, X_1^*),L^\infty(\Omega,X_1)} = \langle \lambda_i^{\gamma}(\cdot),  i_{x_1}(x_1^\gamma, x_2^\gamma(\cdot);\cdot)\mu(\cdot) \rangle_{L^1(\Omega, R^*),L^\infty(\Omega,R)}\\
&\quad= \langle \lambda_i^{\gamma,a},  i_{x_1}(x_1^\gamma, x_2^\gamma(\cdot);\cdot)\mu(\cdot) \rangle_{(L^\infty(\Omega, R))^*,L^\infty(\Omega,R)}\\
&\rightarrow  \langle \hat{\lambda}_i,  i_{x_1}(\hat{x}_1, \hat{x}_2(\cdot);\cdot)\mu(\cdot) \rangle_{(L^\infty(\Omega, R))^*,L^\infty(\Omega,R)} =  \langle  i^*_{x_1}(\hat{x}_1, \hat{x}_2(\cdot);\cdot) \hat{\lambda}_i, \mu \rangle_{(L^\infty(\Omega, X_1))^*,L^\infty(\Omega,X_1)},
\end{align*}
where we used the continuity of the partial derivative afforded by \Cref{assumption:example-problem-MY-ii}.
\item $\langle \rho^{\gamma}, \mu \rangle_{L^1(\Omega, X_1^*),L^\infty(\Omega,X_1)} = \langle\rho^{\gamma,a}, \mu \rangle_{(L^\infty(\Omega, X_1))^*,L^\infty(\Omega,X_1)} \rightarrow \langle \rho^{\gamma,a}, \mu \rangle_{(L^\infty(\Omega, X_1))^*,L^\infty(\Omega,X_1)} $ by weak* convergence of $\rho^{\gamma,a}$ in $(L^\infty(\Omega,X_1))^*$.
\end{itemize}

Thus 
\begin{equation*}
0= \langle  \phi'_1 + e_{x_1}^* {\lambda}'_e+i_{x_1}^*(\hat{x}){\lambda}'_i- {\rho}', \mu  \rangle_{(L^\infty(\Omega,X_1))^*,L^\infty(\Omega,X_1)}.
\end{equation*}

The sets $\mathcal{A}_1$ and $\mathcal{S}_1$ denote the sets of absolutely continuous functionals and singular functionals, respectively, defined on $L^\infty(\Omega,X_1)$.  By \Cref{lemma:regularity-operators}, it follows that
\begin{equation*}
\underbrace{(\hat{\phi}_1^\circ +e_{x_1}^* \hat{\lambda}^\circ_e+i_{x_1}^*(\hat{x})\hat{\lambda}^\circ_i- \hat{\rho}^\circ)}_{\in \mathcal{S}_1} + \underbrace{(\hat{\phi}_1^a + e_{x_1}^* \hat{\lambda}^a_e+i_{x_1}^*(\hat{x})\hat{\lambda}^a_i- \hat{\rho}^a)}_{\in \mathcal{A}_1\cong L^1(\Omega,X_1^*)} = 0.
\end{equation*}
The spaces $\mathcal{A}_1$ and $\mathcal{S}_1$ are complementary spaces, so $\mathcal{A}_1 \cap \mathcal{S}_1 = \{0\}$. This implies
\begin{align}
\label{eq:singular-equation-rho-circ}
\hat{\phi}_1^\circ+ e_{x_1}^* \hat{\lambda}^\circ_e+i_{x_1}^*(\hat{x})\hat{\lambda}^\circ_i- \hat{\rho}^\circ &=0,\\
\label{eq:absolutely-continuous-equation-rho-circ}
\hat{\phi}_1^a + e_{x_1}^* \hat{\lambda}^a_e+i_{x_1}^*(\hat{x})\hat{\lambda}^a_i- \hat{\rho}^a&=0. 
\end{align}
The expression \eqref{eq:absolutely-continuous-equation-rho-circ} immediately implies that \eqref{eq:FOC-MY-problem'-3} is satisfied.

\textit{Verification of  \eqref{eq:FOC-MY-problem'-2} and \eqref{eq:FOC-MY-problem'-4}.} Analogously to the prior step,  we have from \eqref{eq:FOC-MY-problem-gamma-3} for all $\nu \in L^\infty(\Omega,X_2)$ and all $\gamma$
\begin{equation*}
\begin{aligned}
0 &=  \langle \phi_2^\gamma+ e_{x_2}^*(\cdot) \lambda_e^\gamma(\cdot) +i_{x_2}^*(x_1^\gamma, x_2^\gamma(\cdot);\cdot) \lambda_i^\gamma(\cdot),\nu\rangle_{L^1(\Omega,X_2^*),L^\infty(\Omega,X_2)}.
\end{aligned}
\end{equation*}
Hence 
\begin{equation*}
0=  \langle \phi_2'+ e_{x_2}^* \lambda'_e +i_{x_2}^*(\hat{x}) {\lambda}'_i,\nu\rangle_{(L^\infty(\Omega,X_2))^*,L^\infty(\Omega,X_2)}.
\end{equation*}
Grouping by singular and absolutely continuous terms, we obtain
 \eqref{eq:FOC-MY-problem'-2} and \eqref{eq:FOC-MY-problem'-4}.

\textit{Verification of \eqref{eq:FOC-MY-problem'-1} and  \eqref{eq:FOC-MY-problem'-5}.}
The expression \eqref{eq:FOC-MY-problem'-5} follows from \Cref{corollary:affine-maps-are-weak-star-closed} and the fact that $C$ is closed.
From \eqref{eq:FOC-MY-problem-gamma-1}, we have for all $\mu \in X_1$
\begin{equation*}
\begin{aligned}
0 &= \langle \E[\rho^\gamma] +\eta^\gamma +\xi^\gamma, \mu \rangle_{X_1^*,X_1}\\
& = \langle \rho^{\gamma,a}, \mu\rangle_{(L^\infty(\Omega,X_1))^*,L^\infty(\Omega,X_1)} +\langle \eta^\gamma +\xi^\gamma, \mu \rangle_{X_1^*,X_1}.
\end{aligned}
\end{equation*}
After plugging in \eqref{eq:singular-equation-rho-circ}, the accumulation point therefore satisfies
\begin{equation*}
\begin{aligned}
0&= \langle  \hat{\rho}^a +\hat{\rho}^\circ, \mu \rangle_{(L^\infty(\Omega,X_1))^*,L^\infty(\Omega,X_1)} +\langle\hat{\eta} +\hat{\xi}, \mu  \rangle_{X_1^*,X_1}\\
&= \langle \E[\hat{\rho}] +\hat{\phi}_1^\circ+e_{x_1}^*\hat{\lambda}_e^\circ + i^*_{x_1}(\hat{x})\hat{\lambda}_i^\circ +\hat{\eta} +\hat{\xi}, \mu  \rangle_{X_1^*,X_1}
\end{aligned}
\end{equation*}
 for all $\mu \in X_1$, which is equivalent to \eqref{eq:FOC-MY-problem'-1}.

\textit{Verification of \eqref{eq:FOC-MY-problem'-6} and \eqref{eq:FOC-MY-problem'-7}.}
We proved $i(\hat{x}_1, \hat{x}_2(\omega);\omega) \leq_K 0$ in \Cref{lemma:MY-basic-convergence}.
For all $\gamma$, $\lambda_i^\gamma(\omega) \in K_H^\oplus$.
It follows for all $r \in L^\infty(\Omega,R)$ satisfying $r(\omega) \in K_H \cap R = K$ that 
\begin{equation*}
0 \leq \E[(\lambda_i^\gamma, r)_H] =\E[\langle \lambda_i^{\gamma}, r\rangle_{R^*,R}] = \langle  \lambda_i^{\gamma,a}, r \rangle_{(L^\infty(\Omega,R))^*,L^\infty(\Omega,R)}.
\end{equation*}
Therefore
\begin{equation*}
0 \leq \langle \hat{\lambda}_i^a + \hat{\lambda}_i^\circ, r  \rangle_{(L^\infty(\Omega,R))^*,L^\infty(\Omega,R)}.
\end{equation*}
Since $\hat{\lambda}_i^a$ and $\hat{\lambda}_i^\circ$ belong to complementary subspaces, we conclude 
\begin{equation}
\label{eq:lagrange-multipliers-in-dual-cone}
\hat{\lambda}_i(\omega) \in K^\oplus \text{ a.s.} \quad \text{and} \quad  \hat{\lambda}_i^\circ \in \mathcal{K}^\oplus.
\end{equation}

By weak* convergence of $\lambda_i^\gamma$, strong convergence of $x^\gamma$, and continuity of $x \mapsto i(x_1, x_2(\cdot);\cdot)$, we have 
\begin{equation}
\label{eq:inequality-lagrange-multipliers}
\begin{aligned}
\E[(\lambda_i^\gamma, -i(x_1^\gamma, x_2^\gamma(\cdot); \cdot))_H] 
 = \langle \lambda_i^{\gamma,a}, -i(x_1^\gamma, x_2^\gamma(\cdot); \cdot)\rangle_{(L^\infty(\Omega,R))^*,L^\infty(\Omega,R)}\\ \rightarrow \langle \lambda'_i, -i(\hat{x}_1, \hat{x}_2(\cdot); \cdot)\rangle_{(L^\infty(\Omega,R))^*,L^\infty(\Omega,R)} \geq 0.
\end{aligned}
\end{equation}
Notice that for all $\gamma$,
\begin{align*}
\E[(\lambda_i^\gamma, -i(x^\gamma))_H] &= \gamma \E[ \big(i(x^\gamma)+\pi_{K_H}(-i(x^\gamma)), -i(x^\gamma)\big)_H]\\
&=-\gamma \E[\lVert i(x^\gamma) \rVert_H^2] + \gamma \E[ \big(\pi_{K_H}(-i(x^\gamma)) - \pi_{K_H}(0), -i(x^\gamma)\big)_H] \\
&\leq -\gamma \E[\lVert i(x^\gamma) \rVert_H^2]  + \gamma \E[\lVert \pi_{K_H}(-i(x^\gamma)) - \pi_{K_H}(0)\rVert_H \lVert i(x^\gamma)\rVert_H]\\
&\leq -\gamma \E[\lVert i(x^\gamma) \rVert_H^2]+\gamma \E[\lVert i(x^\gamma) \rVert_H^2] = 0, 
\end{align*}
where we used monotonicity of the expectation operator, nonexpansivity of the projection operator, as well as $0 \in K_{H}$. This combined with \eqref{eq:inequality-lagrange-multipliers} yields
\begin{align*}
\langle \hat{\lambda}_i(\omega), i(\hat{x}_1,\hat{x}_2(\omega);\omega)\rangle_{R^*,R} &= 0, \quad \text{and} \quad \langle \hat{\lambda}_i^\circ, i(\hat{x}) \rangle_{(L^\infty(\Omega,R))^*, L^\infty(\Omega,R)} = 0.
\end{align*}

\textit{Sequential statements.}
We note that $X_1$ and $L^\infty(\Omega,X_2)$ are metrizable spaces. Note that $X_1$ is reflexive. From the subnet $(x^\gamma, \eta^\gamma, \xi^\gamma)_{\gamma \in J}$, it is therefore possible to extract a subsequence $\gamma_n \rightarrow \infty$ such that $x^{\gamma_n} \rightarrow \hat{x}$, $\eta^{\gamma_n} \rightharpoonup \hat{\eta}$, and
$ \xi^{\gamma_n} \rightharpoonup \hat{\xi}$. 

\textit{Convergence of elements from subdifferentials.}
Strong convergence  of $x^{\gamma_n}$ and weak convergence of $\xi^{\gamma_n} \in N_C(x_1^{\gamma_n})$ imply $\hat{\xi} \in N_C(\hat{x}_1).$ By the same reasoning, $\hat{\eta} \in \partial J_1(\hat{x}_1)$.
We now reveal the structure of $\phi^{\gamma_n} = \vartheta^{\gamma_n} \zeta^{\gamma_n}$, where $\vartheta^{\gamma_n} \in \partial \cR[J_2(x^{\gamma_n})]$ and $\zeta^{\gamma_n}(\omega) \in \partial J_2(x_1^{\gamma_n}, x_2^{\gamma_n}(\omega);\omega)$ a.s.
Clearly, $\vartheta^{\gamma_n}$ is bounded in $L^{p'}(\Omega)$ due to the boundedness of $x^{\gamma_n}$, hence on a subsequence (we use the same labeling) there exists a $\hat{\vartheta} \in L^{p'}(\Omega)$ such that $\vartheta^{\gamma_{n}}  \rightharpoonup^* \hat{\vartheta}.$ Therefore, strong convergence of $x^{\gamma_n}$ to $\hat{x}$ implies $\hat{\vartheta} \in \partial \cR[J_2(\hat{x})]$. Moreover, for almost every $\omega$, the sequence $\zeta^{\gamma_n}(\omega)$ is bounded in $X_1^* \times X_2^*$ due to the boundedness of $x^{\gamma_n}$, hence there exists a subsequence $\{ n_k\}$ and a limit point $\hat{\zeta}(\omega)$ such that $\zeta^{\gamma_{n_k}}(\omega) \rightharpoonup \hat{\zeta}(\omega).$
From the convergence $x^{\gamma_n} \rightarrow \hat{x}$, we have on a further subsequence (with the same labeling) the pointwise convergence $x^{\gamma_{n_k}}(\omega) \rightarrow \hat{x}(\omega)$ a.s.~and hence $\hat{\zeta}(\omega) \in \partial J_2(\hat{x}_1,\hat{x}_2(\omega);\omega)$ a.s. The fact that $\hat{\zeta} \in L^p(\Omega, X_1^*) \times L^p(\Omega,X_2^*)$ follows by \Cref{rem:weak-compactness-subdifferential-J2}.
\end{proof}

The system \eqref{eq:FOC-MY-problem-limit} in \Cref{thm:optimality-of-limit-primal-dual} is not necessary and sufficient for optimality since it is not generally guaranteed that the sequence given by $\phi^{\gamma_n} = \vartheta^{\gamma_n} \zeta^{\gamma_n}$ has a limit of the form $\hat{\vartheta} \hat{\zeta}$ as required by \eqref{eq:FOC-MY-problem-3}--\eqref{eq:FOC-MY-problem-4}.  Moreover, the system \eqref{eq:FOC-MY-problem-limit} contains the singular terms $\hat{\phi}^\circ$. However, there are two relevant cases in which it is possible to show the stronger result.

\begin{corollary}
\label{cor:expectation-necessary-sufficient}
With the same assumptions as in \Cref{thm:optimality-of-limit-primal-dual}, assume additionally that $\cR = \E$. Then $\hat{\phi} = \hat{\zeta}$, $\hat{\phi}^\circ \equiv 0$, and the system \eqref{eq:FOC-MY-problem-limit} is necessary and sufficient for optimality.
\end{corollary}
\begin{proof}
In the case where $\cR = \E$, we have $\vartheta^\gamma =  \hat{\vartheta} \equiv 1$ for all $\gamma \geq 0$. Hence from $\hat{\zeta}(\omega) \in \partial J_2(\hat{x}_1,\hat{x}_2(\omega);\omega)$ a.s.~and $\phi^{\gamma_n} = \zeta^{\gamma_n}$ (after identification) belongs to $\partial \E[ J_2(x^{\gamma_n})],$ which using strong convergence of $x^{\gamma_n}$ implies that $\hat{\phi} \in \partial \E[J_2(\hat{x})]$. From \Cref{cor:characterization-subdifferentials}, we obtain that $\hat{\phi}^\circ \equiv 0$. By \Cref{lemma:KKT-general}, the conditions posed by \eqref{eq:FOC-MY-problem-limit} are now necessary and sufficient.
\end{proof}

\begin{corollary}
\label{cor:smooth-J2-necessary-sufficient}
With the same assumptions as in \Cref{thm:optimality-of-limit-primal-dual}, assume additionally that $J_2(\cdot,\cdot;\omega)$ is continuously Fr\'echet differentiable on $X_1 \times X_2$ with derivative $D J_2(\cdot,\cdot;\omega)$ for almost every $\omega$; for every $r>0$ there exists $a'_{r} \in L^p(\Omega)$ such that $D J_2(x_1,x_2;\omega) \leq a'_{r}(\omega)$ a.s. for $\lVert x_1\rVert_{X_1} + \lVert x_2\rVert_{X_2} \leq r.$ Then $\hat{\phi} = \hat{\vartheta}\hat{\zeta}$, $\hat{\phi}^\circ \equiv 0$, and the system \eqref{eq:FOC-MY-problem-limit} is necessary and sufficient for optimality.
\end{corollary}
\begin{proof}
Fr\'echet differentiability of $J_2$ implies the subdifferential is a singleton, i.e., $\partial J_2(x_1, x_2;\omega) = \{ D J_2(x_1,x_2;\omega)\}$. From strong convergence $x^{\gamma_n} \rightarrow \hat{x}$ in $X$ one has (potentially on a subsequence, where we use the same labeling) $(x_1^{\gamma_n},x_2^{\gamma_n}(\omega)) \rightarrow (\hat{x}_1,\hat{x}_2(\omega))$ a.s. In particular it follows by continuity of the derivative that
\begin{align*}
\zeta^{\gamma_n}(\omega) = DJ_2(x_1^{\gamma_n}, x_2^{\gamma_n}(\omega);\omega) \rightarrow DJ_2(\hat{x}_1, \hat{x}_2(\omega); \omega) = \hat{\zeta}(\omega) \quad \text{in $X_1^* \times X_2^*$ a.s.}
\end{align*}
Together with the growth condition and Lebesgue's Dominated Convergence theorem, we have strong convergence $\zeta^{\gamma_n} \rightarrow \hat{\zeta}$ in $L^p(\Omega, X_1^*) \times L^p(\Omega, X_2^*)$. We now prove that $\phi^{\gamma_n} \rightharpoonup \hat{\phi} = \hat{\vartheta} \hat{\zeta} \in\mathcal{Y}:=L^1(\Omega,X_1^*) \times L^1(\Omega, X_2^*)$. Naturally, we have
\begin{equation}
\label{eq:split-differential}
\langle v, \phi^{\gamma_n}\rangle_{\mathcal{Y}^*,\mathcal{Y}} = \langle v, \vartheta^{\gamma_n}(\zeta^{\gamma_n}- \hat{\zeta})\rangle_{\mathcal{Y}^*,\mathcal{Y}}  + \langle v, \vartheta^{\gamma_n} \hat{\zeta}\rangle_{\mathcal{Y}^*,\mathcal{Y}} \quad \forall v \in \mathcal{Y}^*.
\end{equation}
Notice that 
\begin{align*}
\langle v, \vartheta^{\gamma_n}(\zeta^{\gamma_n}- \hat{\zeta})\rangle_{\mathcal{Y}^*,\mathcal{Y}} \leq \lVert v\rVert_{\mathcal{Y}^*} (\lVert \vartheta^{\gamma_n} (\zeta_1^{\gamma_n}- \hat{\zeta}_1)\rVert_{L^1(\Omega,X_1^*)} + \lVert \vartheta^{\gamma_n} (\zeta_2^{\gamma_n}- \hat{\zeta}_2)\rVert_{L^1(\Omega,X_2^*)} )
\end{align*}
and for $i=1,2$,
\begin{align*}
\lVert \vartheta^{\gamma_n} (\zeta_i^{\gamma_n}- \hat{\zeta}_i)\rVert_{L^1(\Omega,X_i^*)} &= \int_{\Omega} \lVert \vartheta^{\gamma_n}(\omega) (\zeta_i^{\gamma_n}(\omega)- \hat{\zeta}_i(\omega))\rVert_{X_i^*} \D \pP(\omega) \\
& =  \int_{\Omega} |\vartheta^{\gamma_n}(\omega)| \lVert  (\zeta_i^{\gamma_n}(\omega)- \hat{\zeta}_i(\omega))\rVert_{X_i^*} \D \pP(\omega) \\
&\leq \lVert \vartheta^{\gamma_n}\rVert_{L^{p'}(\Omega)} \lVert \zeta_i^{\gamma_n}- \hat{\zeta}_i\rVert_{L^p(\Omega,X_i^*)},
\end{align*}
where we used H\"older's inequality in the last step.
Combining this with strong convergence $\zeta^{\gamma_n} \rightarrow \hat{\zeta}$ and boundedness of $\vartheta^{\gamma_n}$, the first term on the right hand side of \eqref{eq:split-differential} disappears in the limit. For the second term on the right hand side of  \eqref{eq:split-differential}, we have
\begin{align*}
\langle v, \vartheta^{\gamma_n} \hat{\zeta} \rangle_{\mathcal{Y}^*,\mathcal{Y}} &= \sum_{i=1}^2 \int_{\Omega} \langle v_i(\omega), \vartheta^{\gamma_n}(\omega) \hat{\zeta}(\omega)\rangle_{X_i,X_i^*} \D \pP(\omega) \\
&= \sum_{i=1}^2 \int_{\Omega} \vartheta^{\gamma_n}(\omega) \langle v_i(\omega), \hat{\zeta}(\omega)\rangle_{X_i,X_i^*} \D \pP(\omega).
\end{align*}
Notice that $\langle v_i(\cdot), \hat{\zeta}(\cdot)\rangle_{X_i,X_i^*} \in L^p(\Omega)$ by the regularity of $\hat{\zeta}$. Hence $\vartheta^{\gamma_n} \rightharpoonup^* \hat{\vartheta}$ gives  $\langle v, \vartheta^{\gamma_n} \hat{\zeta}\rangle_{\mathcal{Y}^*,\mathcal{Y}} \rightarrow  \langle v, \hat{\vartheta} \hat{\zeta}\rangle_{\mathcal{Y}^*,\mathcal{Y}}$ for all $v \in \mathcal{Y}^*$. Returning to \eqref{eq:split-differential}, we obtain  $\langle v, \phi^{\gamma_n}\rangle_{\mathcal{Y}^*,\mathcal{Y}} \rightarrow  \langle v, \hat{\vartheta} \hat{\zeta}\rangle_{\mathcal{Y}^*,\mathcal{Y}}$ for all $v \in \mathcal{Y}^*.$ Since $\hat{\phi} := \hat{\vartheta}\hat{\zeta} \in \mathcal{Y},$ the singular term $\hat{\phi}^\circ$ vanishes, and the system \eqref{eq:FOC-MY-problem-limit} is necessary and sufficient for optimality.
\end{proof}

To close this section, we discuss the situation where the adjoint variable $\lambda_e^\gamma$ is more regular than our strict framework would suggest. The choice of $W$ and $R$ clearly need to be compatible with the interior point condition \eqref{eq:constraint-qualification} for the original problem. However, for Problem \eqref{eq:model-PDE-UQ-state-constraints-reformulation}, the condition  \eqref{eq:constraint-qualification} is trivially fulfilled as long as $e_{x_2}$ is surjective. Without the restrictions on $W$, it is possible to gain additional regularity in the adjoint variable $\lambda_e^\gamma$ if we view the state $x_2^\gamma$ in a less regular space, a fact that can be exploited in computations. In the following lemma, we make this idea more precise.
\begin{lemma}
\label{lemma:adjoint-variable-higher-regularity}
Under the same conditions as \Cref{lemma:optimality-regular-problem}, suppose further that $X_2$ and $W$ are continuously embedded in $\tilde{X}_2$ and $\tilde{W}$, respectively. Suppose $e(\cdot,\cdot;\omega) \in \mathcal{L}(X_1\times \tilde{X}_2, \tilde{W})$ and $e_{x_2}(\omega) \in \mathcal{L}(\tilde{X}_2, \tilde{W})$ is a linear isomorphism for almost every $\omega \in \Omega$ with $e_{x_2}^{-1} \in L^\infty(\Omega,\mathcal{L}(\tilde{W},\tilde{X}_2)).$  Additionally, suppose $i(\cdot,\cdot;\omega)$ is continuously Fr\'echet differentiable on $X_1 \times \tilde{X}_2$ for almost every $\omega$ with partial derivative satisfying $i_{x_2}(x_1,x_2;\cdot) \in L^\infty(\Omega,\mathcal{L}(\tilde{X}_2,R))$. Then $\lambda_e^\gamma$ belongs to $L^1(\Omega,\tilde{W}^*)$. 
\end{lemma}

\begin{proof}
From \eqref{eq:FOC-MY-problem-gamma-3} we immediately obtain
\begin{equation*}
 \lambda_e^\gamma(\omega) =- e_{x_2}^{-*}(\omega) (\zeta_2^\gamma(\omega)\vartheta^\gamma(\omega) +i_{x_2}^*(x_1^\gamma,x_2^\gamma(\omega);\omega) \lambda_i^\gamma(\omega) ) \in \tilde{W}^*.
\end{equation*}
The fact that $\lambda_e^\gamma \in L^1(\Omega,\tilde{W}^*)$ follows by the assumptions.
\end{proof}
 
\section{Conclusion and Discussion}
\label{sec:Conclusion}
In this paper, we presented optimality conditions for a large class of potentially risk-averse convex stochastic optimization problems. The applications in mind involve optimal control problems, where the control belongs to a Banach space and is coupled to a state, which is a vector-valued random variable. \cite{Geiersbach2020+} already presented optimality conditions under the assumption of relatively complete recourse, which -- while not satisfiable for nontrivial problems with additional state constraints -- is satisfied for a sequence of regularized problems. In this paper, we strengthened the \cite{Geiersbach2020+} results by presenting optimality conditions without this assumption, resulting in optimality conditions having singular Lagrange multipliers. Additionally, we formulated optimality conditions general enough to handle problems having risk measures in the objective.

As the second main contribution, we showed that regularized problems using a Moreau--Yosida regularization for the conical constraints are consistent with the original problem as the penalization parameter is taken to infinity. While \cite{Geiersbach2020+} already showed consistency with respect to the primal variables for a specific example, we generalize these results further and show consistency with respect to dual variables. 

The results presented here provide the framework for constructing algorithms via a penalization approach. For algorithms, it is preferred to work with sequential statements as opposed to nets. This is possible on sets of vanishingly small measure using the following result.
\begin{lemma}[\cite{Ball1989}]
\label{lemma:Ball}
Suppose $X$ is a reflexive Banach space and $\{f_n\} \subset L^1(\Omega,X)$ is a bounded sequence. Then there exists $\bar{f} \in L^1(\Omega,X)$ and a subsequence $\{f_{n_j} \}$ and nonincreasing sequence of sets $E_k \in \mathcal{F}$ with $\lim_{k \rightarrow \infty} \pP(E_k) =0$ such that
\begin{equation*}
f_{n_j} \rightharpoonup f \quad \text{in } L^1(\Omega\backslash E_k, X)
\end{equation*}
for every fixed $k$.
\end{lemma}
A drawback of \Cref{lemma:Ball} is that one does not know the sets $E_k$. Closing this gap will be the subject of future study. Additionally, it is planned to carry over our theory to the setting of generalized Nash equilibrium problems, where players' strategy sets are coupled, e.g., over a PDE and additionally physical requirements lead to a constraint on the state. 

\newpage
\appendix

\section{Essentially Bounded and Strongly Measurable Linear Operators}
\label{sec:linear-transformations}
\begin{proof}[Proof of \Cref{lemma:regularity-operators}]
(i) First, notice that for any $y \in L^\infty(\Omega,Y)$, we have
\begin{equation*}
\lVert \cA(\cdot) y(\cdot) \rVert_{L^\infty(\Omega,Z)} \leq \lVert \cA \rVert_{L^\infty(\Omega,\mathcal{L}(Y,Z))} \lVert y \rVert_{L^\infty(\Omega,Y)} < \infty.
\end{equation*}
Since the product of strongly measurable functions is again strongly measurable, $\cA(\cdot)y(\cdot)\colon \Omega \rightarrow Z$ is strongly measurable, meaning $\cA_1 \colon L^\infty(\Omega,Y) \rightarrow L^\infty(\Omega,Z), y(\cdot) \mapsto \cA(\cdot)y(\cdot)$ is well-defined and bounded. Hence the adjoint operator $\cA_1^* \colon (L^\infty(\Omega,Z))^* \rightarrow (L^\infty(\Omega,Y))^*$ is bounded. 
Now, let $w^\circ \in (L^\infty(\Omega,Z))^*$ be a singular functional, i.e., by \Cref{def:singular-functionals} there exists a sequence $\{ F_{n}\} \subset \mathcal{F}$ with $F_{n+1} \subset F_n$ such that $\pP(F_{n}) \rightarrow 0$ and $\langle w^\circ,z\rangle_{(L^\infty(\Omega,Z))^*,L^\infty(\Omega,Z)} = 0$ for all $z \in L^\infty(\Omega,Z)$ satisfying $z(\omega) = 0$ for all $\omega \in F_n$ (for some $n$). For this same sequence $\{ F_n\}$, let $y$ be an arbitrary element of $L^\infty(\Omega,Y)$ satisfying $y=0$ on $F_n$. Then linearity of $\cA(\omega)$ implies $\cA(\omega) y(\omega) = 0$ for all $\omega \in F_n$ and so with $\cA_1 y = \cA(\cdot)y(\cdot)$ we have
\begin{equation*}
 \langle \cA_1^* w^\circ, y\rangle_{(L^\infty(\Omega,Y))^*,L^\infty(\Omega,Y)} =  \langle w^\circ,\cA(\cdot) y(\cdot)\rangle_{(L^\infty(\Omega,Z))^*,L^\infty(\Omega,Z)} = 0.
\end{equation*}
Since $y$ is arbitrary, $\cA_1^*w^\circ$ is a singular element of $(L^\infty(\Omega,Y))^*$.
Finally, if $z^* \in L^1(\Omega,Z^*)$, we have 
\begin{equation*}
\lVert {\cA}^*(\cdot) z^*(\cdot) \rVert_{L^1(\Omega,Y^*)} \leq \lVert \cA^* \rVert_{L^\infty(\Omega,\mathcal{L}(Z^*,Y^*))} \lVert z^*\rVert_{L^1(\Omega,Z^*)},
\end{equation*}
meaning $\hat{\cA}_1^*\colon L^1(\Omega,Z^*) \rightarrow L^1(\Omega,Y^*), z^*(\cdot) \mapsto \cA^*(\cdot) z^*(\cdot)$ is bounded. \\
(ii) The second part of the proof follows using analogous arguments to the above.
\end{proof}

\begin{proof}[Proof of \Cref{corollary:affine-maps-are-weak-star-closed}]
(i) Let $\{y_n\}$ be a sequence in $L^\infty(\Omega,Y)$ with $y_n \rightharpoonup^* \hat{y}$. We need to show that $\cA_1^* y_n \rightharpoonup^* \cA_1^* \hat{y}$. We have for all $z^* \in L^1(\Omega, Z^*)$ and all $n$
\begin{align*}
\label{eq:basic-equality-adjoints}
 \langle z^*, \cA_1 y_n \rangle_{L^1(\Omega,Z^*), L^\infty(\Omega,Z)}   &= \E[\langle z^*(\cdot),\cA(\cdot) y_n(\cdot) \rangle_{Z^*,Z}]\\
 & = \E[\langle \cA^*(\cdot) z^*(\cdot),y_n(\cdot)\rangle_{Y^*,Y}] = \langle \hat{\cA}_1^* z^*, y_n \rangle_{L^1(\Omega,Y^*), L^\infty(\Omega,Y)}
\end{align*}
Hence as $n \rightarrow \infty$,  
\begin{equation}
\label{eq:weak-continuity-step}
\langle z^*, \cA_1 y_n \rangle_{L^1(\Omega,Z^*), L^\infty(\Omega,Z)}  \rightarrow \langle \hat{\cA}_1^* z^*, \hat{y} \rangle_{L^1(\Omega,Y^*), L^\infty(\Omega,Y)} =\langle z^*,\cA_1 \hat{y} \rangle_{L^1(\Omega,Z^*),L^\infty(\Omega,Z)}. 
\end{equation}
Therefore $\mathcal{A}_1: L^\infty(\Omega, Y) \rightarrow L^\infty(\Omega, Z)$ is weakly-to-weakly* continuous. If $y_n$ belongs to the set $E:=\{y \in L^\infty(\Omega,Y): \cA(\omega) y(\omega) = b(\omega)  \text{ a.s.}\}$ for all $n$, then \eqref{eq:weak-continuity-step} gives $\E[\langle z^*(\cdot),\cA(\cdot) \hat{y}(\cdot)\rangle_{Z^*,Z}]=\E[\langle z^*(\cdot),b(\cdot)\rangle_{Z^*,Z}]$ for all $z^* \in L^1(\Omega,Z^*)$. Thus $\cA(\omega)\hat{y} (\omega) = b(\omega)$ a.s, meaning $\hat{y}$ belongs to the set $E$, making $E$ weakly* closed.\\
(ii) The second claim follows using similar arguments after noting that $Y$ is reflexive and taking $y_n \rightharpoonup \hat{y}$.
\end{proof}

\section{Existence of Saddle Points}
\label{section:saddle-points-relatively-complete-recourse}
In this section, we adapt proofs from \cite{Geiersbach2020+} to include an objective containing the term $\cR[J_2(x)]$ instead of $\E[J_2(x)].$ Additionally, we omit the abstract constraint $x_2(\omega) \in C_2$ a.s.

We define the dual problem
\begin{equation}\tag{$\textup{D}$}
\label{eq:dual-problem}
\underset{\z \in \Lambda}{\text{maximize}} \left\lbrace g(\z) := \inf_{x \in X} L(x,\z) \right\rbrace.
\end{equation}
By basic duality, the question of the existence of saddle points is the
same as identifying those $(\bx,\bz)$ for which the minimum of
Problem~\eqref{eq:model-problem-abstract-long} and maximum of
Problem~\eqref{eq:dual-problem} is attained, i.e.,
\begin{equation*}
  \inf \textup{P} = \inf_{x \in X} \sup_{\z \in \Lambda} L(x,\z) = \sup_{\z \in \Lambda} \inf_{x \in X} L(x,\z) = \sup \textup{D}.
\end{equation*}
Let $\varphi(x,u) = j(x)$ if $x \in F_{\text{ad}, u}$, and $\varphi(x,u) = \infty$ otherwise. This function is related to the Lagrangian by the relationship $\varphi(x,u) = \sup_{\lambda \in \Lambda} \{ L(x,\lambda)-\langle u,\lambda\rangle_{U,\Lambda}\}.$ Additionally, it is related to the value function through $v(u) = \inf_{x \in X} \varphi(x,u)$. Obviously, $v(0) = \inf \textup{P}$.

We proceed by first proving weak* lower semicontinuity of $\varphi.$  

\begin{lemma}
\label{lemma:lsc-conjugate-function}
Let \Cref{assumption:general-problem} be satisfied. Then the
function $\varphi\colon X \times U \rightarrow \R \cup \{ \infty\}$ is
weakly$^*$ lower semicontinuous.
\end{lemma}

\begin{proof}
Let $X' = X_1^* \times L^1(\Omega,X_2^*)$, which can be paired with $X$.
Let $Y:=X \times U$ and denote the pairing on $Z:=X' \times \Lambda$ by
\begin{equation}
\label{eq:product-dual-pairing}
\langle y,z\rangle_{Y,Z}:= \langle x,x'\rangle_{X,X'} + \langle u,\lambda \rangle_{U,\Lambda}.
\end{equation}  
Since $Y=Z^*$, the topology induced by the pairing \eqref{eq:product-dual-pairing} coincides with the weak$^*$ topology on $Y$. We define
  $\varphi_1(x) = j(x) + \delta_C(x_1)$ and
\begin{align*}
&\varphi_2(\u,\y,u;\omega) = \begin{cases}
0, & \text{if }  e(\u,\y;\omega) =u_e,  i(\u, \y; \omega) \leq_K u_i,\\
\infty, & \text{otherwise}.
\end{cases}
\end{align*}
Obviously, $\varphi(x,u) = \varphi_1(x) +\E[ \varphi_2(\u,\y(\cdot),u(\cdot))]$.
Let $\langle \cdot,\cdot \rangle_{Y',Z'}$ denote the pairing of
$Y':=X_1 \times X_2 \times (W\times R)$ with
$Z':=X_1^* \times X_2^* \times (W^* \times R^*)$; then the conjugate
integrand to $\varphi_2$ is given by
\begin{equation*}
  \varphi_2^*(z';\omega) = \sup_{y' \in Y'} \{ \langle y',z' \rangle_{Y',Z'} - \varphi_2(y';\omega)\}.
\end{equation*}

Defining $h(y';\omega)= 0$ for $y'= (\u,\y,u)$ we have
$h(y';\omega) \leq \varphi_2(y';\omega)$ a.s. The function $h$ is trivially a normal convex integrand that is integrable on
$X \times U$. Thus with the conjugate
integrand $h^*$, $I_h$ and $I_{h^*}$ are conjugate to each other by
\Cref{lemma:duality-conjugate-functions}, meaning that
$I_{h^*} \not\equiv \infty.$ 

Since $h \leq \varphi_2$ we have $h^* \geq \varphi_2^*$,
and hence there exists a point
$z \in Z$ such that $I_{\varphi_2^*}(z) < \infty$.
Since there clearly exists a point such that $I_{\varphi_2}$ is finite, it
follows that $I_{\varphi_2}$ and $I_{\varphi_2^*}$ are conjugate to one another and are weakly$^*$ lower semicontinuous; see \Cref{remark:weak-star-lsc}.

Notice that $J_1+\delta_C$ is weakly* lower semicontinuous, since it is weakly lower semicontinuous with respect to the natural pairing on the reflexive space $X_1$. Hence, together with \Cref{prop:weak-weak*-lsc}, we have weak* lower semicontinuity of $\varphi_1$. 
\end{proof}

\begin{theorem}
\label{thm:minP-supD}
Let \Cref{assumption:general-problem} and \Cref{subasu:existence-saddle-points-i} be satisfied. Then
\begin{equation*}
  -\infty < \min \textup{P} = \sup \textup{D},
\end{equation*}
meaning that the primal problem~\eqref{eq:model-problem-abstract-long} attains its minimum, and the minimal
value coincides with the supremum of the dual problem \eqref{eq:dual-problem}, which need not be attained.
\end{theorem}

\begin{proof}
If $F_{\textup{ad},u}$ is bounded for all $u$ in a neighborhood of zero, then there exist closed, convex, and bounded sets $C_1\subset X_1$ and $C_2 \subset X_2$ such that if $x \in F_{\textup{ad},u}$, then $x_1 \in C_1$ and $x_2(\omega) \in C_2$ a.s. Then the statement follows by \cite[Theorem 3.4]{Geiersbach2020+}. Otherwise if $j$ is radially unbounded, then the statement follows by \cite[Corollary 3.5]{Geiersbach2020+}. In both cases, we take advantage of the weak* lower semicontinuity of $\varphi$ as proven in \Cref{lemma:lsc-conjugate-function}.
\end{proof}

We now define the extended dual problem by
\begin{equation}\tag{$\bar{\textup{D}}$}
\label{eq:dual-tilde-problem}
\underset{(\lambda,\lambda^\circ) \in \Lambda \times \Lambda^\circ}{\text{maximize}}  \left\lbrace \bar{g}(\lambda,\lambda^\circ) := \inf_{x \in X} \bar{L}(x,\lambda,\lambda^\circ) \right\rbrace.
\end{equation}

Clearly, $\bar{g}(\lambda,0) = g(\lambda)$ and thus
$\sup \textup{D} \leq \sup \bar{\textup{D}}.$
Additionally, $ \sup \bar{\textup{D}} \leq \inf \textup{P}$, since
by~\eqref{eq:feasibility-negative-Lagrangian-term-1}, we have
\begin{align*}
\sup_{(\lambda,\lambda^\circ) \in \Lambda \times \Lambda^\circ} \bar{g}(\lambda,\lambda^\circ) &= \sup_{(\lambda,\lambda^\circ) \in \Lambda \times \Lambda^\circ} \inf_{x \in X} \{ L(x,\lambda) + L^\circ(x,\lambda^\circ) \} \\
& \leq \inf_{x \in X} \sup_{(\lambda,\lambda^\circ) \in  \Lambda \times \Lambda^\circ}  \{ L(x,\lambda)  + L^\circ(x,\lambda^\circ) \}. 
\end{align*}

\begin{theorem}
\label{thm:infP-maxDbar}
Let \Cref{assumption:general-problem} and \Cref{subasu:existence-saddle-points-ii} be satisfied. Then 
\begin{equation}
\label{eq:infP-equal-maxD-tilde}
\inf \textup{P} = \max \bar{\textup{D}} < \infty,
\end{equation}
meaning that the dual problem \eqref{eq:dual-tilde-problem} attains its maximum and the maximum value coincides with the infimum of the primal problem~\eqref{eq:model-problem-abstract-long}, which needs not be attained.
\end{theorem}
\begin{proof}
 \Cref{subasu:existence-saddle-points-ii} guarantees that $v$ is bounded above on a neighborhood of zero, so \eqref{eq:infP-equal-maxD-tilde}  follows by, e.g.,~\cite[Theorem 17]{Rockafellar1974}.
\end{proof}

\begin{theorem}
\label{thm:infP-maxD}
Let \Cref{assumption:general-problem}, \Cref{subasu:existence-saddle-points-ii}, and \Cref{assumption:relatively-complete-recourse} be satisfied. Then 
\begin{equation*}
  \inf{\textup{P}}=\max{\textup{D}} < \infty,
\end{equation*}
meaning that the dual problem~\eqref{eq:dual-problem} attains its maximum and the maximum value coincides with the infimum of the primal problem~\eqref{eq:model-problem-abstract-long}, which needs not be attained.
\end{theorem}
\begin{proof}
This proof follows exactly with the same arguments used in \cite[Theorem 3.8]{Geiersbach2020+}. We repeat them here to keep the paper self-contained.  \Cref{assumption:relatively-complete-recourse} implies 
\begin{equation}
\label{eq:relation-dual-extended-dual}
\bar{g}(\lambda,\lambda^\circ) \leq  g(\lambda) \quad \forall (\lambda,\lambda^\circ) \in \Lambda_0 \times \Lambda_0^\circ.
\end{equation}
Then in combination with \Cref{thm:infP-maxDbar}, this implies
\begin{equation*}
 \max \bar{\textup{D}} \le \sup \textup{D} \le \max \bar{\textup{D}},
\end{equation*}
and the proof is done since a solution $(\lambda, \lambda^\circ)$ of $(\bar{\textup{D}})$ gives a solution $\lambda$ of $(\textup{D})$.

To show~\eqref{eq:relation-dual-extended-dual}, let $(\lambda,\lambda^\circ) \in \Lambda_0 \times \Lambda_0^\circ$ be arbitrary. We skip the trivial case $\bar{g}(\lambda,\lambda^\circ) = -\infty$ and obtain by \Cref{lemma:technical-for-singular-terms} that
\begin{equation}
\label{eq:proof-infP-maxD-q}
\bar{g}(\lambda,\lambda^\circ) =\inf_{x \in X_0} \{ L(x,\lambda)  + \ell(\u,\lambda^\circ)\}.
\end{equation}

We now define
\begin{equation*}
  h(\u) = \begin{cases}
\underset{\y \in L^\infty(\Omega,X_2)}{\inf} L(x,\lambda) , & \text{if } \u \in C, \\
 \infty, & \text{else}
  \end{cases}
\end{equation*}
and
\begin{equation*}
  k(\u) = - \ell(\u,\lambda^\circ).
\end{equation*}
Notice that $\bar{g}(\lambda,\lambda^\circ) = \inf_{\u \in X_1} \{ h(\u) - k(\u)\}$.
Additionally, $h \not\equiv \infty$ is convex and $k>-\infty$ is concave.
In fact, since $\bar{g}$ is finite, $k$ cannot be identical to $\infty$
and $h$ must be proper. Therefore by Fenchel's duality theorem
(cf.~\cite[Theorem 6.5.6]{Aubin1990}), with
$h^*(v) = \sup_{\u \in X_1} \{ \langle v, \u \rangle_{X_1^*,X_1} - h(\u)\}$ and
$k^*(v) = \inf_{\u \in X_1} \{ \langle v,\u \rangle_{X_1^*,X_1} - k(\u)\}$, we have
\begin{equation}
\label{eq:dual-max-equality}
\bar{g}(\lambda,\lambda^\circ) = \max_{\u^* \in X_1^*} \{ k^*(\u^*) - h^*(\u^*)\}.
\end{equation}
Let $\u^*$ denote the maximizer of~\eqref{eq:dual-max-equality}, meaning
$\bar{g}(\lambda,\lambda^\circ) = k^*(\u^*) - h^*(\u^*).$
Then by definition of $h^*$, we have for all $\u \in X_1$ that
\begin{equation}
\label{eq:bound-h-star}
h(\u) - \langle \u^*,\u\rangle_{X_1^*,X_1} \geq \bar{g}(\lambda,\lambda^\circ) - k^*(\u^*).
\end{equation}
Likewise by definition of $k$ and $k^*$, we get
\begin{equation*}
  \ell(\u,\lambda^\circ) + \langle \u^*,\u\rangle_{X_1^*,X_1} \geq k^*(\u^*).
\end{equation*}
It is straightforward to see that $\ell(\u,\lambda^\circ) \leq 0$ for all
$\u \in \tilde{C}$. Indeed, $\u \in \tilde{C}$ implies that there
exists a $x_2 \in L^\infty(\Omega,X_2)$ satisfying $e(\u,\y(\omega),\omega)=0$ and
$i(\u,\y(\omega),\omega)\leq_K 0$ a.s. In particular $L^\circ(x,\lambda^\circ) \leq 0$ and we get
\begin{equation*}
  \langle \u^*,\u \rangle_{X_1^*,X_1} \geq k^*(\u^*)
\end{equation*} 
for all $\u \in \tilde{C} \supset C$.
From~\eqref{eq:bound-h-star} we thus have for all $\u \in C$ that 
$h(\u) \geq \bar{g}(\lambda,\lambda^\circ)$ holds, and hence 
\begin{equation*}
  L(x,\lambda)  \geq h(\u) \geq \bar{g}(\lambda,\lambda^\circ)
\end{equation*} 
for all $x \in X_0$ and all
$(\lambda,\lambda^\circ) \in \Lambda\times \Lambda^\circ$.
It follows that $g(\lambda) \geq \inf_{x \in X_0} L(x,\lambda)  \geq \bar{g}(\lambda,\lambda^\circ)$
and we have shown~\eqref{eq:relation-dual-extended-dual} finishing the proof.
\end{proof}

\section{Additional Proofs}
\label{sec:appendix-additional-proofs}
\begin{proof}[Proof of \Cref{lemma:concatenation-convex-functions}.]
By $K$-convexity of $f$, there exists $k \in K$ such that for all $\lambda \in [0,1]$ and all $y, \hat{y} \in Y$ 
\begin{equation*}
-f(\lambda y + (1-\lambda) \hat{y}) = k - \lambda f(y) -(1-\lambda) f(\hat{y}).
\end{equation*}
Therefore
\begin{align*}
g(-f(\lambda y + (1-\lambda) \hat{y})) & = g(k - \lambda f(y) -(1-\lambda) f(\hat{y}))\\
& \leq g(k) + g(- \lambda f(y) -(1-\lambda) f(\hat{y}))\\
& \leq \lambda g(-f(y)) + (1-\lambda)g(-f(\hat{y})),
\end{align*}
where we used convexity of $g$ for both inequalities. 
\end{proof}

\begin{lemma}
\label{lem:subdifferential-bounded-to-bounded}
Let $Y$ be a Banach space and suppose $f\colon Y \rightarrow \R$ is lower semicontinuous, convex, and Lipschitz continuous relative to every bounded subset of $Y$. Then $\partial f$ maps every bounded subset of $Y$ to a bounded set in $Y^*.$ 
\end{lemma}

\begin{proof}
We follow the arguments from \cite[Proposition 16.17]{Bauschke2011} to generalize to the Banach space $Y.$ Let $B_r(x)$ denote the closed ball of radius $r$ centered at $x$. We show that the subgradients of $f$ are uniformly bounded on every open ball centered at zero. Let $\rho >0$ and $\lambda_\rho$ be the Lipschitz constant of $f$ relative to $\text{int}B_{\rho}(0).$ Take any $y \in \text{int}B_{\rho}(0)$ and let $\alpha >0$ be such that $B_{\alpha}(y) \subset \text{int}B_{\rho}(0)$. Now, take $y^* \in \partial f(y).$ Then for all $z \in B_1(0)$, we have
\begin{equation*}
    \langle y^*, \alpha z\rangle_{Y^*,Y} \leq f(y+\alpha z) - f(y) \leq \lambda_\rho \lVert\alpha z\rVert_{Y} \leq \lambda_\rho \alpha.
\end{equation*}
Now, applying the supremum over $z \in B_1(0)$, we get $\alpha\lVert y^*\rVert_{Y^*} \leq \lambda_\rho \alpha$, from which we deduce $\lVert y^*\rVert_{Y^*} \leq \lambda_\rho$. Therefore, $\sup_{y \in \text{int}B_\rho(0)} \partial f(y) \leq \lambda_\rho.$
\end{proof}

\newpage

\section{Acknowledgments}
The second author acknowledges support by the German Research Foundation DFG through the collaborative research center SFB-TRR 154 under project B02.

\bibliographystyle{abbrvnat}
\bibliography{references}
\end{document}